\documentclass[leqno,english]{amsart}

\usepackage[draft]{hyperref}   % www references
\usepackage[square,numbers]{natbib}     % bibliography
\usepackage{amsthm}     % ams theorem environments
\usepackage{amsmath}    % ams math environments
\usepackage{amssymb}    % ams symbols
\usepackage{amsxtra}    % ams extra symbols
\usepackage{eucal}      % eu calligraphic fonts
\usepackage[bbgreekl]{mathbbol}   % black board symbols
\usepackage[all]{xy}    % diagrams
\usepackage{todonotes}
\usepackage{comment}
\usepackage{tikz}
\usepackage{mathtools}
\usepackage{color,soul}

\usetikzlibrary{matrix}
\usetikzlibrary{patterns, shapes}
\usetikzlibrary{arrows}
\usetikzlibrary{calc,3d}
\usetikzlibrary{decorations,decorations.pathmorphing, decorations.pathreplacing}
\usetikzlibrary{through}
\tikzset{ext/.style={circle, draw,inner sep=1pt},int/.style={circle,draw,fill,inner sep=1pt},nil/.style={inner sep=1pt}}
\tikzset{exte/.style={circle, draw,inner sep=3pt},inte/.style={circle,draw,fill,inner sep=3pt}}
\tikzset{diagram/.style={matrix of math nodes, row sep=3em, column sep=2.5em, text height=1.5ex, text depth=0.25ex}}
\tikzset{diagram2/.style={matrix of math nodes, row sep=0.5em, column sep=0.5em, text height=1.5ex, text depth=0.25ex}}

\usepackage{tikz-cd}

\theoremstyle{plain}

\newtheorem{introthm}{Theorem}

\newtheorem{mainthm}{Theorem}

\newtheorem{claim}[mainthm]{Claim}

\newtheorem{thm}{Theorem}
\newtheorem{prop}[thm]{Proposition}
\newtheorem{lemm}[thm]{Lemma}
\newtheorem{cor}[thm]{Corollary}

\theoremstyle{definition}

\newtheorem*{rem}{Remark}
\newtheorem{const}[thm]{Construction}
\newtheorem{defn}[thm]{Definition}

\numberwithin{subsubsection}{subsection}

\title[Mapping spaces for dg Hopf cooperads]{Mapping Spaces for DG Hopf Cooperads and Homotopy Automorphisms of the Rationalization of $E_n$-operads}

\date{\today}

\author{Benoit Fresse}
\address{Univ. Lille, CNRS, UMR 8524 - Laboratoire Paul Painlev\' e, F-59000 Lille, France}
\email{Benoit.Fresse@univ-lille.fr}

\author{Thomas Willwacher}
\address{Department of Mathematics\\
ETH Z\"urich\\
R\"amistrasse 101\\
8092 Z\"urich, Switzerland}
\email{thomas.willwacher@math.ethz.ch}

\DeclareMathOperator{\kk}{\mathbb{k}}   % the ground ring
\DeclareMathOperator{\NN}{\mathbb{N}}
\DeclareMathOperator{\ZZ}{\mathbb{Z}}
\DeclareMathOperator{\QQ}{\mathbb{Q}}

\DeclareMathOperator{\s}{\mathit{s}}      % simplicial objects
\DeclareMathOperator{\dg}{\mathit{dg}}      % dg objects
\DeclareMathOperator{\g}{\mathit{g}}      % graded objects
\DeclareMathOperator{\CCat}{\mathcal{C}}

\DeclareMathOperator{\Vect}{\mathcal{V}\mathit{ect}}
\DeclareMathOperator{\dgVect}{\dg\Vect}
\DeclareMathOperator{\dgZVect}{\dg_{\ZZ}\Vect}

\DeclareMathOperator{\Set}{\mathcal{S}\mathit{et}}  % sets
\DeclareMathOperator{\Top}{\mathcal{T}\mathit{op}}  % topological spaces
\DeclareMathOperator{\sSet}{\mathit{s}\mathcal{S}\mathit{et}}   % simplicial sets
\DeclareMathOperator{\ComCat}{\mathcal{C}\mathit{om}}   % commutative algebras

\DeclareMathOperator{\dgCom}{\dg\ComCat}
\DeclareMathOperator{\dgZCom}{\dg_{\ZZ}\ComCat}

\DeclareMathOperator{\Tree}{\mathcal{T}\mathit{ree}}

\DeclareMathOperator{\Seq}{\mathcal{S}\mathit{eq}}  % symmetric sequences
\DeclareMathOperator{\Op}{\mathcal{O}\mathit{p}}    % operads
\DeclareMathOperator{\Hopf}{\mathcal{H}\mathit{opf}}    % operads

\DeclareMathOperator{\sSetOp}{\sSet\Op}
\DeclareMathOperator{\TopOp}{\Top\Op}

\DeclareMathOperator{\dgZOpc}{\dg_{\ZZ}\Op^c}
\DeclareMathOperator{\dgOpc}{\dg\Op^c}

\DeclareMathOperator{\dgOp}{\dg\Op}
\DeclareMathOperator{\dgZHOpc}{\dg_{\ZZ}\Hopf\Op^c}
\DeclareMathOperator{\dgHOpc}{\dg\Hopf\Op^c}
\DeclareMathOperator{\dgHOp}{\dg\Hopf\Op}
\DeclareMathOperator{\gHOpc}{\g\Hopf\Op^c}
\DeclareMathOperator{\dgSeqc}{\dg\Seq^c}
\DeclareMathOperator{\dgZSeqc}{\dg_{\ZZ}\Seq^c}
\DeclareMathOperator{\gSeqc}{\g\Seq^c}

\DeclareMathOperator{\dgHSeqc}{\dg\Hopf\Seq^c}
\DeclareMathOperator{\dgHSeq}{\dg\Hopf\Seq}
\DeclareMathOperator{\dgZHSeqc}{\dg_{\ZZ}\Hopf\Seq^c}

\DeclareMathOperator{\sLaOp}{\s\Lambda\Op}
\DeclareMathOperator{\dgLaOpc}{\Com^c/\dg\Lambda\Op^c}
\DeclareMathOperator{\dgLaHOpc}{\dg\Hopf\Lambda\Op^c}
\DeclareMathOperator{\dgZLaOpc}{\Com^c/\dg_{\ZZ}\Lambda\Op^c}
\DeclareMathOperator{\dgZLaHOpc}{\dg_{\ZZ}\Hopf\Lambda\Op^c}

\DeclareMathOperator{\Map}{\mathtt{Map}}
\DeclareMathOperator{\Mor}{\mathtt{Mor}}

\DeclareMathOperator{\Hom}{\mathtt{Hom}}

\DeclareMathOperator{\Aut}{\mathtt{Aut}}
\DeclareMathOperator{\Def}{\mathtt{Def}}
\DeclareMathOperator{\BiDer}{\mathtt{BiDer}}

\DeclareMathOperator{\id}{\mathit{id}}

\DeclareMathOperator{\eq}{\mathrm{eq}}

\DeclareMathOperator{\reflexiverightrightarrows}{\stackrel{\curvearrowleft}{\rightrightarrows}} % reflexive parallel mappings
\DeclareMathOperator{\trivialrightarrowtail}{\overset{\sim}{\rightarrowtail}} % acyclic cofibrations
\DeclareMathOperator{\trivialtwoheadrightarrow}{\overset{\sim}{\twoheadrightarrow}} % acyclic fibrations
\DeclareMathOperator{\im}{\mathrm{im}}
\DeclareMathOperator*{\colim}{colim}

\DeclareMathOperator{\FreeOp}{\mathbb{F}}     % free operads
\DeclareMathOperator{\Sym}{\mathbb{S}}      % symmetric (co)algebras
\DeclareMathOperator{\MC}{MC}

\DeclareMathAlphabet{\mathsfit}{OT1}{cmss}{m}{sl}
\DeclareMathOperator{\AOp}{\mathsfit{A}}
\DeclareMathOperator{\BOp}{\mathsfit{B}}
\DeclareMathOperator{\COp}{\mathsfit{C}}
\DeclareMathOperator{\DOp}{\mathsfit{D}}
\DeclareMathOperator{\EOp}{\mathsfit{E}}

\DeclareMathOperator{\IOp}{\mathsfit{I}}
\DeclareMathOperator{\MOp}{\mathsfit{M}}
\DeclareMathOperator{\NOp}{\mathsfit{N}}
\DeclareMathOperator{\POp}{\mathsfit{P}}

\DeclareMathOperator{\ROp}{\mathsfit{R}}
\DeclareMathOperator{\SOp}{\mathsfit{S}}
\DeclareMathOperator{\TOp}{\mathsfit{T}}
\DeclareMathOperator{\ZOp}{\mathsfit{Z}}
\DeclareMathOperator{\eop}{\mathsfit{e}}

\DeclareMathOperator{\Com}{\mathsfit{Com}}

\DeclareMathOperator{\HGC}{\mathrm{HGC}}
\DeclareMathOperator{\Graphs}{\mathsf{Graphs}}
\DeclareMathOperator{\GG}{\mathsf{G}}
\DeclareMathOperator{\GC}{\mathsf{GC}}
\DeclareMathOperator{\ICG}{\mathsf{ICG}}

\DeclareMathOperator{\IG}{\mathsf{IG}}

\DeclareMathOperator{\galg}{\mathfrak{g}}
\DeclareMathOperator{\halg}{\mathfrak{h}}
\DeclareMathOperator{\ad}{ad}

\newcommand{\htimes}{\times}

\begin{document}

\begin{abstract}
We define a simplicial enrichment on the category of differential graded Hopf cooperads (the category of dg Hopf cooperads for short).
We prove that our simplicial enrichment satisfies, in part, the axioms of a simplicial model category structure
on the category of dg Hopf cooperads.
We use this simplicial model structure to define a model of mapping spaces in the category of dg Hopf cooperads
and to upgrade results of the literature about the homotopy automorphism spaces of dg Hopf cooperads
by dealing with simplicial monoid structures.
The rational homotopy theory of operads implies that the homotopy automorphism spaces of dg Hopf cooperads
can be regarded as models for the homotopy automorphism spaces
of the rationalization of operads in topological spaces (or in simplicial sets).
We prove, as a main application, that the spaces of Maurer--Cartan forms on the Kontsevich graph complex Lie algebras
are homotopy equivalent, in the category of simplicial monoids,
to the homotopy automorphism spaces of the rationalization of the operads of little discs.
\end{abstract}

\maketitle

\section*{Introduction}

The rational homotopy theory of operads has been developed in~\cite{OperadHomotopyBook,ExtendedRHT}
by using a model given by the category of cooperads in differential graded commutative algebras (the category of dg Hopf cooperads for short).
The definition of this model works as follows.
The category of dg Hopf cooperads, denoted by $\dgHOpc$, is equipped with a model structure.
To relate this model category to the category of operads in simplicial sets $\sSetOp$
we use a Quillen adjunction
\begin{equation*}
G: \dgHOpc\rightleftarrows\sSetOp^{op} :\Omega_{\sharp}^*,
\end{equation*}
which extends the standard Quillen adjunction of the rational homotopy theory between the category of dg commutative algebras and the category of simplicial sets.
For $\POp\in\sSetOp$ a cofibrant simplicial operad, we can now define the rationalization as
\begin{equation*}
\POp^{\QQ} := LG\Omega_{\sharp}^*(\POp),
\end{equation*}
where $LG$ denotes the left derived functor of $G$.
The correspondence between our operadic adjunction and the standard Quillen adjunction of the rational homotopy theory ensures that,
under usual nilpotence and cohomological finiteness assumptions,
the components of this operad $\POp^{\QQ}$ are weakly equivalent to the rationalization of the spaces $\POp(r)$ underlying our object $\POp$.
Furthermore, we can regard the object $\Omega_{\sharp}^*(\POp)$ as a dg Hopf cooperad model for the rational homotopy type of the operad~$\POp$.

In what follows, we use that we can prolong the definition of this model of the rational homotopy theory to the category of operads in topological spaces
by using that the classical Milnor equivalence between simplicial sets and topological spaces preserves operads.
Let $|-|$ denote the classical geometric realization functor on simplicial sets.
For an operad in topological spaces $\POp$, we explicitly set $\Omega_{\sharp}^*(\POp) := \Omega_{\sharp}^*(\ROp)$,
where $\ROp$ is a cofibrant operad in simplicial sets such that $|\ROp|\sim\POp$.
We use this observation to apply our constructions to the topological operads of little discs $\DOp_n$.

The first goal of this paper is to prove that the model structure on $\dgHOpc$ can be enhanced to a simplicial model structure.
For this purpose, we adapt a definition of mapping spaces due to Bousfield--Gugenheim~\cite[Section 5]{BousfieldGugenheim} for dg commutative algebras
and to Hinich~\cite{Hinich} for dg coalgebras.
In fact, we only prove that our category $\dgHOpc$ is lax cotensored over the category of simplicial sets and, for a simplicial set $K$, our function object functor $\BOp\mapsto\BOp^K$
behaves properly with respect to finite limits only.
(We refer to Proposition \ref{prop:adjunctionrelation}-\ref{prop:pullbackcorner} and Theorem \ref{thm:map comparison} for the precise statements.)
Thus, we do not have a full simplicial model category structure on our category of dg Hopf cooperads.
(To feature our result, we say that we have a right finitely continuous simplicial model structure on $\dgHOpc$.)
Nevertheless, we have enough to ensure that the mapping spaces of our simplicial structure
are weakly equivalent to the hom-objects of the Dwyer-Kan simplicial localization of our category,
and therefore, have the right homotopy type.
In particular, our simplicial structure allows us to define a good model for the homotopy automorphism spaces of objects
in the category of dg Hopf cooperads $\Aut^h_{\dgHOpc}(\AOp)$.
To be specific, the definition of this model from a simplicial structure implies that our homotopy automorphism space inherits a natural monoid structure
and this monoid structure agrees with the one naturally associated to homotopy automorphism spaces
in the Dwyer-Kan simplicial localization of categories.

In the case of the dg Hopf cooperad $\AOp = \Omega_{\sharp}^*(\POp)$ associated to a (good) cofibrant operad in simplicial sets $\POp$,
we have a weak equivalence
\begin{equation*}
\Aut^h_{\dgHOpc}(\Omega_{\sharp}^*(\POp))\sim\Aut^h_{\sSetOp}(\POp^{\QQ})
\end{equation*}
and our model therefore enables us to determine the homotopy automorphism space of the rationalized operad $\POp^{\QQ}$ as a simplicial monoid.
We use this result to upgrade computations of the literature about the spaces of rational homotopy automorphisms of the little discs operads~\cite{FTW,FW}.

To express our result, we consider the Kontsevich graph complex $\GC_n^2$,
which is a complete dg Lie algebra whose elements are formal series of connected undirected graphs (see Section~\ref{sec:graph_complexes}).
We form a simplicial group such that
\begin{equation*}
Z_{\bullet}(\GC_n^2) := Z^0(\GC_n^2\hat\otimes\Omega^*(\Delta^{\bullet})),
\end{equation*}
where $Z^0(-)$ denotes the set of closed elements of degree zero in the cochain complex $\GC_n^2\hat\otimes\Omega^*(\Delta^{\bullet})$
with the group structure determined by the Baker--Campbell--Hausdorff formula.
We then use that the Lie algebra $\GC_n^2$ is graded by loop order to form a semidirect product simplicial group $\QQ^\times \ltimes Z_{\bullet}(\GC_n^2)$,
with $\lambda\in\QQ^\times$ acting on a graph of loop order $N$ by multiplication by $\lambda^N$,
and we establish the following main statement.

\begin{introthm}\label{thm:main_intro}
For $n\geq 2$, there are weak equivalences of simplicial monoids
\begin{equation*}
\Aut^h_{\TopOp}(\DOp_n^{\QQ})\sim\Aut^h_{\dgHOpc}(\Omega_{\sharp}^*(\DOp_n))\sim\QQ^\times\ltimes Z_{\bullet}(\GC_n^2).
\end{equation*}
\end{introthm}

Recall that for a topological operad such as $\POp = \DOp_n$, we set $\Omega_{\sharp}^*(\DOp_n) := \Omega_{\sharp}^*(\ROp_n)$,
where $\ROp_n$ is a cofibrant operad in simplicial sets such that $|\ROp_n|\sim\DOp_n$.

In~\cite{OperadHomotopyBook}, the first author already obtained a computation of the space of rational homotopy automorphisms of the little $2$-discs operad as:
\begin{equation*}
\Aut^h_{\sSetOp}(\DOp_2^{\QQ})\sim S^1_{\QQ}\ltimes\mathrm{GRT},
\end{equation*}
where $\mathrm{GRT}$ denotes the Grothendieck-Teichm\"uller group.
Together with Theorem~\ref{thm:main_intro} above,
this statement yields another proof of the second author's result~\cite{Willwacher}
that $H^0(\GC_2)$ is identified with the Grothendieck--Teichm\"uller Lie algebra $\mathfrak{grt}_1$,
where $\GC_2$ denotes a reduced version (with no bivalent vertices allowed) of the graph complex $\GC_2^2$.

We give a brief summary of our conventions in the next preliminary section.
We explain the definition of our (right finitely continuous) simplicial model structure on dg Hopf cooperads in Section~\ref{sec:simplicial model structure}.
We proof that our simplicial structure satisfies the axioms of simplicial model categories in Section~\ref{sec:pullbackcorner}
after observing that our constructions restrict to the category of dg cooperads
where we forget about the dg commutative algebra structures
attached to our objects.
In fact, we carry out our verifications of the axioms in the model category of dg cooperads.

We tackle our study of homotopy automorphism spaces in Section~\ref{sec:hautspaces}.
We review general invariance properties of the monoids of homotopy automorphisms
and we check that we can compute the homotopy automorphism space of the rationalization of an operad
in terms of our dg Hopf cooperad model.
We explain afterwards a general approach for the computation of homotopy automorphism spaces of dg Hopf cooperads
in terms of deformation complexes and Maurer--Cartan spaces of $L_{\infty}$-algebras.
We devote Section~\ref{sec:deformationcomplexes} to this subject.
We apply this method to the Kontsevich graph complex $\GC_n^2$, which we associate to a graph complex model of the cohomology of the little discs operad $\DOp_n$,
regarded as a dg Hopf cooperad equipped with a trivial differential.
We recall the definition of these objects in Section~\ref{sec:graph_complexes} and we prove afterwards, in Section~\ref{sec:aut thm proof},
that the Kontsevich graph complex determines the monoid of homotopy automorphisms of the little discs operad,
as asserted in Theorem~\ref{thm:main_intro}.

We complete this paper by the study of a counterpart of our simplicial model structure for the category of dg Hopf $\Lambda$-cooperads,
which is used to model the rational homotopy of operads such that $\POp(0) = *$.
We treat this extension of our construction in Section~\ref{sec:Lambdaoperadmodel}.
We also devote an appendix section to the definition of an auxiliary model structure on the category of $\ZZ$-graded dg Hopf cooperads.
(We use this model structure in the constructions of Section~\ref{sec:aut thm proof}.)

\setcounter{section}{-1}
\section{Notation and conventions}\label{sec:notation}

In what follows we generally use the conventions and notation of the book ``Homotopy of operads and Grothendieck--Teichm\"uller groups'' \cite{OperadHomotopyBook}
for the rational homotopy of operads (with slight adjustments)
and of the paper \cite{ExtendedRHT} for the extension of this rational homotopy theory to operads with unary operations.

We denote by $\sSet$ the category of simplicial sets.
We equip this category $\sSet$ with the standard Kan model structure,
with the weak equivalences that correspond to the weak equivalences of topological spaces under geometric realization,
the fibrations given by the Kan fibrations,
and the cofibrations given by the class of dimensionwise injective maps of simplicial sets (see~\cite[Theorem II.1.3.12]{OperadHomotopyBook}).

We fix a ground field of characteristic zero $\kk$, which we specialize to $\QQ$ where appropriate.
We deal with two categories of cochain complexes, namely the category of general $\ZZ$-graded cochain complexes, which we denote by $\dgZVect$,
and the category of cochain complexes concentrated in non-negative degrees,
which we denote by $\dgVect$.
We use cohomological conventions all along this paper (in contrast with~\cite{OperadHomotopyBook,ExtendedRHT}
where cohomological conventions are only used for the cochain models of the rational homotopy theory).
If necessary, then we can use the standard equivalence $A_* = A^{-*}$ to convert a cohomological grading into a homological grading.

We actually use several categories of objects equipped with a differential graded structure (a dg structure for short).
We may assume that our dg objects are defined either within the category of non-negatively graded cochain complexes
or within the category of $\ZZ$-graded cochain complexes.
We do not specify the range of the grading of our dg objects in our terminologies in general,
but we mainly deal with dg objects defined in the category of non-negatively graded cochain complexes
when we form our models for the rational homotopy of operads
and we only use $\ZZ$-graded dg objects in auxiliary constructions.
We therefore generally assume by default that our categories of dg objects are formed within non-negatively graded cochain complexes
and we use a prefix $\dg$ (with no decoration) in the notation of such categories.
We just add a $\ZZ$ subscript to this notation when we consider dg objects in $\ZZ$-graded cochain complexes.
For example, we denote by $\dgCom$ the category of dg commutative algebras in non-negatively graded cochain complexes
while $\dgZCom$ denotes the category of dg commutative algebras in $\ZZ$-graded cochain complexes.
We adopt similar conventions for categories of graded objects (equivalent to cochain complexes with a zero differential),
which we single out by a prefix $\g$ or $\g_{\ZZ}$ depending on the range of the grading.

Let $\AOp$ be any dg object (for example a dg Hopf cooperad).
We denote by $\AOp^{\flat}$ the underlying graded object, obtained by setting the differential to zero, but retaining the rest of the structure.
In the case of dg Hopf cooperads for instance, this mapping gives a functor $(-)^{\flat} : \dgHOpc \rightarrow \gHOpc$,
where $\gHOpc$ denotes the category of Hopf cooperads in graded vector spaces.

The category of non-negatively graded cochain complexes $\dgVect$ is equipped with the model structure
where the weak equivalences are the quasi-isomorphisms,
the fibrations are the degreewise surjective maps,
and the cofibrations are the maps that are injective in positive degrees.
The category of $\ZZ$-graded cochain complexes $\dgZVect$ is equipped with the model structure
where the weak equivalences are the quasi-isomorphisms,
the fibrations are the degreewise surjective maps,
and the cofibrations are the maps that are injective in all degrees.
We also consider the usual model structure on the category of dg commutative algebras $\dgCom$,
which is defined by adjunction from the model structure on $\dgVect$
by assuming that a morphism of dg commutative algebras defines a weak equivalence (respectively, a fibration)
if this morphism defines a weak equivalence (respectively, a fibration) in $\dgVect$.
We consider the similarly defined model structure on $\dgZCom$ when we deal with $\ZZ$-graded dg commutative algebras.

Recall that we use the notation $\dgHOpc$ for the category of dg Hopf cooperads, which we define as the category of cooperads in dg commutative algebras.
We also consider the category of dg cooperads, defined as the category of cooperads in $\dgVect$ and which we denote by $\dgOpc$.
In our constructions, we also deal with $\ZZ$-graded variants of these categories,
which we respectively denote by $\dgZHOpc$ and $\dgZOpc$ in accordance with our conventions.
We generally assume that our cooperads have no term in arity zero, are coaugmented and conilpotent. (But our cooperads possibly have non-trivial cooperations in arity one.)
We use the model category structure defined in \cite{ExtendedRHT} for the categories $\dgOpc$ and $\dgHOpc$.
Recall simply that the classes of weak equivalences and cofibrations in $\dgOpc$ are created in $\dgVect$
as the classes of morphisms of dg cooperads that form weak equivalences and cofibrations of cochain complexes arity-wise,
while the class of fibrations of dg coperads is defined as the class of morphisms that have the right lifting property with respect to the acyclic cofibrations.
The model structure $\dgHOpc$ is created by adjunction from the model structure in $\dgOpc$
by assuming that a morphism of dg Hopf operads is a weak equivalence (respectively, a fibration)
if this morphism defines a weak equivalence (respectively, a fibration) of dg cooperads.
We have analogous model structures on the categories $\dgZOpc$ and $\dgZHOpc$ (see~Appendix~\ref{sec:modelcat}).

We call symmetric sequence the structure, underlying a cooperad, formed by a collection of objects $\MOp(r)$, $r = 1,2,\ldots$,
whose terms $\MOp(r)$ are equipped with an action of the symmetric groups $\Sigma_r$.
We stress that, by convention, our symmetric sequences have no term in arity zero.
We again use the expression `dg symmetric sequence' for the category of symmetric sequences in $\dgVect$ (or in $\dgZVect$)
and the expression `dg Hopf symmetric sequence' for the category of symmetric sequences in $\dgCom$ (or in $\dgZCom$).
We use the notation $\dgSeqc$ for the category of dg symmetric sequences and the notation $\dgHSeqc$ for the category of dg Hopf symmetric sequences
(and yet we use the notation $\dgZSeqc$ and $\dgZHSeqc$ for the $\ZZ$-graded variants of these categories).
We equip our categories of dg symmetric sequences and of dg Hopf symmetric sequences with the usual projective model structure of diagram categories,
with the class of weak equivalences and the class of fibrations created termwise in the underlying category of cochain complexes and of dg commutative algebras.
In our context, a morphism defines a cofibration in the category of dg symmetric sequences
as soon as this morphism forms a cofibration of cochain complexes termwise,
because we assume that our ground ring is a field of characteristic zero.

Following the conventions of~\cite{OperadHomotopyBook}, we denote the morphism sets of a category $\CCat$ by $\Mor_{\CCat}(-)$,
while we reserve $\Hom_{\CCat}(-)$ for the graded or dg hom-objects of an enriched category structure.
Mapping spaces will be denoted by $\Map_{\CCat}(-)$.
The subscript $\CCat$ is omitted from the notation whenever the context makes the choice of the category clear.
We use the symbol `$\sim$' to denote the class of weak equivalences in a model category, while the symbol `$\simeq$' denotes an isomorphism.

\section{The mapping spaces and function objects of dg Hopf cooperads}\label{sec:simplicial model structure}
In this section, we give the definition of our model of mapping spaces in the category of dg Hopf cooperads $\dgHOpc$.
We also prove that these mapping spaces are represented by function objects.
We mainly use the latter observation in the next section in order to check the validity of our model.
Throughout this paper, we use the notation $\Omega^*(\Delta^{\bullet})$ for the Sullivan simplicial dg algebra,
which consists of the piece-wise linear forms
on the simplices (see for instance~\cite[\S II.7.1]{OperadHomotopyBook}).

\begin{const}\label{subsubsection:construction}
Let $\AOp$ and $\BOp$ be dg Hopf cooperads.
For any $n\in\NN$, we set:
\begin{equation}
\label{equ:Mapping space def}
\Map(\AOp,\BOp)_n = \Mor_{\dgHOp^c_{\Omega^*(\Delta^n)}}(\AOp\otimes\Omega^*(\Delta^n),\BOp\otimes\Omega^*(\Delta^n)),
\end{equation}
where $\dgHOp^c_{\Omega^*(\Delta^n)}$ denotes the category of dg Hopf cooperads defined over the ground dg algebra $R = \Omega^*(\Delta^n)$,
and we consider the objects
\begin{gather*}
\AOp\otimes\Omega^*(\Delta^n),\BOp\otimes\Omega^*(\Delta^n)\in\dgHOp^c_{\Omega^*(\Delta^n)}
\intertext{such that}
(\AOp\otimes\Omega^*(\Delta^n))(r) = \AOp(r)\otimes\Omega^*(\Delta^n),\quad(\BOp\otimes\Omega^*(\Delta^n))(r) = \BOp(r)\otimes\Omega^*(\Delta^n),
\end{gather*}
for each arity $r>0$.
We take the set of morphisms between these objects in $\dgHOp^c_{\Omega^*(\Delta^n)}$.
To any map $u: \underline{m}\rightarrow\underline{n}$ in the simplicial category $\Delta$
we can associate the simplicial operator $u^*: \Map(\AOp,\BOp)_n\rightarrow\Map(\AOp,\BOp)_m$
such that:
\begin{equation*}
u^*\phi(\underbrace{a\otimes\omega}_{\in\AOp(r)\otimes\Omega^*(\Delta^m)})
= (\id\otimes u^*)(\phi(\underbrace{a\otimes 1}_{\in\AOp(r)\otimes\Omega^*(\Delta^n)}))\cdot\omega,
\end{equation*}
for any morphism $\phi: \AOp\otimes\Omega^*(\Delta^n)\rightarrow\BOp\otimes\Omega^*(\Delta^n)$,
so that the collection $\Map(\AOp,\BOp)_n$, $n\in\NN$, inherits the structure of a simplicial set.

Note that these mapping spaces satisfy the relation $\Map(\AOp,\BOp)_0 = \Mor(\AOp,\BOp)$
and inherit obvious composition operations $\circ: \Map(\COp,\BOp)\times\Map(\AOp,\COp)\rightarrow\Map(\AOp,\BOp)$,
which extend the composition of morphisms in the category of dg Hopf cooperads.
Hence, our construction provides the category of dg Hopf cooperads
with a simplicial enrichment.
\end{const}

\begin{rem}\label{rem:finite limits}
Note that our construction of the simplicial sets $\Map(\AOp,\BOp)$ is not compatible with arbitrary limits,
since the tensor product with the cochain complex $\Omega^*(\Delta^{\bullet})$ does not commute with arbitrary limits.
Nevertheless we do have an identity
\begin{equation*}
\Map(\AOp,\lim_{j\in J}\BOp_j) = \lim_{j\in J}\Map(\AOp,\BOp_j)
\end{equation*}
when the functor $j\mapsto\BOp_j$ is defined on a finite category $J$.
\end{rem}

\begin{prop}\label{prop:adjunctionrelation}
For $K\in\sSet$, we have an adjunction formula:
\begin{equation}
\label{equ:prop adjunctionrelation}
\Mor_{\sSet}(K,\Map(\AOp,\BOp))\simeq\Mor_{\dgHOpc}(\AOp,\BOp^K),
\end{equation}
where $\BOp^K$ is a dg Hopf cooperad naturally associated to the pair $(\BOp,K)$.
\end{prop}

\begin{proof}
This proposition follows from usual adjoint functor theorems.
We use that the category of dg Hopf cooperads is comonadic over the category of dg Hopf symmetric sequences.
To obtain our result, we only need to define a simplicial object in the category of dg Hopf cooperads $\BOp^{\Delta^{\bullet}}$ such that we have a natural bijection:
\begin{equation}\tag{*}\label{equ:adjunctionrelation:simplicial}
\Map(\AOp,\BOp)\simeq\Mor_{\dgHOpc}(\AOp,\BOp^{\Delta^{\bullet}})
\end{equation}
in the category of simplicial sets.
Indeed, as soon as we have a simplicial dg Hopf cooperad $\BOp^{\Delta^{\bullet}}$, we can determine the image of a simplicial set $K$ under our function object functor $K\mapsto\BOp^K$ by a natural reflexive equalizer formula (see~\cite[\S II.2.0]{OperadHomotopyBook}).

In a first step, we consider the case of a cofree dg Hopf cooperad $\BOp = \FreeOp^c(\MOp)$ generated by a dg Hopf symmetric sequence $\MOp\in\dgHSeqc$.
We then set:
\begin{gather*}
\FreeOp^c(\MOp)^{\Delta^{\bullet}} = \FreeOp^c(\MOp\otimes\Omega^*(\Delta^{\bullet})),
\intertext{where we consider the dg Hopf symmetric sequence such that}
(\MOp\otimes\Omega^*(\Delta^{\bullet}))(r) = \MOp(r)\otimes\Omega^*(\Delta^{\bullet}),
\end{gather*}
for each arity $r>0$.
We have an identity $\FreeOp_{\kk}^c(\MOp)\otimes\Omega^*(\Delta^{\bullet}) = \FreeOp_{\Omega^*(\Delta^{\bullet})}^c(\MOp\otimes\Omega^*(\Delta^{\bullet}))$ in the category $\dgHOp^c_{\Omega^*(\Delta^n)}$, where we use the notation $\FreeOp_{\kk}^c(-) = \FreeOp^c(-)$ for the cofree cooperad functor over the ground field $\kk$ and the notation $\FreeOp_{\Omega^*(\Delta^{\bullet})}^c(-)$ for the cofree cooperad functor over the dg algebra $\Omega^*(\Delta^{\bullet})$.
We use this relation, the cofree cooperad adjunction, and an obvious scalar extension adjunction relation in the category of Hopf symmetric seqences to get the relations:
\begin{align*}
\Map(\AOp,\FreeOp^c(\MOp) &= \Mor_{\dgHOp^c_{\Omega^*(\Delta^{\bullet})}}(\AOp\otimes\Omega^*(\Delta^{\bullet}),\FreeOp^c(\MOp)\otimes\Omega^*(\Delta^{\bullet}))\\
&\begin{aligned}
& \simeq\Mor_{\dgHOp^c_{\Omega^*(\Delta^{\bullet})}}(\AOp\otimes\Omega^*(\Delta^{\bullet}),
\FreeOp_{\Omega^*(\Delta^{\bullet})}^c(\MOp\otimes\Omega^*(\Delta^{\bullet}))) \\
& \simeq\Mor_{\dgHSeq^c_{\Omega^*(\Delta^{\bullet})}}(\AOp\otimes\Omega^*(\Delta^{\bullet}),\MOp\otimes\Omega^*(\Delta^{\bullet})) \\
& \simeq\Mor_{\dgHSeqc}(\AOp,\MOp\otimes\Omega^*(\Delta^{\bullet})).
\end{aligned}
\end{align*}
We then have:
\begin{equation*}
\Mor_{\dgHSeqc}(\AOp,\MOp\otimes\Omega^*(\Delta^{\bullet}))\simeq\Mor_{\dgHOpc}(\AOp,\FreeOp^c(\MOp\otimes\Omega^*(\Delta^{\bullet}))).
\end{equation*}
by the cofree cooperad adjunction again.
We eventually obtain that the object $\FreeOp^c(\MOp)^{\Delta^{\bullet}} = \FreeOp^c(\MOp\otimes\Omega^*(\Delta^{\bullet}))$
satisfies the requested adjunction relation:
\begin{equation*}
\Map(\AOp,\FreeOp^c(\MOp))\simeq\Mor_{\dgHOpc}(\AOp,\FreeOp^c(\MOp)^{\Delta^{\bullet}})
\end{equation*}
for a cofree cooperad $\BOp = \FreeOp^c(\MOp)$.
We use this adjunction relation and the Yoneda lemma to establish that the construction $\FreeOp^c(\MOp)^{\Delta^{\bullet}} = \FreeOp^c(\MOp\otimes\Omega^*(\Delta^{\bullet}))$ defines a functor on the full subcategory of the category of dg Hopf cooperads generated by the cofree objects $\FreeOp^c(\MOp)$.

In the second step, we use that any object $\BOp\in\dgHOpc$ is given by a reflexive equalizer of cofree objects $\BOp = \eq(\FreeOp^c(\MOp^0)\reflexiverightrightarrows\FreeOp^c(\MOp^1))$
and that this construction is functorial in $\BOp$.
We then set:
\begin{equation*}
\BOp^{\Delta^{\bullet}} = \eq(\FreeOp^c(\MOp^0)^{\Delta^{\bullet}}\reflexiverightrightarrows\FreeOp^c(\MOp^1)^{\Delta^{\bullet}}).
\end{equation*}
We just use that both sides of our adjunction relation~\eqref{equ:adjunctionrelation:simplicial} preserves the equalizers in the variable $\BOp$
to conclude that the above object $\BOp^{\Delta^{\bullet}}$ satisfies this adjunction relation,
for all $\BOp\in\dgHOpc$.
\end{proof}

\begin{rem}
The adjunction relation of this proposition implies that our function object bifunctor $(\BOp,K)\mapsto\BOp^K$
preserves the finite limits in the variable $\BOp$, like our mapping space $\Map(-)$,
but this functor does not preserve arbitrary limits either.
\end{rem}

\begin{rem}
For a pair of simplicial sets, we have a natural transformation $(\BOp^K)^L\rightarrow\BOp^{K\times L}$,
which is a weak-equivalence (at least when $\BOp$ is fibrant as a Hopf dg cooperad),
but not a genuine isomorphism.
Therefore, we deduce from the construction of Proposition~\ref{prop:adjunctionrelation}
that the category of dg Hopf cooperads is cotensored over simplicial sets
in the lax sense but not in the strong sense.

This natural transformation can be defined
from a morphism of bisimplicial objects
$(\BOp^{\Delta^{\bullet}})^{\Delta^{\bullet}}\rightarrow\BOp^{\Delta^{\bullet}\times\Delta^{\bullet}}$
by using an obvious generalization of the reflexive coequalizer argument
alluded to in the proof of Proposition~\ref{prop:adjunctionrelation}.
For a cofree cooperad $\BOp = \FreeOp^c(M)$, we have $(\BOp^{\Delta^{\bullet}})^{\Delta^{\bullet}} = \FreeOp^c(M\otimes\Omega^*(\Delta^{\bullet})\otimes\Omega^*(\Delta^{\bullet}))$, $\BOp^{\Delta^{\bullet}\times\Delta^{\bullet}} = \FreeOp^c(M\otimes\Omega^*(\Delta^{\bullet}\times\Delta^{\bullet}))$,
and this morphism of bisimplicial function objects
is induced by the codiagonal map
$\Omega^*(\Delta^k)\otimes\Omega^*(\Delta^l)\rightarrow\Omega^*(\Delta^k\times\Delta^l)$
on the Sullivan dg algebra functor (see for instance~\cite[\S II.7.1]{OperadHomotopyBook}
for the definition of this map).

Then we can rely on the pullback-corner property, which we establish next, and general model category arguments
to establish that the natural transformation $(B^K)^L\rightarrow B^{K\times L}$,
which we associate to our morphism of bisimplicial objects
$(B^{\Delta^{\bullet}})^{\Delta^{\bullet}}\rightarrow B^{\Delta^{\bullet}\times\Delta^{\bullet}}$
defines a weak-equivalence for any pair of simplicial sets $(K,L)$
when $B$ is fibrant.

In the case $K = L = \Delta^n$, we may also see that the composition operation $\circ: \Map(\COp,\BOp)\times\Map(\AOp,\COp)\rightarrow\Map(\AOp,\BOp)$,
which we associate to our mapping space bifunctor in Construction~\ref{subsubsection:construction},
can be given by the formula $f\circ g = \Delta^*(f^{\Delta^n}\circ g)$,
for any pair of morphisms $f\in\Mor(\COp,\BOp^{\Delta^n})$, $g\in\Mor(\AOp,\COp^{\Delta^n})$,
where we use the functoriality of the construction $(-)^{\Delta^n}$
and we compose the morphism $f^{\Delta^n}\circ g\in\Mor(\AOp,(\BOp^{\Delta^n})^{\Delta^n})$
with our natural transformation $(\BOp^{\Delta^n})^{\Delta^n})\rightarrow\BOp^{\Delta^n\times\Delta^n}$
and the morphism $\Delta^*: \BOp^{\Delta^n\times\Delta^n}\rightarrow B^{\Delta^n}$
induced by the diagonal map $\Delta: \Delta^n\rightarrow\Delta^n\times\Delta^n$.
\end{rem}

\begin{prop}\label{prop:pullbackcorner}
The bifunctor $(\BOp,K)\mapsto\BOp^K$ of the previous proposition satisfies the pullback-corner axiom.
To be explicit, the pullback-corner morphism $(p_*,i^*): \AOp^L\rightarrow\BOp^L\times_{\BOp^K}\AOp^K$
associated to a cofibration of simplicial sets $i: K\rightarrowtail L$
and to a fibration of dg Hopf cooperads $p: \AOp\twoheadrightarrow\BOp$
is a fibration of dg Hopf cooperads, which is also acyclic if $i$ or $p$ is so.
\end{prop}

We put off the proof of this proposition until the next section.
We use this proposition in the proof of the following result:

\begin{thm}\label{thm:map comparison}
The definition of Construction~\ref{subsubsection:construction} gives a model of the mapping space $\Map(\AOp,\BOp)$ associated to the objects $(\AOp,\BOp)$
in the model category of dg Hopf cooperads.
\end{thm}

\begin{proof}
For $K = \Delta^{\bullet}$, the result of Proposition~\ref{prop:adjunctionrelation}
implies that we have an identity of the form $\Map(\AOp,\BOp) = \Mor_{\dgHOpc}(\AOp,\BOp^{\Delta^{\bullet}})$,
and we deduce from the result of Proposition~\ref{prop:pullbackcorner}
that the object $\BOp^{\Delta^{\bullet}}$
defines a simplicial framing on $\BOp$. The conclusion follows.
\end{proof}

\begin{rem}
The construction of the simplicial enrichment of the category of dg Hopf cooperads $\dgHOpc$
and the simplicial cotensor structure given in Proposition \ref{prop:adjunctionrelation}
work as well in the category of $\ZZ$-graded dg Hopf cooperads $\dgZHOpc$.
However, we do not know whether Proposition \ref{prop:pullbackcorner} holds in this setting.
Below we will incidentally use mapping spaces in $\dgZHOpc$,
but we have to keep in mind that a priori they are not guaranteed to be weakly equivalent to the proper mapping spaces
of our model category. (We define the model structure on $\dgZHOpc$ in Appendix~\ref{sec:modelcat}.)
\end{rem}

\section{The proof of the pullback-corner property in dg cooperads}\label{sec:pullbackcorner}
\label{sec:pcp proof}
We prove the claim of Proposition~\ref{prop:pullbackcorner} in this section. In fact, we gain our result in the category of dg cooperads.
We therefore forget about the Hopf structure all along this section and we entirely work within the category of dg cooperads.

To perform this reduction, we use that the definition of the bifunctor $(\BOp,K)\mapsto\BOp^K$ in the proof of Proposition~\ref{prop:adjunctionrelation}
extends to the category of dg cooperads
\begin{equation*}
(-)^K: \dgOpc\rightarrow\dgOpc
\end{equation*}
and that the forgetful functor from dg Hopf cooperads to dg cooperads preserves our function objects
because the cofree objects and the equalizers of dg Hopf cooperads
are created in the category of dg cooperads.
The forgetful functor creates the fibrations and the acyclic fibrations in the category of dg Hopf cooperads by definition of this model category.
The verification of the pullback-corner property in the category of dg cooperads therefore implies the validity of the pullback-corner property in the category of dg Hopf cooperads.

In our verifications, we also use that the bifunctor $(\BOp,K)\mapsto\BOp^K$
still satisfies an adjunction relation of the form $\Mor(K,\Map(\AOp,\BOp))\simeq\Mor(\AOp,\BOp^K)$ in the category of cooperads,
for a mapping space $\Map(\AOp,\BOp)$ defined by forgetting about Hopf structures
in Construction~\ref{subsubsection:construction}.

We use the bar duality between dg operads and dg cooperads in our proofs.
We give brief recollections on this subject in a preliminary subsection.
We mainly use that every fibration of dg cooperads occurs as the retract of a fibration that we obtain by pulling back a fibration of the form $B B^c(p): B B^c(\COp)\rightarrow B B^c(\DOp)$,
where we consider the bar construction of dg operads $B(-)$ and the cobar construction of dg cooperads $B^c(-)$.
We devote the second subsection of this section to this result.
We check that the pullback-corner property is valid for (acyclic) fibrations of dg cooperads of this form $B B^c(p): B B^c(\COp)\rightarrow B B^c(\DOp)$
and that the class of morphisms of dg cooperads that fulfill the pullback-corner property is stable under pullbacks and retracts to obtain our result.
We devote the third subsection of the section to these verifications.

\subsection{Recollections on the bar duality of operads}\label{sec:barduality}
We just give a brief reminder on the bar duality of operads. We are going to extend results of~\cite[\S II.9.4]{OperadHomotopyBook}.
We therefore refer to the surveys of this book for more details on the background of our constructions.

The bar duality mainly asserts that we have a cobar construction functor $B^c$,
from the category of dg cooperads to the category of (augmented) dg operads,
and a bar duality functor $B$, from the category of (augmented) dg operads to dg cooperads,
which are adjoint to each other and such that the adjunction unit $\COp\rightarrow B B^c(\COp)$
and the adjunction augmentation $B^c B(\ROp)\rightarrow\ROp$ are quasi-isomorphisms.
For our purpose, we mainly need a reminder on the expression of the cobar construction $B^c$ as a quasi-free dg operad
and of the bar construction $B$ as a quasi-cofree cooperad. We recall this definition of the cobar construction first.

The cobar construction $B^c(\COp)$ of a dg cooperad $\COp$ is a dg operad such that:
\begin{equation*}
B^c(\COp) = (\FreeOp(\Sigma^{-1}\bar{\COp}),\partial),
\end{equation*}
where we take the free operad $\FreeOp(-)$ on an arity-wise desuspension $\Sigma^{-1}$ of the coaugmentation coideal of our dg cooperad $\bar{\COp}$.
In this expression, the term $\partial$ denotes a twisting derivation, determined by the composition coproducts of our cooperad $\COp$,
which is added to the internal differential of the cooperad $\COp$
in order to determine the total differential of the cobar construction.

The idea is that the collection of the partial composition coproducts of our cooperad $\circ_i^*: \COp(k+l-1)\rightarrow\COp(k)\otimes\COp(l)$
are equivalent to a map $\theta: \Sigma^{-1}\bar{\COp}\rightarrow\FreeOp_2(\Sigma^{-1}\bar{\COp})$,
where $\FreeOp_2(-)$ denotes the homogeneous component of weight $2$ of the free operad $\FreeOp(-)$.
The twisting derivation of the cobar construction $\partial$
is the unique derivation of the free operad $\partial = \partial_{\theta}: \FreeOp(\Sigma^{-1}\bar{\COp})\rightarrow\FreeOp(\Sigma^{-1}\bar{\COp})$
such that
\begin{equation*}
\partial_{\theta}|_{\Sigma^{-1}\bar{\COp}} = \theta.
\end{equation*}
The preservation of the internal differential of the cooperad $\COp$ by the partial composition coproducts of our cooperad
is equivalent to the commutation relation $\delta\partial+\partial\delta = 0$
for our twisting derivation $\partial$,
whereas the associativity of the partial composition coproducts is equivalent to the identity $\partial^2 = 0$.
These relations $\delta\partial+\partial\delta = \partial^2 = 0$
imply that the sum $\delta+\partial$ defines a differential on the free operad $\FreeOp(\Sigma^{-1}\bar{\COp})$.

The dg operad $B^c(\COp)$ is equipped with an augmentation $\epsilon: B^c(\COp)(1)\rightarrow\kk$
induced by the projection onto the component $\FreeOp_0(-) = \kk 1$
of the free operad $\FreeOp(-)$.
In what follows, we also use the notation $\bar{B}^c(\COp)$ for the augmentation ideal of this augmented dg operad $B^c(\COp)$.

In cohomological grading, the desuspension $\Sigma^{-1}$ is given by $(\Sigma^{-1}M)^* = M^{*-1}$, for any cochain complex $M$.
The cobar construction is defined for $\ZZ$-graded dg cooperads, but in the case of a non-negatively dg cooperad $\COp\in\dgOpc$,
we get that the augmentation ideal of the dg operad $B^c(\COp)$ is concentrated in positive degrees.

In the next sections, we actually consider an extension of the cobar construction to homotopy dg cooperads.
In our context, the structure of a homotopy cooperad can be defined as the structure defined by a dg symmetric sequences $\COp$
equipped with a coaugmentation $\eta: \kk\rightarrow\COp(1)$
and a twisting map
\begin{equation*}
\theta: \Sigma^{-1}\bar{\COp}\rightarrow\bigoplus_{m>1}\FreeOp_m(\Sigma^{-1}\bar{\COp})
\end{equation*}
with values in all components of weight $m>1$ of the free operad $\FreeOp_m(-)$
such that the associated derivation $\partial = \partial_{\theta}: \FreeOp(\Sigma^{-1}\bar{\COp})\rightarrow\FreeOp(\Sigma^{-1}\bar{\COp})$
satisfies the twisting relation $\delta\partial+\partial\delta+\partial^2 = 0$
on $\FreeOp(\Sigma^{-1}\bar{\COp})$.
This relation $\delta\partial+\partial\delta+\partial^2 = 0$ ensures that the sum $\delta+\partial$
satisfies the relation of differentials $(\delta+\partial)^2 = 0$.
(We again use the letter $\delta$ for the differential of the free operad induced by the internal differential of the dg symmetric sequence $\COp$.)
For our purpose, we also consider (strict) morphisms of homotopy dg cooperads,
which are morphisms of coaugmented dg symmetric sequences $f: \COp\rightarrow\DOp$
that satisfy the relation $f\theta_{\DOp} = \theta_{\COp}\FreeOp(f)$,
where $\theta_{\COp}$ and $\theta_{\DOp}$ denote the twisting maps associated to our objects $\COp$ and $\DOp$.

The bar construction of an augmented dg operad $B(\ROp)$, defined by taking the dual construction of the cobar construction,
is a quasi-cofree cooperad such that:
\begin{equation*}
B(\ROp) = (\FreeOp^c(\Sigma\bar{\ROp}),\partial),
\end{equation*}
where we now take the cofree cooperad $\FreeOp^c(-)$ on an arity-wise suspension $\Sigma$ of the augmentation ideal of our object $\bar{\ROp}$
and we consider a twisting coderivation $\partial$ determined by the composition products of the operad $\ROp$.
To be more precise, for the bar construction, we use that the partial composition products of our operad
are equivalent to a map $\theta: \FreeOp_2^c(\Sigma\bar{\ROp})\rightarrow\Sigma\bar{\ROp}$,
where we take the homogeneous component of weight $2$ of the cofree cooperad $\FreeOp_2^c(-)$,
and the twisting coderivation is the unique coderivation $\partial: \FreeOp^c(\Sigma\bar{\ROp})\rightarrow\FreeOp^c(\Sigma\bar{\ROp})$
that lifts this map to the cofree cooperad through the canonical projection $\pi: \FreeOp^c(\Sigma\bar{\ROp})\rightarrow\Sigma\bar{\ROp}$.

In cohomological grading, the suspension $\Sigma$ is given by $(\Sigma M)^* = M^{*+1}$, for any cochain complex $M$.
In particular, the bar construction of a dg operad forms a $\ZZ$-graded dg cooperad in general,
but if the augmentation of the ideal of the dg operad $\ROp$ is concentrated in positive degrees,
then $B(\ROp)$ forms a non-negatively graded dg cooperad.

The cobar construction preserves the quasi-isomorphisms of non-negatively graded dg cooperads and the bar construction preserves all quasi-isomorphisms.
We also already recalled that the unit morphism and the augmentation morphism of the cobar-bar adjunction are quasi-isomorphisms.

\subsection{The bar-cobar construction and the characterization of (acyclic) fibrations in the category dg cooperads}
The purpose of this subsection is to give a characterization of (acyclic) fibrations of dg cooperads in terms of the bar-cobar construction.

We actually consider general homotopy dg cooperads in our construction.
We mainly use that the cobar construction of a homotopy dg cooperad $B^c(\COp) = (\FreeOp(\Sigma^{-1}\bar{\COp}),\partial)$
is equipped with a descending filtration
\begin{equation*}
B^c(\COp) = F^0 B^c(\COp)\supset F^1 B^c(\COp)\supset\cdots\supset F^s B^c(\COp)\supset\cdots
\end{equation*}
such that $B^c(\COp) = \lim_s B^c(\COp)/F^s B^c(\COp)$, we have $B^c(\COp)/F^1 B^c(\COp) = \kk 1$
and the composition products behave additively with respect to the filtration order.
(We have $a\in F^s B^c(\COp)(k),b\in F^t B^c(\COp)(l)\Rightarrow a\circ_i b\in F^{s+t} B^c(\COp)(k+l-1)$.)
For this purpose, we just set $F^s B^c(\COp) = (\bigoplus_{m\geq s}\FreeOp_m(\Sigma^{-1}\bar{\COp}),\partial)$,
where we consider the components of weight $m\geq s$ of the free operad $\FreeOp(-)$.
The relation $B^c(\COp) = \lim_s B^c(\COp)/F^s B^c(\COp)$, which is only valid when the dg cooperad $\COp$ is non negatively graded,
follows from the observation that $\Sigma^{-1}\COp$ is concentrated in positive degrees
and the cochain complex $\FreeOp_m(\Sigma^{-1}\bar{\COp})$,
which consists of tensors of order $m$ in the generating symmetric sequence $\Sigma^{-1}\bar{\COp}$,
vanishes in degrees less or equal to $m$.
The relation $B^c(\COp)/F^1 B^c(\COp) = \kk 1$ is equivalent to the identity $\bar{B}^c(\COp) = F^1 B^c(\COp)$,
where we again use the notation $\bar{B}^c(\COp)$ for the augmentation ideal of the cobar construction.

When we pass to the bar construction, we get a quasi-cofree dg cooperad structure $B B^c(\COp) = (\FreeOp(\Sigma\bar{B}^c(\COp)),\partial)$
which satisfies the assumption of the following proposition.

\begin{prop}\label{prop:fibrancy}
Let $(\FreeOp^c(\MOp),\partial)$ and $(\FreeOp^c(\NOp),\partial)$ be quasi-cofree dg cooperads.
We assume that the twisting coderivation of these cooperads $\partial = \partial_{\MOp}$ and $\partial = \partial_{\NOp}$
are determined by twisting maps $\theta_{\MOp}: \FreeOp^c(\MOp)\rightarrow\MOp$ and $\theta_{\NOp}: \FreeOp^c(\NOp)\rightarrow\NOp$
such that $\theta_{\MOp}|_{\MOp} = \theta_{\NOp}|_{\NOp} = 0$.
We also assume that $\MOp$ and $\NOp$ are equipped with descending filtrations $\MOp = F^1\MOp\supset\cdots\supset F^s\MOp\supset\cdots$
and $\NOp = F^1\NOp\supset\cdots\supset F^s\NOp\supset\cdots$
such that $\MOp = \lim_s\MOp/F^s\MOp$, $\NOp = \lim_s\NOp/F^s\NOp$,
and the twisting maps $\theta_{\MOp}$ and $\theta_{\NOp}$ behave additively with respect to these filtrations.

Let $\psi_f = \FreeOp^c(f): (\FreeOp^c(\MOp),\partial)\rightarrow(\FreeOp^c(\NOp),\partial)$ be a morphism of dg cooperads
induced by a filtration preserving morphism of symmetric sequences $f: \MOp\rightarrow \NOp$
such that $f\theta_{\MOp} = \theta_{\NOp}\FreeOp^c(f)$.
If $f$ is surjective, then $\psi_f$ is a fibration.
\end{prop}

\begin{proof}
The proof of this proposition is essentially identical to the proof of~\cite[Proposition II.9.2.10]{OperadHomotopyBook}.
We just replace the arity filtration by the filtrations given in the assumptions of our proposition
and we consider the tower decompositions such that $(\FreeOp^c(\MOp),\partial) = \lim_s(\FreeOp^c(\MOp/F^s\MOp),\partial)$
and $(\FreeOp^c(\NOp),\partial) = \lim_s(\FreeOp^c(\NOp/F^s\NOp),\partial)$
rather than the arity-wise decomposition.
\end{proof}

\begin{rem}
This proposition is still valid for $\ZZ$-graded dg symmetric sequences, for the model structure which we define in Appendix~\ref{sec:modelcat}.
Recall however that we need to start with a non-negative graded objects in order to apply this proposition
to the bar-cobar construction of homotopy dg cooperads.
\end{rem}

We then have the following statement:

\begin{prop}\label{prop:barcobarfibrations}
The image of a morphism of homotopy dg cooperads under the bar-cobar construction $B B^c(p): B B^c(\COp)\rightarrow B B^c(\DOp)$
defines a fibration of dg cooperads as soon as the morphism $p: \COp\rightarrow\DOp$
defines a fibration in the category of dg symmetric sequences.
This fibration $B B^c(p)$ is also acyclic when $p$ is so.
\end{prop}

\begin{proof}
The first assertion of the proposition is a direct corollary of the result of Proposition~\ref{prop:fibrancy}
since we observed that the bar-cobar construction is endowed with a filtration
which satisfies the assumptions of this lemma.
To establish the second assertion of the proposition,
we just use that the cobar construction preserves the quasi-isomorphisms of non-negatively graded homotopy dg cooperads
and that the bar constructions preserves all quasi-isomorphisms.
\end{proof}

We immediately deduce from the general properties of (acyclic) fibrations
in a model category that every morphism $q: B B^c(\COp)\times_{B B^c(\DOp)}\DOp\rightarrow\DOp$
defined by taking the pullback of an (acyclic) fibration of the form of Proposition~\ref{prop:barcobarfibrations}
along the unit morphism of the cobar-bar adjunction $\DOp\xrightarrow{\sim}B B^c(\DOp)$, for $\DOp$ a dg cooperad,
is still an (acyclic) fibration in the category of dg cooperads.
We aim to prove that all (acyclic) fibrations of the category of dg cooperads occur as retracts of pullbacks of (acyclic) fibrations of this form.
We establish this claim in a series of lemmas.
We start with the following result.

\begin{lemm}\label{lemm:barcobarfactorization}
Let $f: \COp\rightarrow\DOp$ be any morphism in the category of dg cooperads.
We pick a factorization of this morphism in the category of coaugmented dg symmetric sequences $f = p i$
such that $i: \COp\rightarrow\ZOp$ is an acyclic cofibration
and $p: \ZOp\rightarrow\DOp$ is a fibration, where $\ZOp$ denotes the middle term of this factorization.
This dg symmetric sequence $\ZOp$ can be equipped with a homotopy dg cooperad structure
such that the morphisms $i$ and $p$ define morphisms of homotopy dg cooperads.
\end{lemm}

\begin{proof}
Let $\theta_{\COp}$ and $\theta_{\DOp}$ be the twisting maps that determine the homotopy dg cooperad structure of the objects $\COp$ and $\DOp$.
In what follows, we also use the decompositions $\theta_{\COp} = \sum_{m>1}\theta_{\COp}^m$ and $\theta_{\DOp} = \sum_{m>1}\theta_{\DOp}^m$,
where $\theta_{\COp}^m$ and $\theta_{\DOp}^m$ denote the components of our twisting maps
with values in the component of weight $m>1$ of the free operad.

We build the components $\theta_{\ZOp}^m: \Sigma^{-1}\bar{\ZOp}\rightarrow\FreeOp_m(\Sigma^{-1}\bar{\ZOp})$
of our twisting map on~$\ZOp$
\begin{equation*}
\theta_{\ZOp} = \sum_{m>1}\theta_{\ZOp}^m: \Sigma^{-1}\bar{\ZOp}\rightarrow\FreeOp(\Sigma^{-1}\bar{\ZOp}) = \bigoplus_{m>1}\FreeOp_m(\Sigma^{-1}\bar{\ZOp})
\end{equation*}
by induction on the weight $m>1$.
We use that the relation $\delta\partial_{\theta_{\ZOp}}+\partial_{\theta_{\ZOp}}\delta + \partial_{\theta_{\ZOp}}\partial_{\theta_{\ZOp}} = 0$
holds as soon as we have the identity $\delta(\theta_{\ZOp}) + \partial_{\theta_{\ZOp}}\theta_{\ZOp} = 0$
in the dg-hom object
\begin{equation*}
T_{\ZOp,\ZOp} = \Hom(\Sigma^{-1}\bar{\ZOp},\FreeOp(\Sigma^{-1}\bar{\ZOp}))
\end{equation*}
with the differential such that $\delta(u) = \delta u - \pm u\delta$ for any map $u\in\Hom(\Sigma^{-1}\bar{\ZOp},\FreeOp(\Sigma^{-1}\bar{\ZOp}))$.
When we project this relation onto $\FreeOp_m(\Sigma^{-1}\bar{\ZOp})$,
we get a relation of the form
\begin{equation*}
\delta(\theta_{\ZOp}^m) + \partial_{\theta_{\ZOp}^*}\theta_{\ZOp}^* = 0,
\end{equation*}
where the composite $\partial_{\theta_{\ZOp}^*}\theta_{\ZOp}^*: \Sigma^{-1}\bar{\ZOp}\rightarrow\FreeOp_m(\Sigma^{-1}\bar{\ZOp})$
only involves the components $\theta_{\ZOp}^l$ of weight $l<m$
of our twisting map $\theta_{\ZOp}$.

We easily check that this composite $\partial_{\theta_{\ZOp}^*}\theta_{\ZOp}^*$ defines a cocycle in our dg hom object
and our aim is to prove that this cocycle is a coboundary.
We also assume by induction that the morphisms $i$ and $p$ preserve our twisting maps in weight $l<m$.
We then have a cocycle of the form
\begin{equation*}
\zeta = (\partial_{\theta_{\COp}^*}\theta_{\COp}^*,\partial_{\theta_{\ZOp}^*}\theta_{\ZOp}^*,\partial_{\theta_{\DOp}^*}\theta_{\DOp}^*)
\end{equation*}
in the fiber product $T_{\COp,\COp}^m\times_{T_{\COp,\ZOp}^m}\times T_{\ZOp,\ZOp}^m\times_{T_{\ZOp,\DOp}^m}\times T_{\DOp,\DOp}^m$,
where for $\MOp,\NOp\in\{\COp,\DOp,\ZOp\}$ we set $T_{\MOp,\NOp}^m = \Hom(\Sigma^{-1}\bar{\MOp},\FreeOp_m(\Sigma^{-1}\bar{\NOp}))$.
We consider the map
\begin{equation*}
T_{\COp,\COp}^m\times_{T_{\COp,\ZOp}^m}\times T_{\ZOp,\ZOp}^m\times_{T_{\ZOp,\DOp}^m}\times T_{\DOp,\DOp}^m
\rightarrow T_{\COp,\COp}^m\times_{T_{\COp,\DOp}^m} T_{\DOp,\DOp}^m
\end{equation*}
induced by the obvious projection of hom-objects.
We easily check that this map defines an acyclic fibration of cochain complexes, because we can identify this map with a base extension
of the pullback-corner morphism $T_{\ZOp,\ZOp}^m\rightarrow T_{\COp,\ZOp}^m\times_{T_{\COp,\DOp}^m} T_{\ZOp,\DOp}^m$
which is an acyclic fibration as soon as the map $i: \COp\rightarrow\ZOp$ is an acyclic cofibration
whereas $p: \ZOp\rightarrow\DOp$ is a fibration in the category of dg symmetric sequences.
We now have the coboundary relation
\begin{equation*}
(\delta(\theta_{\COp}^m),\delta(\theta_{\DOp}^m)) = (-\partial_{\theta_{\COp}^*}\theta_{\COp}^*,-\partial_{\theta_{\DOp}^*}\theta_{\DOp}^*)
\end{equation*}
in $T_{\COp,\COp}^m\times_{T_{\COp,\DOp}^m} T_{\DOp,\DOp}^m$.
We just lift this coboundary relation through our acyclic fibration to obtain the existence of a map $\theta_{\ZOp}^m\in T_{\ZOp,\ZOp}^m$
such that:
\begin{equation*}
\theta_{\ZOp}^m i^* = i_*\theta_{\COp}^m,
\quad p_*\theta_{\ZOp}^m = \theta_{\DOp}^m p^*,
\quad\text{and}
\quad\delta(\theta_{\ZOp}^m) = -\partial_{\theta_{\ZOp}^*}\theta_{\ZOp}^*,
\end{equation*}
where $(i^*,i_*)$ and $(p^*,p_*)$ denote the morphisms induced by our maps $i: \COp\rightarrow\ZOp$ and $p: \ZOp\rightarrow\DOp$ on hom-objects,
and this result completes our construction.
\end{proof}

\begin{lemm}\label{lem:fib1}
Let $f: \COp\rightarrow\DOp$ be a morphism of dg cooperads as in Lemma~\ref{lemm:barcobarfactorization}.
Let $\ZOp$ be the homotopy dg cooperad given by the result of this lemma.
The morphism of dg cooperads
\begin{equation*}
g: \COp\rightarrow B B^c(\ZOp)\times_{B B^c(\DOp)}\DOp
\end{equation*}
that we deduce from the pullback diagram
\begin{equation*}
\xymatrix{ \COp\ar[d]\ar@{.>}[dr]^{g}\ar@/^1em/[drr]^{f} && \\
B B^c(\COp)\ar@/_1em/[dr]_{B B^c(i)} & B B^c(\ZOp)\times_{B B^c(\DOp)}\DOp\ar@{.>>}[r]\ar@{.>}[d] & \DOp\ar[d] \\
& B B^c(\ZOp)\ar@{->}[r] & \DOp }
\end{equation*}
is an acyclic cofibration.
\end{lemm}

\begin{proof}
We use that the cooperadic product preserves quasi-isomorphisms: if $\AOp\rightarrow\BOp$ is a quasi-isomorphism,
then the map $\AOp\times\DOp\rightarrow\BOp\times\DOp$ is a quasi-isomorphism,
for any dg cooperad $\DOp$.
This statement readily follows from the treewise description of cooperadic products (see \cite[\S C.1.16]{OperadHomotopyBook}).

The acyclicity implies that the cofibration $\COp\rightarrow\ZOp$ is injective in all degrees (not only in positive degrees).
The morphism $\COp\rightarrow B B^c(\COp)\rightarrow B B^c(\ZOp)$,
which is an acyclic cofibration since the unit morphism of the cobar-bar adjunction is so and the composite functor $B B^c(-)$ preserves the quasi-isomorphisms as well as the injective maps,
factors through the map considered in the lemma:
\begin{equation*}
\COp\xrightarrow{g} B B^c(\ZOp)\times_{B B^c(\DOp)}\DOp\rightarrow B B^c(\ZOp).
\end{equation*}
This observation implies that $g$ is injective and hence forms a cofibration.
Furthermore, to complete our verification, we are left with verifying that the map $h: B B^c(\ZOp)\times_{B B^c(\DOp)}\DOp\rightarrow B B^c(\ZOp)$
is a quasi-isomorphism.

In a first step, we give a more explicit description of the fiber product of dg cooperads $B B^c(\ZOp)\times_{B B^c(\DOp)}\DOp$.
For this purpose, we fix a splitting of the map of dg symmetric sequences $p: \ZOp\twoheadrightarrow\DOp$,
so that we can write
\begin{equation*}
\ZOp^{\flat}\simeq\DOp^{\flat}\oplus\SOp,
\end{equation*}
for some graded symmetric sequence $\SOp$.
This splitting gives rise to a splitting of graded symmetric sequences at the level of the cobar construction
\begin{equation*}
B^c(\ZOp)^{\flat}\simeq B^c(\DOp)^{\flat}\oplus\TOp,
\end{equation*}
where $\TOp$ is spanned by decorated trees such that at least one vertex carries a decoration in $\SOp$,
and we get an isomorphism
\begin{equation*}
B B^c(\ZOp)^{\flat}\simeq\FreeOp^c(\Sigma\bar{B}^c(\DOp)^{\flat}\oplus\Sigma\TOp)\simeq B B^c(\DOp)^{\flat}\times\FreeOp^c(\Sigma\TOp)
\end{equation*}
when we pass to the cobar-bar construction.
For our fiber product, we similarly have:
\begin{equation*}
B B^c(\ZOp)^{\flat}\times_{B B^c(\DOp)^{\flat}}\DOp^{\flat}\simeq\DOp^{\flat}\times\FreeOp^c(\Sigma\TOp),
\end{equation*}
and the map $h: B B^c(\ZOp)\times_{B B^c(\DOp)}\DOp\rightarrow B B^c(\ZOp)$ is identified with the product of the identity map on $\FreeOp^c(\Sigma\TOp)$
with the canonical quasi-isomorphism $\DOp\rightarrow B B^c(\TOp)$.
	
We filter both the domain and the target of this map $h$ by the total cohomological degree in $\TOp$.
This filtration is bounded.
The differential on the associated graded of the object $B B^c(\ZOp)\times_{B B^c(\ZOp)}\DOp$
reduces to the differential of the dg cooperad $\DOp$,
whereas the differential on the associated graded of $B B^c(\ZOp)$
reduces to the differential of the object $B B^c(\DOp)$.
We deduce from this inspection that the associated graded of the map $h$ is a quasi-isomorphism and, hence, so is the morphism $h$
by standard spectral sequences arguments.
This verification completes the proof of our lemma.
\end{proof}

\begin{rem}
Note that the above argument does not readily extend to the category $\dgZOpc$,
because in this context the filtration used in the proof of our lemma
is not bounded.
\end{rem}

We can now complete the proof of our description of the class of (acyclic) fibrations in the category of dg cooperads:

\begin{prop}\label{prop:f and Wf}
Every (acyclic) fibration of dg cooperads $f: \COp\rightarrow\DOp$ can be obtained as a retract of a pullback of an (acyclic) fibration
of the form $B B^c(p): B B^c(\ZOp)\rightarrow B B^c(\DOp)$,
where $p: \ZOp\rightarrow\DOp$ is as in the statement of Lemma~\ref{lemm:barcobarfactorization}.
\end{prop}

\begin{proof}
The fibration $f$ fits into the following diagram (solid arrows), in which the right-hand square is a pullback square:
\begin{equation*}
\xymatrix{ \COp\ar[d]_{=}\ar@{>->}[]!R+<4pt,0pt>;[r]^-{\sim} & B B^c(\ZOp)\times_{B B^c(\DOp)}\DOp\ar[d]\ar[r]\ar@{.>}[dl] & B B^c(\ZOp)\ar[d] \\
\COp\ar@{->>}[r] & \DOp\ar[r] & B B^c(\DOp) }.
\end{equation*}
We use the result of Lemma~\ref{lem:fib1} to check that the upper left horizontal arrow of this diagram is an acyclic cofibration
and we use that our fibration $p: \COp\rightarrow\DOp$
has the right lifting property with respect to this acyclic cofibration to pick the dotted diagonal arrow in our diagram,
which exhibits our map $p: \COp\rightarrow\DOp$ as a retract of the middle vertical arrow $B B^c(\ZOp)\times_{B B^c(\DOp)}\DOp\rightarrow\DOp$.
\end{proof}

\begin{rem}
The result of Proposition~\ref{prop:f and Wf} implies that every fibration in the category of dg cooperads is surjective,
because one can check that this is the case for the map $B B^c(\ZOp)\times_{B B^c(\DOp)}\DOp\rightarrow\DOp$,
associated to a fibration of dg symmetric sequences $p: \ZOp\twoheadrightarrow\DOp$,
which we produce in our construction.
\end{rem}

\subsection{The bar-cobar construction and the proof of the pullback-corner property for function objects}
We carry out the verification of the pullback-corner property in the category of dg cooperads in this subsection:

\begin{claim}\label{claim:pullbackcorner}
The pullback-corner morphism $(p_*,i^*): \COp^L\rightarrow\DOp^L\times_{\DOp^K}\COp^K$
associated to a cofibration of simplicial sets $i: K\rightarrowtail L$
and to a fibration of dg cooperads $p: \COp\twoheadrightarrow\DOp$
is a fibration of dg cooperads, which is also acyclic if $i$ or $p$ is so.
\end{claim}

We prove in a first step that the pullback-corner property is valid for the (acyclic) fibrations of dg cooperads
of the form $B B^c(p): B B^c(\ZOp)\rightarrow B B^c(\DOp)$,
where we consider a morphism of homotopy dg cooperads $p: \ZOp\rightarrow\DOp$
that defines an (acyclic) fibration in the category of dg symmetric sequences.
We rely on the following computation of function objects for the bar construction of augmented dg operads:

\begin{prop}\label{prop:barfunctionobject}
For the bar construction of an augmented dg operad $\BOp = B(\ROp)$, we have an identity
\begin{equation*}
B(\ROp)^{\Delta^{\bullet}} = B(\ROp\bar{\otimes}\Omega^*(\Delta^{\bullet})),
\end{equation*}
where we consider the dg operad such that
\begin{equation*}
(\ROp\bar{\otimes}\Omega^*(\Delta^{\bullet}))(r) = \begin{cases} \kk\oplus\bar{\ROp}(1)\otimes\Omega^*(\Delta^{\bullet}), & \text{for $r=1$}, \\
\ROp(r)\otimes\Omega^*(\Delta^{\bullet}), & \text{for $r>1$}. \end{cases}
\end{equation*}
\end{prop}

\begin{proof}
We go back to the definition of the object $\BOp^{\Delta^{\bullet}}$ by the adjunction relation $\Map(\AOp,\BOp) = \Mor(\AOp,B(\ROp)^{\Delta^{\bullet}})$.
We have an identity
\begin{equation*}
B(\ROp)\otimes\Omega^*(\Delta^{\bullet}) = B_{\Omega^*(\Delta^{\bullet})}(\ROp\otimes\Omega^*(\Delta^{\bullet})),
\end{equation*}
where on the right-hand side we consider the bar construction of the object such that $(\ROp\otimes\Omega^*(\Delta^{\bullet}))(r) = \ROp(r)\otimes\Omega^*(\Delta^{\bullet})$
in the category of dg operads over $\Omega^*(\Delta^{\bullet})$.

In what follows, we use the notation $\dgOp_R^+$ for the category of augmented dg operads over a dg commutative algebra $R$,
which we take to be either the dg algebra $R = \Omega^*(\Delta^{\bullet})$ (as above) or the ground ring $R = \kk$,
in which case we set $\dgOp^+ = \dgOp_{\kk}^+$ for short.
The bar duality implies that we have an adjunction relation:
\begin{multline*}
\Mor_{\dgOp^c_{\Omega^*(\Delta^{\bullet})}}(\AOp\otimes\Omega^*(\Delta^{\bullet}),B(\ROp)\otimes\Omega^*(\Delta^{\bullet}))\\
\begin{aligned}
& \simeq\Mor_{\dgOp^c_{\Omega^*(\Delta^{\bullet})}}(\AOp\otimes\Omega^*(\Delta^{\bullet}),B_{\Omega^*(\Delta^{\bullet})}(\ROp\otimes\Omega^*(\Delta^{\bullet}))) \\
& \simeq\Mor_{\dgOp_{\Omega^*(\Delta^{\bullet})}^+}(B^c_{\Omega^*(\Delta^{\bullet})}(\AOp\otimes\Omega^*(\Delta^{\bullet})),
\ROp\otimes\Omega^*(\Delta^{\bullet})),
\end{aligned}
\end{multline*}
where $B^c_{\Omega^*(\Delta^{\bullet})}(-)$ denotes the cobar construction on the category of dg cooperads over $\Omega^*(\Delta^{\bullet})$
and we consider the set of morphisms in the category of augmented dg operads over $\Omega^*(\Delta^{\bullet})$.
We still have an identity $B^c_{\Omega^*(\Delta^{\bullet})}(\AOp\otimes\Omega^*(\Delta^{\bullet})) = B^c(\AOp)\otimes\Omega^*(\Delta^{\bullet})$
and, by an immediate scalar extension relation, we have:
\begin{equation*}
\Mor_{\dgOp_{\Omega^*(\Delta^{\bullet})}^+}(B^c(\AOp)\otimes\Omega^*(\Delta^{\bullet}),\ROp\otimes\Omega^*(\Delta^{\bullet}))
\simeq\Mor_{\dgOp^+}(B^c(\AOp),\ROp\bar{\otimes}\Omega^*(\Delta^{\bullet})),
\end{equation*}
where we consider the augmented operad of the proposition $\ROp\bar{\otimes}\Omega^*(\Delta^{\bullet})$.
(We just have to take care that the summand of the operadic unit is apart when we use the extension of scalars.)

By applying the bar duality yet again, we eventually obtain:
\begin{equation*}
\Mor_{\dgOp^+}(B^c(\AOp),\ROp\bar{\otimes}\Omega^*(\Delta^{\bullet}))
\simeq\Mor_{\dgOpc}(\AOp,B(\ROp\bar{\otimes}\Omega^*(\Delta^{\bullet}))),
\end{equation*}
and the conclusion of the proposition follows.
\end{proof}

\begin{rem}
The result of this Proposition implies that the object $B(\ROp)^{\Delta^{\bullet}}$
is given by a quasi-cofree dg cooperad such that $B(\ROp)^{\Delta^{\bullet}} = (\FreeOp^c(\Sigma\bar{\ROp}\otimes\Omega^*(\Delta^{\bullet})),\partial)$.
In fact, this statement is more general.
To be more precise, for a quasi-cofree dg cooperad $\BOp = (\FreeOp^c(\MOp),\partial)$ with a twisting coderivation determined by a map $\theta: \FreeOp^c(\MOp)\rightarrow\MOp$,
we have an identity $\BOp^{\Delta^{\bullet}} = (\FreeOp^c(\MOp\otimes\Omega^*(\Delta^{\bullet}),\partial)$,
where the twisting coderivation is determined by a map $\theta': \FreeOp^c(\MOp\otimes\Omega^*(\Delta^{\bullet}))\rightarrow\MOp\otimes\Omega^*(\Delta^{\bullet})$
defined by using the map $\theta$ and the products of the dg algebra $\Omega^*(\Delta^{\bullet})$.
\end{rem}

The construction of the previous proposition is actually used in~\cite[\S II.9.4]{OperadHomotopyBook} (without the correspondence with our mapping space)
in order to define simplicial frames in the category of dg cooperads.
We can now establish the following statement.

\begin{lemm}\label{lem:pcp_Wf}
The pullback-corner property of Claim~\ref{claim:pullbackcorner} is satisfied for the morphisms of dg cooperads
of the form $B B^c(p): B B^c(\ZOp)\rightarrow B B^c(\DOp)$,
where $p: \ZOp\rightarrow\DOp$ is a morphism of homotopy dg cooperads that defines a fibration
or an acyclic fibration in the category of dg symmetric sequences.
\end{lemm}

\begin{proof}
We use that we can reduce the proof of the pullback-corner property to the case of generating cofibrations $i: \partial\Delta^n\rightarrow\Delta^n$
and of generating acyclic cofibrations $i: \Lambda_k^n\rightarrow\Delta^n$
in the category of simplicial sets.
We assume that $i: K\rightarrow L$ is such a morphism.
In the case where $K = \partial\Delta^n,\Lambda_k^n$, we can extend the result of Proposition~\ref{prop:barfunctionobject}
to establish that the function object $B(\ROp)^K$
associated to the bar construction of an augmented dg operad $\BOp = B(\ROp)$
has the form
\begin{equation*}
B(\ROp)^K = B(\ROp\bar{\otimes}\Omega^*(K)),
\end{equation*}
where we again set:
\begin{equation*}
(\ROp\bar{\otimes}\Omega^*(K))(r) = \begin{cases} \kk\oplus\bar{\ROp}(1)\otimes\Omega^*(\Delta^{\bullet}), & \text{for $r=1$}, \\
\ROp(r)\otimes\Omega^*(K), & \text{for $r>1$}. \end{cases}
\end{equation*}
In the case $K = \partial\Delta^n$, this observation follows from the relations:
\begin{align*}
& \partial\Delta^n = \colim_{\substack{u: \underline{k}\rightarrow\underline{n}\\k<n}}\Delta^k,
\quad\Omega^*(\partial\Delta^n) = \lim_{\substack{u: \underline{k}\rightarrow\underline{n}\\k<n}}\Omega^*(\Delta^k),
\intertext{and from the limit interchange formula:}
& B(\ROp)^{\partial\Delta^n} = \lim_{\substack{u: \underline{k}\rightarrow\underline{n}\\k<n}}B(\ROp)^{\Delta^k}
\simeq\lim_{\substack{u: \underline{k}\rightarrow\underline{n}\\k<n}} B(\ROp\bar{\otimes}\Omega^*(\Delta^k))
\simeq B(\ROp\bar{\otimes}\lim_{\substack{u: \underline{k}\rightarrow\underline{n}\\k<n}}\Omega^*(\Delta^k))
\end{align*}
(see the proof of~\cite[Proposition II.9.4.3]{OperadHomotopyBook}).
We use a similar argument line in the case $K = \Lambda_k^n$.
(The identity $B(\ROp)^K = B(\ROp\bar{\otimes}\Omega^*(K))$ actually holds as soon as our simplicial set $K$ has finitely many non degenerate simplices.)

For a morphism of dg operads $\phi: \ROp\rightarrow\SOp$,
we also have the identity
\begin{equation*}
B(\ROp)^L\times_{B(\SOp)^K} B(\SOp)^K = B(\ROp\bar{\otimes}\Omega^*(L)\times_{\ROp\bar{\otimes}\Omega^*(K)}\SOp\bar{\otimes}\Omega^*(K)),
\end{equation*}
because the bar construction preserves limits by adjunction.
Hence, the morphism of the lemma is identified with a morphism of quasi-cofree cooperads of the form
\begin{equation}\tag{*}\label{eqn:barcobarpullbackcornermorphism}
\FreeOp^c(p_*,i^*): (\FreeOp^c(\MOp\otimes\Omega^*(L)),\partial)
\rightarrow(\FreeOp^c(\NOp\otimes\Omega^*(L)\times_{\NOp\otimes\Omega^*(K)}\MOp\otimes\Omega^*(K)),\partial),
\end{equation}
where we set $\MOp = \Sigma\bar{B}^c(\ZOp)$ and $\NOp = \Sigma\bar{B}^c(\DOp)$ for short
and we consider the pullback-corner of the morphisms $p_*: \Sigma\bar{B}^c(\ZOp)\rightarrow\Sigma\bar{B}^c(\DOp)$ and $i^*: \Omega^*(L)\rightarrow\Omega^*(K)$
in the category of dg symmetric sequences.
The functor $B^c(-)$ trivially preserves fibrations in the category of symmetric sequences,
as well as the acyclic fibrations
since the cobar construction also preserves the quasi-isomorphisms
of non-negatively graded dg cooperads (see Section~\ref{sec:barduality}).
This result implies that the pullback-corner morphism $(p_*,i^*)$ is a fibration of cochain complexes arity-wise,
which is also acyclic as soon as $p$ or $i$ is so,
because the tensor product with the dg commutative algebras $\Omega^*(K)$
define a bifunctor that fulfills the pullback-corner property in the category of cochain complexes (adapt the observations
of~\cite[\S II.7.1]{OperadHomotopyBook}).

Both dg symmetric sequences $\SOp = \MOp\otimes\Omega^*(L)$ and $\TOp = \NOp\otimes\Omega^*(L)\times_{\NOp\otimes\Omega^*(K)}\MOp\otimes\Omega^*(K)$
inherit an appropriate filtration from the cobar construction, so that we can apply the result of Proposition~\ref{prop:fibrancy}
to conclude that our morphism of quasi-cofree cooperads~\eqref{eqn:barcobarpullbackcornermorphism}
is a fibration of dg cooperads.
Then we just use that the bar construction preserves the quasi-isomorphisms to obtain that this fibration is also acyclic
whenever $p$ or $i$ is so.
\end{proof}

We then have the following stability results.

\begin{lemm}\label{lem:pcp_pullback_stable}
If the pullback-corner property of Claim~\ref{claim:pullbackcorner} holds for a given (acyclic) fibration of dg cooperads $p: \COp\twoheadrightarrow\DOp$,
then this result is also valid for the (acyclic) fibration $q: \COp\times_{\DOp}\BOp\rightarrow\BOp$
which we obtain by pulling back our morphism $p$
along a morphism of dg cooperads $h: \BOp\rightarrow\DOp$.
\end{lemm}

\begin{proof}
Recall that our function object bifunctor $(\DOp,K)\mapsto\DOp^K$ preserves the finite limits in the variable $\DOp$. We accordingly have:
\begin{equation*}
(\COp\times_{\DOp}\BOp)^L\simeq\COp^L\times_{\DOp^L}\BOp^L
\quad\text{and}
\quad\BOp^L\times_{\BOp^K}(\COp\times_{\DOp}\BOp)^K \simeq\COp^K\times_{\DOp^K}\BOp^L,
\end{equation*}
by standard properties of fiber products and we can identify the morphism
\begin{equation*}
(\COp\times_{\DOp}\BOp)^L\rightarrow\BOp^L\times_{\BOp^K}(\COp\times_{\DOp}\BOp)^K
\end{equation*}
with the image of the pullback-corner morphism $\COp^L\rightarrow\DOp^L\times_{\DOp^K}\COp^K$
under the base extension $-\times_{\DOp^L}\BOp^L$.
The result of the lemma follows.
\end{proof}

\begin{lemm}\label{lem:pcp_retract_stable}
If the pullback-corner property of Claim~\ref{claim:pullbackcorner} holds for a given (acyclic) fibration of dg cooperads $q: \COp\twoheadrightarrow\DOp$,
then this result is also valid for any retract of this (acyclic) fibration $f: \AOp\twoheadrightarrow\BOp$.
\end{lemm}

\begin{proof}
This lemma follows from the fact that the pullback-corner morphism $(f_*,i^*): \AOp^L\rightarrow\BOp^L\times_{\BOp^K}\AOp^K$
can be identified with a retract of the pullback-corner morphism $(q_*,i^*): \COp^L\rightarrow\DOp^L\times_{\DOp^K}\COp^K$
associated to our given fibration $q: \COp\twoheadrightarrow\DOp$.
\end{proof}

We can now conclude that:

\begin{prop}
The pullback-corner property of Claim~\ref{claim:pullbackcorner} holds for every (acyclic) fibration of dg cooperads $p: \COp\rightarrow\DOp$.
\end{prop}

\begin{proof}
We just that any (acyclic) fibration of dg cooperads $p: \COp\rightarrow\DOp$ is a retract of an (acyclic) fibration
of the form $q: B B^c(\ZOp)\times_{B B^c(\DOp)}\DOp$
by Proposition~\ref{prop:f and Wf} and we use the results of Lemma~\ref{lem:pcp_Wf}-\ref{lem:pcp_retract_stable}
to establish that this retract satisfies the pullback-corner property of our claim.
\end{proof}

The proof of this proposition completes the verification of the analogous pullback-corner property in the category of dg Hopf cooperads in Proposition~\ref{prop:pullbackcorner}
since we observed that the forgetful functor from dg Hopf cooperads to dg cooperads
preserves all the structures involved in our constructions.
The proof of Proposition~\ref{prop:pullbackcorner} is therefore complete.\qed

\section{Homotopy automorphism spaces}\label{sec:hautspaces}
Recall that the homotopy automorphism space of an object $X$ in a simplicial model category $\CCat$
is defined to be the simplicial monoid:
\begin{equation*}
\Aut^h(X)\sim\Map_{\CCat}(RX,RX)^{\htimes},
\end{equation*}
where $RX$ is a cofibrant and fibrant replacement of $X$ and $\Map_{\CCat}(-)^{\htimes}\subset\Map_{\CCat}(-)$
consists of the connected components of the mapping space
associated to homotopy invertible morphisms $\phi: RX\xrightarrow{\sim} RX$.

We refer to \cite[\S II.2.2]{OperadHomotopyBook} for a detailed account on this definition and for a proof that this simplicial monoid
does not depend on the choice of the cofibrant replacement up to weak equivalence
in the category of simplicial monoids.
We should also note that the object $\Aut^h(X)$ does not depend on the choice of the simplicial structure that we associate to our category $\CCat$.
This result follows from~\cite[Proposition 4.8]{DwyerKan},
where it is proved that the mapping spaces of a simplicial model category are weakly equivalent to the hom-objects of a simplicial localization
whose definition only depends on the choice of the class of weak equivalences in $\CCat$.
The weak equivalences between the mapping spaces of the simplicial model structure and the hom-objects of the simplicial localization
preserve the composition of morphisms by construction, and hence, preserve homotopy automorphism spaces.

Though we do not have a full simplicial model structure on our category of dg Hopf cooperads, one can check that our right finitely continuous simplicial structure
suffices for the proofs of the cited references to go through and yield well-defined homotopy automorphism spaces.

In the proof of the main theorem of this article, we use a general recognition result for homotopy automorphism spaces
and a description of the homotopy automorphism spaces of the rationalization of operads
in terms of the model in the category of dg Hopf cooperads. We devote this section to the verification of these statements.

In the above definition of the space of homotopy automorphisms of an object $X$, we use the notation $\Map_{\CCat}(-)^{\htimes}\subset\Map_{\CCat}(-)$
for the connected components of a mapping space associated to homotopy invertible morphisms.
In the case of the mapping space $\Map_{\CCat}(RX,RX)$ associated to a cofibrant-fibrant object $RX$,
every weak equivalence $u: RX\xrightarrow{\sim} RX$
automatically defines a homotopy equivalence.
More care is necessary in what follows, because we apply the subspace construction $\Map_{\CCat}(X,Y)^{\htimes}\subset\Map_{\CCat}(X,Y)$
to objects $(X,Y)$ which are not necessarily both cofibrant and fibrant. We then assume that $\Map_{\CCat}(X,Y)^{\htimes}$
consists of the connected components of the mapping space $\Map_{\CCat}(X,Y)$
that are associated to weak equivalences $u: X\xrightarrow{\sim} Y$.
Note simply that every morphism $v$ in the same connected component of a weak equivalence $u: X\xrightarrow{\sim} Y$
is also a weak equivalence,
because the function object $Y^{\Delta^1}$ satisfies the relation $Y\sim Y^{\Delta^1}$
and hence defines a path-object for $Y$
in the model category $\CCat$.

\subsection{Recognition of homotopy automorphism spaces}
In a general simplicial model category $\CCat$, the recognition result which we use in our computations reads as follows:

\begin{prop}\label{prop:hAut recognition}
Let $f: X\xrightarrow{\sim}Y$ be a weak equivalence in a simplicial model category $\CCat$, with $X$ cofibrant and $Y$ fibrant.
Let $G_{\bullet}$ be a simplicial monoid acting on $X$, in the sense that we have a map of simplicial monoids $G_{\bullet}\rightarrow\Map_{\CCat}(X,X)^{\htimes}$.
If the composition with $f$ induces a weak equivalence of simplicial sets
\begin{equation*}
G_{\bullet}\xrightarrow{\sim}\Map_{\CCat}(X,Y)^{\htimes},
\end{equation*}
then $G_{\bullet}$ is weakly equivalent to the homotopy automorphism space of $X$ as a simplicial monoid:
\begin{equation*}
G_{\bullet}\sim\Aut^h_{\CCat}(X).
\end{equation*}
\end{prop}

\begin{proof}
We follow \cite[\S II.2.2]{OperadHomotopyBook}. We first show the statement when $Y$ is both fibrant and cofibrant and $f$ is a cofibration.
We then consider the pullback diagram in simplicial sets
\begin{equation*}
\xymatrix{ H_{\bullet}\ar@{.>}[r]\ar@{.>}[d] & \Map_{\CCat}(Y,Y)^{\htimes}\sim\Aut^h_{\CCat}(X)\ar@{->>}[d]_{f_*}^{\sim} \\
G_{\bullet}\ar[r]^-{\sim} & \Map_{\CCat}(X,Y)^{\htimes} }.
\end{equation*}

The right-hand side vertical arrow of this diagram is an acyclic fibration by the pullback-corner property for the mapping space bifunctor $\Map_{\CCat}(-)$
and because the map $f_*: u\mapsto f u$ induces a bijection between the set of connected components $\Map_{\CCat}(Y,Y)_u\subset\Map_{\CCat}(Y,Y)$
associated to the homotopy classes of weak-equivalences $u: Y\xrightarrow{\sim}Y$
and the set of connected components $\Map_{\CCat}(X,Y)_v\subset\Map_{\CCat}(X,Y)$
associated to the homotopy classes of weak-equivalences $v: X\xrightarrow{\sim}Y$.
(Recall that we have an identity $[S,T]\simeq\pi_0\Map_{\CCat}(S,T)$, whenever $S$ is cofibrant and $T$ is fibrant,
where $[-,-]$ denotes the set of homotopy classes of morphisms in $\CCat$,
and observe that the morphism $v$ is a weak-equivalence if and only if $u$ is a weak-equivalence when we have the relation $v\sim f u$.)

We now use that the class of acyclic fibrations is preserved by pullbacks and the two-out-of-three axiom
to conclude that all morphisms in our pullback diagrams are weak equivalences.
We also check (as in \cite[Lemma II.2.2.3]{OperadHomotopyBook}) that $H_{\bullet}$ is equipped with the structure of a simplicial monoid
so that our weak equivalences $G_{\bullet}\xleftarrow{\sim} H_{\bullet}\xrightarrow{\sim}\Map_{\CCat}(Y,Y)^{\htimes}$
define weak equivalences of simplicial monoids.
We therefore get the result of the lemma in the special case where $Y$ is fibrant and cofibrant and $f$ is a cofibration.
	
In the general case, we choose a factorization of $f$ into an acyclic cofibration followed by a fibration, which is automatically acyclic by the two-out-of-three property:
\begin{equation*}
X\trivialrightarrowtail RY\trivialtwoheadrightarrow Y.
\end{equation*}
The object $RY$ is (automatically) fibrant and cofibrant.
But then we can factor the morphism $G_{\bullet}\xrightarrow{\sim}\Map_{\CCat}(X,Y)^{\htimes}$
given in our proposition
as:
\begin{equation*}
G_{\bullet}\rightarrow\Map_{\CCat}(X,RY)^{\htimes}\xrightarrow{\sim}\Map_{\CCat}(X,Y)^{\htimes},
\end{equation*}
and the second arrow of this composite is a weak equivalence.
The first arrow $G_{\bullet}\rightarrow\Map_{\CCat}(X,RY)^{\htimes}$
is also a weak equivalence by the two-out-of-three property. Hence, we can apply the special case of our claim to the object $RY$
and to the morphism $g: X\trivialrightarrowtail RY$
to conclude that we have a weak equivalence of simplicial monoids $G_{\bullet}\sim\Aut^h_{\CCat}(X)$
as stated in our Proposition.
\end{proof}

\begin{rem}
Note that this proof still goes through in our case $\CCat = \dgHOpc$ though we do not have a full model structure in then.
Hence, the result of this proposition is still valid for $\CCat = \dgHOpc$
and our mapping space construction.
\end{rem}

\subsection{The homotopy automorphism space of the rationalization of a good operad}
Recall that an operad $\POp\in\sSetOp$ is $\QQ$-good if the natural map
\begin{equation*}
H^*(\POp,\QQ)\rightarrow H^*(\POp^{\QQ},\QQ)
\end{equation*}
is an isomorphism (see~\cite[\S II.10.2.3]{OperadHomotopyBook}).
In this situation, we have the following statement:

\begin{prop}\label{prop:good and Aut}
If $\POp\in\sSetOp$ is a $\QQ$-good cofibrant operad, then we have a weak equivalence of simplicial monoids:
\begin{equation*}
\Aut^h_{\sSetOp}(\POp^{\QQ})\sim\Aut^h_{\dgHOpc}(\Omega_{\sharp}^*(\POp)),
\end{equation*}
where we consider the model of our operad in the category of dg Hopf cooperads $\Omega_{\sharp}^*(\POp)\in\dgHOpc$.
\end{prop}

\begin{proof}
This proposition essentially follows from generalities on simplicial model categories,
though we have to mind that the adjoint functors $(G,\Omega_{\sharp}^*)$
do not give rise to a simplicial Quillen adjunction in our case.
More precisely, for a cofibrant operad as in our statement $\POp\in\sSetOp$ and a cofibrant dg Hopf cooperad $\AOp\in\dgHOpc$,
we have a map of simplicial sets
\begin{equation*}
\Map_{\dgHOpc}(\AOp,\Omega_{\sharp}^*(\POp))\rightarrow\Map_{\sSetOp}(\POp,G(\AOp)),
\end{equation*}
which is a weak equivalence by Quillen adjunction, but not an isomorphism.
We pick a factorization $\Com^c\rightarrowtail\AOp\trivialtwoheadrightarrow\Omega_{\sharp}^*(\POp)$
of the initial map $\Com^c\rightarrow\Omega_{\sharp}^*(\POp)$
in $\dgHOpc$ (where $\Com^c$ denotes the commutative cooperad, which forms the initial object in the category of dg Hopf cooperads).
The dg Hopf cooperad $\Omega_{\sharp}^*(\POp)$ is fibrant as long as the operad $\POp$ is cofibrant since $\Omega_{\sharp}^*: \sSetOp^{op}\rightarrow\dgHOpc$
is a right Quillen functor. Hence, our dg Hopf cooperad $\AOp$ is both cofibrant and fibrant and defines a cofibrant-fibrant replacement
of the object $\Omega_{\sharp}^*(\POp)$ in $\dgHOpc$.
We similarly pick a cofibrant object $\ROp$ together with a an acyclic fibration $\IOp\rightarrowtail\ROp\trivialtwoheadrightarrow G(\AOp)$ in $\sSetOp$
(where $\IOp$ is the initial object of the category of operads in simplicial sets).
We again get that $\ROp$ is both cofibrant and fibrant and defines a cofibrant-fibrant replacement of the object $G(\AOp)$ in $\sSetOp$.
We use the notation $\phi: \ROp\xrightarrow{\sim} G(\AOp)$ for the acyclic fibration that we obtain from this factorization process.
We consider the associated adjoint morphism $\psi: \AOp\rightarrow\Omega_{\sharp}^*(\ROp)$
in $\dgHOpc$. This morphism $\psi$ is a weak equivalence because we assume that the operad $\POp$ is $\QQ$-good.
(Recall that $\POp^{\QQ} = LG(\Omega_{\sharp}^*(\POp)) = G(\AOp)$.)

We have (essentially by definition):
\begin{equation*}
\Aut^h(\POp^{\QQ})\sim\Map(\ROp,\ROp)^{\htimes}
\quad\text{and}\quad\Aut^h(\Omega_{\sharp}^*(\POp))\sim\Map(\AOp,\AOp)^{\htimes}.
\end{equation*}
We aim to prove that these objects are weakly equivalent as simplicial monoids.
We essentially follow the proof of \cite[Theorem II.2.2.5]{OperadHomotopyBook}.
We form the pullback diagram:
\begin{equation*}
\xymatrix{ U_{\bullet}\ar@{.>}[dd]\ar@{.>}[r] & \Map(\AOp,\AOp)^{\htimes}\ar[d]_-{\psi_*}^-{\sim} \\
& \Map(\AOp,\Omega_{\sharp}^*(\ROp))^{\htimes}\ar[d]^-{\sim} \\
\Map(\ROp,\ROp)^{\htimes}\ar@{->>}[r]^-{\phi_*}_-{\sim} & \Map(\ROp,G(\AOp))^{\htimes} },
\end{equation*}
where the lower horizontal arrow is an acyclic fibration by the pullback-corner property for the mapping space
of operads in simplicial sets $\Map(-) = \Map_{\sSetOp}(-)$ (see~\cite[Lemma 2.2.4]{OperadHomotopyBook}),
and because, as in the proof of Proposition~\ref{prop:hAut recognition},
the mapping $\phi_*: u\mapsto\phi u$
induces a bijection between the set of connected components $\Map(\ROp,\ROp)_u\subset\Map(\ROp,\ROp)$
associated to the homotopy classes of weak-equivalences $u: \ROp\xrightarrow{\sim}\ROp$
and the set of connected components $\Map(\ROp,G(\AOp))_v\subset\Map(\ROp,G(\AOp))$
associated to the homotopy classes of weak-equivalences $v: \ROp\xrightarrow{\sim}G(\AOp)$.
The upper horizontal arrow of our diagram is also an acyclic fibration (since the class of acyclic fibrations in a model category is stable under pullbacks)
and the left-hand vertical arrow of our diagram is a weak equivalence too
by the two-out-of-three property.
We again check as in \cite[Lemma II.2.2.3]{OperadHomotopyBook} that $U_{\bullet}$ is equipped with the structure of a simplicial monoid
so that our weak equivalences $\Map(\ROp,\ROp)^{\htimes}\xleftarrow{\sim} U_{\bullet}\xrightarrow{\sim}\Map(\AOp,\AOp)^{\htimes}$
define weak equivalences of simplicial monoids.
We therefore get the claim of the proposition.
\end{proof}

Recall that any operad defined by a collection of nilpotent spaces of finite $\QQ$-type
is $\QQ$-good (see \cite[\S II.7.3.11 and \S II.10.2.3]{OperadHomotopyBook}).
The little discs operads $\DOp_n$, which satisfy this assumption for $n\geq 3$, are $\QQ$-good therefore.
The operad $\DOp_2$ is not $\QQ$-good (by~\cite{QBadSpace}).
Nevertheless, the second author proved by hand computation that the result of the previous proposition is still valid
in this case (see~\cite[\S III.5.3]{OperadHomotopyBook}).
Hence, for the little discs operads, we have the following general statement:

\begin{prop}\label{prop:Aut E2}
We have a weak equivalence of simplicial monoids:
\begin{equation*}
\Aut^h(\DOp_n^{\QQ})\sim\Aut^h(\Omega_{\sharp}^*(\DOp_n))
\end{equation*}
for all $n\geq 2$.
\end{prop}

Recall that we set $\Omega_{\sharp}^*(\DOp_n) := \Omega_{\sharp}^*(\ROp_n)$ for the choice of a cofibrant operad in simplicial sets $\ROp_n$
such that $|\ROp_n|\sim\DOp_n$.

\section{The computation of homotopy automorphism spaces through deformation complexes}\label{sec:deformationcomplexes}
In this section, we explain our approach to computing homotopy automorphism spaces through deformation complexes.
The general idea is to represent the morphisms between two objects $\phi: \AOp\rightarrow\BOp$
as the solutions of a Maurer--Cartan equation in an $L_{\infty}$-algebra $\galg$,
whereas the homotopy relations between morphisms correspond to gauge equivalences in $\galg$.
This $L_{\infty}$-algebra can also be used to control the deformations of an object,
and therefore, we generally refer to such an $L_{\infty}$-algebra
as a deformation complex.

In~\cite[Section 3]{FTW}, a deformation complex is defined for dg Hopf cooperads of a particular form
(namely, in this reference, we assume that the source object
is given by the Chevalley--Eilenberg cochain complex
of an operad in Lie algebras).
The first purpose of this section is to explain the definition of a deformation complex for dg Hopf cooperads of a more general form,
which we consider in our computations of homotopy automorphism spaces.
In a first step, we revisit the definition of $L_{\infty}$-algebras by using functors of points with values in graded commutative algebras.
We treat this subject in the first subsection of this section.
We explain the definition of our deformation complex of dg Hopf cooperads afterwards, in the second subsection of the section.
We examine the definition of the differential of this deformation complex more thoroughly in a third subsection.

Then we explain the applications of deformation complexes for the study of homotopy automorphism spaces.
We make explicit a mapping between the deformation complex and the homotopy automorphism space of a dg Hopf cooperad in a fourth subsection
and we study a extension of this mapping to general mapping spaces
in a fifth subsection.
We eventually review a general homotopy invariance theorem for Maurer--Cartan spaces that we use in our computations.
We devote a sixth subsection to this survey.

\subsection{The schematic view of $L_{\infty}$-algebras}\label{sec:schematic Lie}
We explain the definition of $L_{\infty}$-algebras by using generalized point functors in this subsection.
We make this approach explicit in the next paragraph.
We revisit the definition of $L_{\infty}$-morphisms afterwards and we review the definition of the twisting of $L_{\infty}$-algebra structures
to complete our account.

\subsubsection{On the definition of $L_{\infty}$-algebras}\label{sec:Linfinity}
We consider an $L_{\infty}$-algebra $\galg$, whose structure is defined by $L_{\infty}$-operations of the form
\begin{equation*}
\mu_n: S_n(\galg[1])\rightarrow\galg[1],\quad n\geq 1,
\end{equation*}
where $S_n(-)$ denotes the homogeneous component of weight $n$ of the symmetric algebra $S(-)$
and $\galg[1]$ denotes the degree shift operation $\galg[1]^* = \galg^{*+1}$.
We assume that $\galg$ is equipped with a complete descending filtration
$\galg = F^1\galg\supset F^2\galg\supset\cdots\supset F^s\galg\supset\cdots$
such that
$\mu_n(F^{p_1}\galg,\dots,F^{p_n}\galg)\subset F^{p_1+\dots+p_n}\galg$,
for all $n\geq 1$. This condition ensures the convergence of the power series
\begin{equation}\label{equ:Mdef}
M(x) := \sum_{n=1}^{\infty}\frac{1}{n!}\mu_n(x,\dots,x)
\end{equation}
which collects our $L_{\infty}$-operations. The Maurer--Cartan equation in our $L_{\infty}$-algebra can then be expressed as $M(x) = 0$,
for any homogeneous element $x\in\galg^1$.

Let now $R$ be a graded commutative algebra. The completed tensor product $\galg\hat{\otimes}R$ is again an $L_{\infty}$-algebra
equipped with a compatible complete descending filtration.
By $R$-linear extension of the power series~\eqref{equ:Mdef}, we obtain the power series function
\begin{equation*}
M^R : (\galg\hat{\otimes}R)^1\rightarrow(\galg\hat{\otimes}R)^2
\end{equation*}
and $M^R(x) = 0$ represents the Maurer--Cartan equation in the extended $L_{\infty}$-algebra $\galg\hat{\otimes}R$.
The functions $M^R$ are functorial in $R$, in the sense that for a morphism of graded commutative algebras $\phi: R\rightarrow S$ we have a commutative diagram
\begin{equation*}
\begin{tikzcd}
(\galg\hat{\otimes}R)^1 \ar{r}{M^R} \ar{d}{\galg\hat{\otimes}\phi} & (\galg\hat{\otimes}R)^2 \ar{d}{\galg\hat{\otimes}\phi} \\
(\galg\hat{\otimes}S)^1 \ar{r}{M^S} & (\galg\hat{\otimes}S)^2
\end{tikzcd}.
\end{equation*}
The $L_{\infty}$-operations $\mu_n$ can also be recovered from the collection of power series $M^R$ by graded polarization.
For a collection of homogeneous elements in our $L_{\infty}$-algebra $x_1,\dots,x_n\in\galg$,
we consider the graded algebra $R = \QQ[\epsilon_1,\dots,\epsilon_n]$
generated by variables $\epsilon_i$ of degree $1-|x_i|$.
Then $\pm\mu_n(x_1,\dots,x_n)$ is the coefficient of the monomial $\epsilon_1\cdots\epsilon_n$ in $M^R(\sum_i\epsilon_i x_i)$.

In fact, one can easily check that there is a one-to-one correspondence between the collections of operations $\{\mu_n\}_n$
and functorial collections of power series $\{M^R\}_R$.
Furthermore, the $L_{\infty}$-equations for the operations $\mu_n$ can be fully expressed in terms of the function $M^R(x)$
and read:
\begin{equation}\label{equ:new Linfty}
D M^R(x)(M^R(x)) = 0,
\end{equation}
where the notation $D M^R(x)(h)$ refers to the differential of the function $M^R$ at position $x$
applied to $h$.
(This operation $F(x)\{G(x)\} = DF(x)(G(x))$ actually represents the pre-Lie product on the vector space of power series.)

\subsubsection{On the definition of $L_{\infty}$-morphisms}
We can also use the construction of the previous paragraph to encode $L_{\infty}$-morphisms.
Indeed, to an $L_{\infty}$-morphism between filtered complete $L_{\infty}$-algebras $U: \galg\rightarrow\halg$,
defined by operations
\begin{equation*}
U_n: S_n(\galg[1])\rightarrow\halg[1],\quad n\geq 1,
\end{equation*}
we associate the power series
\begin{equation}\label{equ:Udef}
U(x) = \sum_{n=1}^\infty\frac{1}{n!} U_n(x,\dots,x)
\end{equation}
and we consider again the function
\begin{equation*}
U^R: (\galg\hat{\otimes}R)^1\rightarrow(\halg\hat{\otimes}R)^1
\end{equation*}
defined by taking the obvious scalar extension of the map~\eqref{equ:Udef}.
The collection of power series $\{U^R\}_R$ is again functorial in $R$ and we can still retrieve all the operations $U_n$ from $U^R$
by graded polarization.
The $L_{\infty}$-equations for our morphism $U$
are equivalent to the relation:
\begin{equation}\label{equ:new Linfty mor}
D U^R(x)(M_{\galg}(x)) = M_{\halg}(U(x)),
\end{equation}
where we again use the notation $D U^R(x)(h)$ for the differential of the function $U^R$ at position $x$
applied to $h$.

In what follows, we use the representation of $L_{\infty}$-algebras and of $L_{\infty}$-morphisms
given in the previous paragraphs.
Equations~\eqref{equ:new Linfty} and~\eqref{equ:new Linfty mor}
have the advantage of being compact and completely sign-free
in comparison to the standard definition.
If no confusion can arise, then we drop the graded commutative algebra $R$ from our notation,
and we just set $M(x) = M^R(x)$ and $U(x) = U^R(x)$
for short.

\subsubsection{Maurer--Cartan elements and twisted structures}\label{sec:twisted_MC}
To an $L_{\infty}$-algebra $\galg$, we associate the set of Maurer--Cartan elements:
\begin{equation*}
\MC(\galg) := \{x\in\galg^1\mid M(x) = 0\}.
\end{equation*}
For $m\in\MC(\galg)$, we define the twisted $L_{\infty}$-algebra $\galg^m$ as the object defined by the same underlying graded vector space as $\galg$,
but with the $L_{\infty}$-operations such that:
\begin{equation*}
\mu_n^m(x_1,\dots,x_n) = \sum_{j\geq 0} \frac{1}{j!}\mu_{n+j}(\underbrace{m,\dots,m}_{j\times},x_1,\dots,x_n).
\end{equation*}
In our representation, this structure is associated to the power series $M_{\galg^m}^R$
such that:
\begin{equation}\label{equ:twisted M}
M_{\galg^m}^R(x) = M_{\galg}^R(x+m).
\end{equation}
For an $L_{\infty}$-morphism $U: \galg\rightarrow\halg$, we similarly have a twisted morphism $U_m: \galg^m\rightarrow\halg^{U(m)}$,
which, in our representation, is associated to the power series such that:
\begin{equation*}
(U^m)^R(x) = U^R(x+m).
\end{equation*}

\subsubsection{The pro-algebraic graded variety of Maurer--Cartan elements}\label{sec:schematic MC}
To an $L_{\infty}$-algebra $\galg$, we can also associate the functor $\MC(\galg\hat{\otimes}-): \g\ComCat\rightarrow\Set$
such that:
\begin{equation*}
\MC(\galg\hat{\otimes}R) := \{x\in(\galg\hat{\otimes}R)^1\mid M^R(x) = 0\},
\end{equation*}
where we consider the set of Maurer--Cartan elements in the $R$-linear extension of our $L_{\infty}$-algebra.
This set $\MC(\galg\hat{\otimes}R)$ can be regarded as the set of $R$-points
of a pro-affine graded variety naturally associated to $\galg$ (since the equation $M^R(x) = 0$
is polynomial modulo $F^s\galg$, for any $s\geq 1$).
The representation of the previous paragraphs reflects this correspondence between $L_{\infty}$-algebras and pro-affine graded varieties.

In what follows, we also consider an extension of the definition of the Maurer--Cartan set $\MC(\galg\hat{\otimes}R)$
in the case where $R$ is a dg commutative algebra.
In this case, we have to add the internal differential of our object $\delta: R\rightarrow R$
to the operation $\mu_1: \galg\rightarrow\galg$ of our $L_{\infty}$-algebra structure
when we form the power series $M^R(x)$, for $x\in(\galg\hat{\otimes}R)^1$.
In particular, we set $\MC_{\bullet}(\galg) = \MC(\galg\hat{\otimes}\Omega^*(\Delta^{\bullet}))$,
and we refer to this simplicial set $\MC_{\bullet}(\galg)$
as the Maurer--Cartan space associated to the $L_{\infty}$-algebra $\galg$.

\subsubsection*{Remark}
There is a generalization of $L_{\infty}$-algebras, called curved $L_{\infty}$-algebras, in which the presence
of a null-ary operation $\mu_0\in \galg^2$ (referred to as the curvature)
is allowed.
One can readily extend our formalism to curved $L_{\infty}$-algebras and $L_{\infty}$-morphisms.
For this purpose, we just let the series~\eqref{equ:Mdef} and~\eqref{equ:Udef} start at $n=0$ instead of $n=1$.

%\subsubsection*{Remark}
%In the above constructions it is possible to weaken the conditions on the filtration on $\galg$
%in order to deal with complete descending filtration
%with a zero term:
%\begin{equation*}
%\galg = F^0\galg\supset F^1\galg\supset\cdots\supset F^s\galg\supset\cdots.
%\end{equation*}
%In this case, we need an extra condition on the zero term of our filtration in order to ensure the convergence of our power series.
%For instance, we can assume that we have a splitting of graded vector spaces $\galg\simeq\galg/F^1\galg\oplus F^1\galg$
%such that all but finitely many of the operations $\mu_n$ factor through $S_n(F^1\galg[1])\subset S_n(\galg[1])$.

\subsection{Morphism sets and deformation complexes of dg Hopf cooperads}\label{sec:defcomplex}
The morphisms sets $\Mor_{\dgHOp^c_R}(\AOp\otimes R,\BOp\otimes R)$ associated to dg Hopf cooperads $\AOp$ and $\BOp$
define a functor from the category of graded commutative algebras
to the category of sets.
The goal of this subsection is to prove that, under suitable conditions on $\AOp$ and $\BOp$,
this functor can be represented by an $L_{\infty}$-algebra
in the sense of the previous subsection.
To be explicit, we make the following assumptions
all along this section:
\begin{enumerate}
\item\label{sec:defcomplex:sourcecondition}
We assume that the dg Hopf cooperad $\AOp$ forms a quasi-free dg commutative algebra arity-wise, so that we have an identity
\begin{equation*}
\AOp(r)=(S(\MOp(r)),\partial)
\end{equation*}
for each arity $r>0$, for a given generating graded symmetric sequence $\MOp = \MOp(1),\MOp(2),\dots$.
In addition, we assume that $\MOp$ is equipped with an exhaustive filtration
\begin{equation*}
0 = F^0\MOp\subset F^1\MOp\subset F^2\MOp\subset\cdots\subset F^s\MOp\subset\cdots\subset\MOp,
\end{equation*}
such that the differential $\partial: S(\MOp(r))\rightarrow S(\MOp(r))$ carries $F^s\MOp(r)$ into $\bar{S}(F^s\MOp(r))$,
where $\bar{S}(-)$ denotes the augmentation ideal of the symmetric algebra $S(-)$,
and the composition coproducts of the cooperad structure $\circ^*: \AOp(k+l-1)\rightarrow\AOp(k)\otimes\AOp(l)$
carry the graded vector space $F^s\MOp(k+l-1)\subset\AOp(k+l-1)$
into $\sum_{p+q=s} F^p S(\MOp(k))\otimes F^q S(\MOp(l))$,
where we take:
\begin{equation*}
\qquad\qquad F^s S(\MOp(r)) = \sum_{\substack{p_1+\cdots+p_m\leq s\\m\geq 0}}\im\Bigl(F^{p_1}\MOp(r)\otimes\dots\otimes F^{p_m}\MOp(r)\rightarrow S(\MOp(r))\Bigr),
\end{equation*}
for each $r>0$.
We also assume that the graded vector space $E_0^s\MOp(r) = F^s\MOp(r)/F^{s-1}\MOp(r)$ is finite dimensional for each arity $r>0$
and for each filtration degree $s\geq 0$.
\item\label{sec:defcomplex:targetcondition}
We assume that the dg Hopf cooperad $\BOp$ is quasi-cofree as a dg cooperad,
so that we have an identity $\BOp = (\FreeOp^c(\NOp),\partial)$,
for some graded symmetric sequence $\NOp$.
\end{enumerate}
Note that the assumptions on the dg Hopf cooperad $\AOp$ imply that this object is equipped with an augmentation $\epsilon: \AOp\rightarrow\Com^c$ induced by the map $\MOp\rightarrow 0$,
where $\Com^c$ denotes again the commutative cooperad, defined by $\Com^c(r) = \kk$, for each arity $r>0$.

Then we have the following result:

\begin{prop}\label{prop:pre_defcx}
Let $R$ be a graded commutative algebra. Let $\g\Hopf\Op^c_R$ denote the category of graded Hopf cooperads over $R$.
Let $\g\Seq_R$ denote the category of graded symmetric sequences over $R$.
If the dg Hopf cooperads $\AOp$ and $\BOp$ satisfy the above conditions (\ref{sec:defcomplex:sourcecondition}-\ref{sec:defcomplex:targetcondition}),
then we have a bijection
\begin{equation*}
\Mor_{\g\Hopf\Op^c_R}(\AOp^{\flat}\otimes R,\BOp^{\flat}\otimes R)\xrightarrow{\simeq}\Mor_{\g\Seq^c_R}(\MOp\otimes R,\NOp\otimes R),
\end{equation*}
given by the mapping $\psi\mapsto\pi\psi\iota$
such that:
\begin{itemize}
\item
we take the composition of morphisms with the inclusion of generators $\iota: \MOp\hookrightarrow S(\MOp)$
on the source $\AOp^{\flat}\otimes R = S(\MOp)\otimes R$
\item
and the composition with the projection onto cogenerators $\pi: \FreeOp^c(\NOp)\rightarrow\NOp$
on the target $\BOp^{\flat}\otimes R = \FreeOp^c(\NOp)\otimes R$.
\end{itemize}
\end{prop}

We remind the reader that $(-)^{\flat}$ refers to the forgetting of differentials in any category of dg objects.

\begin{proof}
Note that we completely forget about the differentials of our dg Hopf cooperads in the construction of our statement.

Let $f\in\Mor_{\g\Seq^c_R}(\MOp\otimes R,\NOp\otimes R)$ be given.
We construct the corresponding morphism $\psi_f\in\Mor_{\g\Hopf\Op^c_R}(\AOp^{\flat}\otimes R,\BOp^{\flat}\otimes R)$.
We prove along the way that this morphism is unique.
We proceed by induction, using the given filtration on $\MOp$ and the induced filtration on $\AOp = S(\MOp)$.
To start the induction, we note that the map
\begin{equation*}
\psi_f: F^0 S(\MOp) = \Com^c\rightarrow\BOp
\end{equation*}
is uniquely determined, since algebra units have to be mapped to algebra units in a dg Hopf cooperad.
We now assume that $\psi_f$ is defined on $F^{s-1}\AOp$ and we aim to extend our map to $F^s\AOp$.
We have $F^s S(\MOp) = F^s\MOp\oplus\SOp$, where $\SOp$ consists of products of elements of lower filtration $F^p\MOp$, $p<s$, within the symmetric algebra $S(\MOp)$.
Hence $\psi_f$ is determined on the summand $\SOp$ by compatibility of the morphism with the products.
We use that $\BOp$ is quasi-cofree to determine our map $\psi_f$ on the summand $F^s\MOp$ from $f$.

To be more explicit, by the construction of cofree cooperads, we have
\begin{gather*}
\BOp^{\flat}(r) = \bigoplus_{T\in\Tree(r)}\FreeOp_T^c(\NOp)(r),
\intertext{where the expansion runs over the set of trees with $r$ ingoing leaves $T\in\Tree(r)$ and}
\FreeOp_T^c(\NOp)(r) = \bigotimes_{v\in VT}\NOp(r_v)
\end{gather*}
denotes the tensor product of the graded symmetric sequence $\NOp$
over the set of vertices of the tree $T$.
In this expression, the arity of a factor $r_v$ is given by the number of ingoing edges of the vertex $v$.
For $x\in F^s\MOp(r)$, we consider the reduced tree-wise coproducts of our element
\begin{equation*}
\Delta_T(x) = \sum_{(x)}\Bigl(\bigotimes_{v\in VT} x_v\Bigr)\in\FreeOp_T^c(\bar{\AOp})(r),
\end{equation*}
where $\bar{\AOp}$ denotes the coaugmentation coideal of the dg cooperad $\AOp$.
By assumption on our filtration, the factors of this tree-wise coproducts satisfy $x_v\in F^{s_v} S(\MOp(r_v))$ and we have $\sum_v s_v = s$.
If we have $s_v = s$ for some vertex $v$, then we necessarily have $s_w = 0$ and hence $F^{s_w} S(\MOp(r_w)) = F^0 S(\MOp(r_w)) = \Com^c(r_w)$
for the factors associated to the other vertices $w$,
but as we withdraw the term $\Com^c(1) = \kk 1$ from the coaugmentation coideal $\bar{\AOp}$,
this relation implies $r_w>1$ and hence we have $r_v<r$
in this case.
Thus, as soon as the tree $T$ has at least two vertices, the reduced tree-wise coproduct $\Delta_T(x) = \sum_{(x)} x_v$
consists either of factors of lower filtration $x_v\in F^{s_v} S(\MOp(r_v))$, $s_v<s$,
or of factors of lower arity $x_v\in F^{s_v} S(\MOp(r_v))$, $r_v<r$.
In this context, we can use the induction to determine the image of our element in $\FreeOp_T^c(\NOp)(r)$
by the formula
\begin{equation*}
\psi_f(x)_T = \sum_{(x)}\Bigl(\bigotimes_{v\in VT}\pi\psi_f(x_v)\Bigr)
\end{equation*}
when $T$ has at least two vertices,
where we take the image of $x_v\in F^{s_v} S(\MOp(r_v))$ under the composite of the recursively defined map $\psi_f: F^{s_v} S(\MOp(r_v))\rightarrow\FreeOp^c(\NOp)(r_v)$
with the canonical projection $\pi: \FreeOp^c(\NOp)(r_v)\rightarrow\NOp(r_v)$,
and we merely take:
\begin{equation*}
\psi_f(x)_Y = f(x)
\end{equation*}
when $T = Y$ is a corolla (a tree with a single vertex).

This recursively defined morphism $\psi_f$ automatically preserves the products by construction. We easily check that $\psi_f$ preserves the composition coproducts too (use the compatibility between the products and the composition coproducts in Hopf cooperads,
together with the coassociativity relation of tree-wise coproducts
in cooperads).
Note that our construction is forced by the constraint $f = \pi \psi_f\iota$ and by these compatibility relations.
Hence, our morphism $\psi_f$ is uniquely determined by the map $f$.
\end{proof}

We can interpret the result of this proposition in terms of the scheme theoretic picture of the previous section.
For this purpose, we observe that $S(R) = \Mor_{\g\Hopf\Op^c_R}(\AOp^{\flat}\otimes R,\BOp^{\flat}\otimes R)$
can be regarded as the set of $R$-points
of a pro-algebraic subvariety of the affine space $V(R) = \Mor_{\g\Seq^c_R}(\AOp^{\flat}\otimes R,\BOp^{\flat}\otimes R)$.
We use that the vector space $W(R) = \Mor_{\g\Seq^c_R}(\MOp\otimes R,\NOp\otimes R)$
forms an affine pro-algebraic variety too.
We make the structure of these pro-algebraic varieties explicit in the next proposition:

\begin{prop}
For any graded commutative algebra $R$, we have identities
\begin{align*}
V(R) & = \Mor_{\g\Seq^c_R}(\AOp^{\flat}\otimes R,\BOp^{\flat}\otimes R) = (\Hom_{\g\Seq^c}(\AOp^{\flat},\BOp^{\flat})\hat{\otimes}R)^0\\
W(R) & = \Mor_{\g\Seq^c_R}(\MOp\otimes R,\NOp\otimes R) = (\Hom_{\g\Seq^c}(\MOp,\NOp)\hat{\otimes}R)^0,
\end{align*}
where $T = \Hom_{\g\Seq^c}(-,-)$ denotes the natural graded hom-object bifunctor associated to the category of graded symmetric sequences.
These graded vector spaces $T = \Hom_{\g\Seq^c}(\SOp,\TOp)$, where $(\SOp,\TOp) = (\MOp,\NOp)$ or $(\SOp,\TOp) = (\AOp^{\flat},\BOp^{\flat})$,
are equipped with the complete descending filtration
such that:
\begin{equation*}
F^s T = \ker\Bigl(\Hom_{\g\Seq^c}(\SOp,\TOp)\rightarrow\prod_{p+q-1=s}\Hom_{\g\Seq^{\leq q}}(F^p\SOp,\TOp)\Bigr),\quad s\geq 0,
\end{equation*}
where we consider the filtration $F^p\SOp$ of the symmetric sequences $\SOp = \MOp$ or $\SOp = S(\MOp)$
and $\g\Seq^{\leq q}$ denotes the category of $q$-truncated symmetric sequences,
which is formed by forgetting about the components of arity $r>q$
in the definition of a symmetric sequence.

The object $S(R) = \Mor_{\g\Hopf\Op^c_R}(\AOp^{\flat}\otimes R,\BOp^{\flat}\otimes R)$ forms a pro-algebraic subvariety of $V(R) = T\hat{\otimes}R$,
in the sense that $S(R)$ is defined by polynomial equations in $T/F^s T$,
for each filtration degree $s\geq 0$.
\end{prop}

\begin{proof}
Let $(\SOp,\TOp)$ be as in the proposition.
The identity $\Mor_{\g\Seq^c_R}(\SOp\otimes R,\TOp\otimes R) = (\Hom_{\g\Seq^c}(\SOp,\TOp)\hat{\otimes}R)^0$
follows from the relation
\begin{equation*}
\Mor_{\g\Seq^{\leq q}_R}(F^p\SOp\otimes R,\TOp\otimes R)
\simeq\Mor_{\g\Seq^{\leq q}}(F^p\SOp,\TOp\otimes R)
\simeq\Mor_{\g\Seq^{\leq q}}(F^p\SOp,\TOp)\otimes R,
\end{equation*}
where, to get the second isomorphism, we use that $F^p\SOp(r)$ is finite dimensional in each arity $r>0$
for both $\SOp(r) = \MOp(r)$ and $\SOp(r) = S(\MOp(r))$.
The claim that $S(R)$ forms a pro-algebraic subvariety of $V(R)$ follows from the fact that the morphisms of graded Hopf cooperads
are defined by polynomial relations.
\end{proof}

We can now make the following observation:

\begin{prop}
The bijection
\begin{multline*}
\Mor_{\g\Seq^c_R}(\MOp\otimes R,\NOp\otimes R)\xrightarrow[\simeq]{\Phi}\Mor_{\g\Hopf\Op^c_R}(\AOp^{\flat}\otimes R,\BOp^{\flat}\otimes R)\\
\subset\Mor_{\g\Seq^c_R}(\AOp^{\flat}\otimes R,\BOp^{\flat}\otimes R),
\end{multline*}
provided by the correspondence of Proposition~\ref{prop:pre_defcx},
is given by a power series function:
\begin{equation*}
\Phi(x) = \sum_{n\geq 0}\Phi_n(x,\dots,x),
\end{equation*}
defined by a scalar extension to $R$ of operations involving the unit, the products and the composition coproducts of our Hopf cooperads,
independently of the graded commutative algebra $R$.
\end{prop}

\begin{proof}
This proposition follows from an immediate inspection of the construction of our map $f\mapsto\psi_f$
in the proof of Proposition~\ref{prop:pre_defcx}.
\end{proof}

We use the above proposition in the following statement:

\begin{thm}\label{thm:defcx}
The graded vector space $\galg = \Hom_{\g\Seq^c}(\MOp,\NOp)$ inherits an $L_{\infty}$-algebra structure
such that:
\begin{equation*}
\MC(\Hom_{\g\Seq^c}(\MOp,\NOp)\hat{\otimes}R) = \Mor_{\dg\Hopf\Op^c_R}(\AOp\otimes R,\BOp\otimes R),
\end{equation*}
where we consider the morphisms between the objects $\AOp\otimes R$ and $\BOp\otimes R$
in the category of dg Hopf cooperads.

This $L_{\infty}$-algebra structure is associated to the power series function such that:
\begin{equation*}
M(x) := \pi(\partial_{\BOp}\Phi(x)-\Phi(x)\partial_{\AOp})\iota,
\end{equation*}
for any $x\in\Mor_{\g\Seq^c_R}(\MOp\otimes R,\NOp\otimes R)$, where $\partial = \partial_{\AOp}$ and $\partial = \partial_{\BOp}$
are the maps induced by the differential of the dg Hopf cooperads $\AOp$ and $\BOp$
inside the graded vector space $T = \Hom_{\g\Seq^c_R}(\AOp^{\flat}\otimes R,\BOp^{\flat}\otimes R)$.
\end{thm}

\begin{proof}
From the result of Proposition~\ref{prop:pre_defcx}, we also obtain that a biderivation $\theta: \AOp^{\flat}\rightarrow\BOp^{\flat}$
of a morphism of graded Hopf cooperads $\psi = \psi_f: \AOp^{\flat}\rightarrow\BOp^{\flat}$
is uniquely determined by a map such that:
\begin{equation*}
h := \pi\theta\iota\in\Hom_{\g\Seq^c}(\MOp,\NOp).
\end{equation*}
Indeed, giving such a biderivation amounts to giving a morphism of graded Hopf cooperads
of the form
$\psi+\epsilon\theta: \AOp^{\flat}\otimes R\rightarrow\BOp^{\flat}\otimes R$,
where we take $R = \kk[\epsilon]/(\epsilon^2)$.
To retrieve the biderivation $\theta = \theta_h$ from the corresponding map $h = \pi\theta\iota$,
we use the formula:
\begin{equation*}
\psi + \epsilon\theta = \Phi(x+\epsilon h) = \Phi(x) + \epsilon D\Phi(x)(h)\Rightarrow\theta = D\Phi(x)(h),
\end{equation*}
where $D\Phi$ is again the differential and $x$ is such that $\Phi(x) = \psi$.

Note that $\partial(\Phi(x)) = \partial_{\BOp}\Phi(x)-\Phi(x)\partial_{\AOp}$
is a biderivation of the map $\psi = \Phi(x)$.
Hence, from the above formula, we get the identity
\begin{equation*}
\partial_{\BOp}\Phi(x)-\Phi(x)\partial_{\AOp} = D\Phi(x)(M(x))
\end{equation*}
where $M(x)$ is the power series function defined in our theorem $M(x) = \pi(\delta_{\BOp}\Phi(x)- \Phi(x)\delta_{\AOp})\iota$.
From this identity, we obtain the vanishing relation:
\begin{align*}
DM(X)(M(X)) & = \pi(\partial_{\BOp}D\Phi(x)(M(x))- D\Phi(x)(M(x))\partial_{\AOp})\iota\\
& = \pi(\partial_{\BOp}(\partial_{\BOp}\Phi(x)-\Phi(x)\partial_{\AOp}) - (\partial_{\BOp}\Phi(x)-\Phi(x)\partial_{\AOp})\partial_{\AOp})\iota = 0,
\end{align*}
which proves that our power series does define an $L_{\infty}$-structure.

Then we just use that $M(x) = \pi\partial(\Phi(x))\iota$ fully determines the biderivation $\partial(\Phi(x)) = \partial_{\BOp}\Phi(x)-\Phi(x)\partial_{\AOp}$
to conclude that the Maurer--Cartan equation $M(x) = 0$ is equivalent to the relation $\partial_{\BOp}\Phi(x) - \Phi(x)\partial_{\AOp}$
in $\Mor_{\g\Seq^c_R}(\AOp^{\flat}\otimes R,\BOp^{\flat}\otimes R)$
and hence, to conclude that the set of solution of the Maurer--Cartan equation $M(x) = 0$
in $\Hom_{\g\Seq^c}(\MOp,\NOp)\hat{\otimes}R = \Mor_{\g\Seq^c_R}(\MOp\otimes R,\NOp\otimes R)$
is in bijection with the set of morphisms of graded Hopf cooperads $\psi = \Phi(x)$
that preserve the differentials.
The identity of the theorem $\MC(\galg\hat{\otimes}R) = \Mor_{\dg\Hopf\Op^c_R}(\AOp\otimes R,\BOp\otimes R)$
follows.
\end{proof}

\begin{defn}
We adopt the notation $\Def(\AOp,\BOp) = \Hom_{\g\Seq^c}(\MOp,\NOp)$ for the $L_{\infty}$-algebra of Theorem~\ref{thm:defcx}.
We may note that this $L_{\infty}$-algebra depends on structures, such as the choice of the graded symmetric sequence $\MOp$ and $\NOp$,
which we forget in this notation.
The important fact for us is the isomorphism of Theorem~\ref{thm:defcx}:
\begin{equation*}
\Mor_{\dg\Hopf\Op^c_R}(\AOp\otimes R,\BOp\otimes R)\simeq\MC(\Def(\AOp,\BOp)\hat{\otimes}R),
\end{equation*}
which holds for every graded commutative algebra $R$. This relation actually extends to the case where $R$ is a dg commutative algebra.
We then add the internal differential of our object $\delta: R\rightarrow R$ to the operation $\mu_1: \galg\rightarrow\galg$
of the $L_{\infty}$-algebra $\galg = \Def(\AOp,\BOp)$
when we form our set of Maurer--Cartan elements (see Paragraph~\ref{sec:schematic MC}).

Let $x\in\Def(\AOp,\BOp)$ be the Maurer--Cartan element that corresponds to a morphism of dg Hopf cooperads $\psi: \AOp\rightarrow\BOp$.
Then we use the notation
\begin{equation*}
\Def(\AOp\xrightarrow{\psi}\BOp) := \Def(\AOp,\BOp)^x
\end{equation*}
for the twisted $L_{\infty}$-algebra $\Def(\AOp,\BOp)^x$ (see Paragraph~\ref{sec:twisted_MC}).
This object can be identified with the complex of biderivations of the morphism $\psi$ in \cite[Section 3]{FTW}:
\begin{equation*}
\Def(\AOp\xrightarrow{\psi}\BOp)\simeq\BiDer(\AOp\xrightarrow{\psi}\BOp).
\end{equation*}
\end{defn}

\subsection{The differential of the deformation complex}\label{sec:explicit differential}
The $L_{\infty}$-algebra structure on the deformation complexes considered in the previous subsection is combinatorially relatively complicated.
Nevertheless, under some simplifying assumption that will be satisfied in our examples,
one can at least derive relatively explicit formulas for the differential $\partial = \mu_1$
of this $L_{\infty}$-algebra.

We fix a dg Hopf cooperad $\AOp = (S(\MOp),\partial)$ as in the last section and we consider a morphism of dg Hopf cooperads $\psi: \AOp\rightarrow\BOp$
with values in another non-negatively graded dg Hopf cooperad $\BOp$.
We then take a fibrant resolution of $\BOp$ given by the Hopf cooperadic $W$ construction $W^c(\BOp)$ (see \cite[Section 5.2]{FTW}).
Recall that $W^c(\BOp)$ is quasi-cofree as a cooperad.
In \cite[Section 5.2]{FTW}, the notation $\mathring{W}^c(\BOp)$ is used for a dg symmetric sequence
such that $W^c(\BOp)^{\flat} = \FreeOp^c(\Sigma\mathring{W}^c(\BOp)^{\flat})$.
We can actually identify this cogenerating dg symmetric sequence $\mathring{W}^c(\BOp)$
with the augmentation ideal of an augmented dg operad $\mathring{W}^c(\BOp)_1$
such that $W^c(\BOp) = B(\mathring{W}^c(\BOp)_1)$,
where we consider the operadic bar construction (see \cite[Lemma 5.4]{FTW}).
For simplicity, we generally omit the suspension in the expression of the cogenerating dg symmetric sequence
of the dg Hopf cooperad $W^c(\BOp)$
in what follows.
We consider the deformation complex such that:
\begin{equation*}
\Def(\AOp\xrightarrow{\psi}W^c(\BOp)) = (\Hom_{\g\Seq^c}(\MOp,\mathring{W}^c(\BOp)^{\flat}),\partial),
\end{equation*}
where we denote the composition of the map $\psi$ with the inclusion $\BOp\rightarrow W^c(\BOp)$
by $\psi$ as well, abusing notation.
The goal of this subsection is to make explicit the differential of this deformation complex.

Recall that the elements of this deformation complex $h\in\Hom_{\g\Seq^c}(\MOp,\mathring{W}^c(\BOp)^{\flat})$
correspond to biderivations $\theta = \theta_h$
of the morphism of dg Hopf cooperads $\psi: \AOp\rightarrow W^c(\BOp)$.
We have by definition:
\begin{equation*}
\partial(h) = \pi(\partial_{W^c(\BOp)}\theta_h - \pm\theta_h\partial_{\AOp})\iota
\end{equation*}
where $\pi: W^c(\BOp)^{\flat}\rightarrow\mathring{W}^c(\BOp)^{\flat}$ is the projection onto the cogenerators of the dg Hopf cooperad $W^c(\BOp)$
while $\iota: \MOp\rightarrow\AOp^{\flat}$ is the inclusion of the generators of the dg Hopf cooperad $\AOp$.
We set:
\begin{equation*}
\partial'(h) = \pi\partial_{W^c(\BOp)}\theta_h\iota
\quad\text{and}
\quad\partial''(h) = -\pm\pi\theta_h\partial_{\AOp}\iota
\end{equation*}
so that $\partial(h) = \partial'(h)+\partial''(h)$.
In what follows, we refer to $\partial'(h)$ as the operadic part of the differential $\partial(h)$
and to $\partial''(h)$ as the algebra part of the differential.

We get the following statement:

\begin{prop}\label{prop:explicit differential}
Let $x\in\MOp(r)$. Let $\circ_i^*(x) = \sum_{\circ_i^*(x)} x_1'\cdots x_m'\otimes x_1''\cdots x_n''$
denote the expansion of the composition coproducts
of this element in $\AOp(k)\otimes\AOp(l) = S(\MOp(k))\otimes S(\MOp(l))$.
We have the formula:
\begin{multline}\label{equ:dprime}
\partial'(h)(x) = \partial_{\mathring{W}^c(\BOp)} h(x)\\
+ \sum_{\circ_i^*(x)}\Bigl(\sum_{j=1}^m\pm(\psi(x_1')\cdots h(x_j')\cdots\psi(x_m'))\circ_i\pi\psi(x_1''\cdots x_n'')\Bigr.\\
+ \Bigl.\sum_{j=1}^l\pm\pi\psi(x_1'\cdots x_m')\circ_i(\psi(x_1'')\cdots h(x_j'')\cdots\psi(x_n''))\Bigr),
\end{multline}
where $\partial_{\mathring{W}^c(\BOp)}$ denotes the internal differential of the cogenerating dg symmetric sequence of the dg Hopf cooperad $W^c(\BOp)$,
while the external sum runs over all the composition coproducts or our element $x$
in the dg Hopf cooperad $\AOp$.
We then use that the image of our morphism $\psi$ lies in $\BOp\subset W^c(\BOp)$ and that the objects $\mathring{W}^c(\BOp)(r)$
inherit both the structure of a bimodule over the dg commutative algebras $\BOp(r)$
and operadic composition products $\circ_i: \mathring{W}^c(\BOp)(k)\otimes\mathring{W}^c(\BOp)(l)\rightarrow\mathring{W}^c(\BOp)(k+l-1)$
(we refer to~\cite[Section 5.2]{FTW}).

Let similarly $\partial_{\AOp}(x) = \sum x_1\cdots x_m$ denote the differential of our element $x\in\MOp(r)$
in the dg commutative algebra $\AOp(r) = (S(\MOp(r)),\partial_{\AOp})$.
We have the formula:
\begin{equation}\label{equ:dprimeprime}
\partial''(h)(x) = \sum\Bigl(\sum_{j=1}^m\pm\psi(x_1)\cdots h(x_j)\cdots\psi(x_m)\Bigr),
\end{equation}
where we again use the action of the dg commutative algebra $\BOp(r)$
on the graded vector space $\mathring{W}^c(\BOp)(r)$.
\end{prop}

\begin{proof}
The main combinatorial difficulty is computing the biderivation $\theta = \theta_h$ such that $\pi\theta\iota = h$.
This can be done by following the proof of Proposition \ref{prop:pre_defcx},
albeit with the simplification that we are working with biderivations
as opposed to Hopf cooperad morphisms.

The key observation that makes $\theta$ explicitly computable is that our morphism $\psi$ factors through the map $\BOp\rightarrow W^c(\BOp)$.
Equivalently, the elements $\psi(\alpha)\in W^c(\BOp)$
correspond to decorated trees whose edges are labelled by the trivial differential form $1\in\Omega^0(\Delta^1)$
(we refer to the definition of the object $W^c(\BOp)$ in~\cite[Section 5.2]{FTW}).
This observation implies that we have the commutation relation with respect to the product in $W^c(\BOp)(r)$:
\begin{equation}\label{equ:prod_simple}
\pi(\psi(\alpha)\cdot\omega) = \psi(\alpha)\cdot\pi(\omega),
\end{equation}
for all $\alpha\in\AOp(r)$ and $\omega\in W^c(\BOp)(r)$,
where on the right-hand side, we consider the action of the dg commutative algebra $\BOp(r)$ on $\mathring{W}^c(\BOp)(r)$.

In a first step, we determine the composite map $\pi\theta$ from $h = \pi\theta\iota$.
For $x_1,\dots,x_m\in\MOp(r)$, we merely get:
\begin{equation}\label{equ:proj_derivation}
(\pi\theta)(x_1\cdots x_m)
\begin{aligned}[t]
& = \pi(\sum_i\pm \psi(x_1)\cdots\theta(x_i)\cdots\psi(x_m))\\
& = \sum_i\pm\psi(x_1)\cdots\pi\theta(x_i)\cdots\psi(x_m)\\
& = \sum_i \pm \psi(x_1)\cdots h(x_i)\cdots\psi(x_m)
\end{aligned}
\end{equation}
using that $\theta$ is a derivation with respect to the product and relation~\eqref{equ:prod_simple}.

Then, for $\alpha\in\AOp(r)$, we can compute the components $\theta_T(\alpha)\in\FreeOp_T^c(\mathring{W}^c(\BOp))$
of the element $\theta(\alpha)\in W^c(\BOp)$
in the summands of the cofree cooperad $W^c(\BOp)^{\flat} = \FreeOp^c(\mathring{W}^c(\BOp))$
by:
\begin{equation}\label{equ:expression_coderivation}
\theta_T(\alpha) = \sum_i\pm(\pi\psi\otimes\cdots\otimes\pi\theta\otimes\cdots\otimes\pi\psi)\Delta_T(\alpha),
\end{equation}
where we take the reduced treewise composition coproduct of our element $\Delta_T(\alpha)\in\FreeOp_T^c(\bar{\AOp})(r)$
(as in the proof of Proposition~\ref{prop:pre_defcx}),
we apply the map $\pi\theta$ to one factor of this treewise tensor product
and the map $\pi\psi$ to the other factors.

For an element $\alpha = x\in\MOp(r)$ such that $\partial_{\AOp}(x) = \sum x_1\cdots x_m$, the application of the formula~\eqref{equ:proj_derivation}
to $\partial''(h)(x) = \pi\theta(\partial_{\AOp}x)$
gives the formula of the proposition, \emph{i.e.}~\eqref{equ:dprimeprime}.
For an element $\alpha = x\in\MOp(r)$ such that $\circ_i^*(x) = \sum_{\circ_i^*(x)} x_1'\cdots x_m'\otimes x_1''\cdots x_n''$,
we use that $W^c(\BOp)$
is identified with a bar construction on $\mathring{W}^c(\BOp)$,
to get an identity
\begin{multline}
\partial'(h)(x) = \partial_{\mathring{W}^c(\BOp)}(x) + \sum_{\circ_i^*(x)}\Bigl(\pm\pi\theta(x_1'\cdots x_m')\circ_i\pi\psi(x_1''\cdots x_n'')\Bigr.\\
+ \Bigl.\pm\pi\psi(x_1'\cdots x_m')\circ_i\pi\theta(x_1''\cdots x_n'')\Bigr)
\end{multline}
and we apply the derivation formula~\eqref{equ:proj_derivation} again to obtain the expression of the proposition, \emph{i.e.}~\eqref{equ:dprime}.
\end{proof}

Note that the deformation complexes discussed in this subsection are just variants of the deformation bicomplexes of dg Hopf cooperads
described in \cite[Section 1.3-1.4]{FW}.

\subsection{Construction: dg Lie actions and automorphisms}
\label{sec:dgLieactions}
We now assume that $\galg$ is a pro-nilpotent dg Lie algebra, equipped with a complete filtration
\begin{equation*}
\galg = F^1\galg\supset F^2\galg\supset\cdots\supset F^s\galg\supset\cdots
\end{equation*}
such that
\begin{equation*}
[F^s\galg,F^t\galg]\subset F^{s+t}\galg,
\end{equation*}
for any $s,t\geq 1$.
To such a dg Lie algebra $\galg$, we associate the discrete group of degree zero cocycles:
\begin{equation*}
Z^0(\galg) = \{\xi\in\galg^0|\delta(\xi)=0\}
\end{equation*}
with the Baker-Campbell-Hausdorff product as the group multiplication, so that we have $\exp(x*y) = \exp(x)\exp(y)$.
We more generally set:
\begin{equation*}
Z_{\bullet}(\galg) := Z(\galg\hat{\otimes}\Omega^*(\Delta^{\bullet})),
\end{equation*}
where we take the completed tensor product of our dg Lie algebra $\galg$ with the simplicial dg commutative algebra $R = \Omega^*(\Delta^{\bullet})$.
This construction returns a simplicial group naturally associated to $\galg$.

We have the following observation:

\begin{prop}
Let $\AOp$ be a dg Hopf cooperad on which $\galg$ acts by biderivations such that for each element $\alpha\in\AOp(r)$,
we have $\ad(F^s\galg)(\alpha) = 0$ for $s\gg 1$,
where `$\ad$' denotes the action of the dg Lie algebra $\galg$
on $\AOp$.
Then we have a map of simplicial monoids
\begin{equation}\label{equ:ZtoAut}
Z_{\bullet}(\galg)\rightarrow\Map(\AOp,\AOp)^{\htimes} =
\Mor_{\dg\Hopf\Op^c_{\Omega^*(\Delta^{\bullet})}}(\AOp\otimes\Omega^*(\Delta^{\bullet}),\AOp\otimes \Omega^*(\Delta^{\bullet}))^{\htimes},
\end{equation}
which carries any cocycle $\xi\in Z_{\bullet}(\galg)$
to the morphism $\exp\ad(\xi): \AOp\otimes\Omega^*(\Delta^{\bullet})\rightarrow\AOp\otimes\Omega^*(\Delta^{\bullet})$
such that:
\begin{equation}
\exp\ad(\xi)(\alpha) := \sum_{k=0}^\infty \frac{1}{k!}\ad(\xi)^k(\alpha),
\end{equation}
for any $\alpha\in\AOp(r)\otimes\Omega^*(\Delta^{\bullet})$.
\end{prop}

\begin{proof}
This proposition follows from standard arguments. In brief, we mainly use that the exponential converts the derivation relation
with respect to the product $\ad(\xi)(\alpha\cdot\beta) = \ad(\xi)(\alpha)\cdot\beta + \alpha\cdot\ad(\xi)(\beta)$
into the product formula $\exp\ad(\xi)(\alpha\cdot\beta) = \exp\ad(\xi)(\alpha)\cdot\exp\ad(\xi)(\beta)$,
and similarly when we deal with the coderivation relation
with respect to the composition coproducts
of our our cooperad structure.
Recall that we use the superscript `$\htimes$' to denote the sum of the connected components
that are associated to weak-equivalences in the mapping space $\Map(-)$.
We trivially have $\exp\ad(\xi)\in\Map(\AOp,\AOp)^{\htimes}$, for any $\xi\in Z_{\bullet}(\galg)$, since the object $Z_{\bullet}(\galg)$
forms a simplicial group.
\end{proof}

Let us also note that the simplicial group $Z_{\bullet}(\galg)$ acts on the simplicial operad $G(\AOp)$
by essentially the same formula as in the proposition,
where we consider the simplicial realization functor $G: \dgHOpc\rightarrow\sSetOp^{op}$
of our adjunction between dg Hopf cooperads
and operads in simplicial sets.

\begin{rem}
Note that the dg Lie algebra $\galg$ and the dg Hopf cooperad $\AOp$ may be (cohomologically) $\ZZ$-graded.
Below we shall also need to pass to non-negatively graded dg Hopf cooperads.
To this end, we consider the truncation functor $\tau$ on dg Lie algebras such that
\begin{equation*}
(\tau\galg)^k = \begin{cases} Z^0(\galg) & \text{for $k=0$}, \\
\galg^k & \text{for $k<0$}.
\end{cases}
\end{equation*}
We easily check that the (non-positively graded) dg Lie algebra $\tau\galg$ acts on the non-negatively graded dg Hopf cooperad $\tau_{\sharp}\AOp$
by biderivations. (We refer to the Appendix for the definition of the truncation functor $\tau_{\sharp}$ on dg Hopf cooperads.)
We may hence apply the construction of the above proposition to obtain a map of simplicial monoids
\begin{equation*}
Z_{\bullet}(\tau\galg) = Z_{\bullet}(\galg)\rightarrow\Map_{\bullet}(\tau_{\sharp}\AOp,\tau_{\sharp}\AOp)^{\htimes}.
\end{equation*}
\end{rem}

\subsection{On dg Lie actions and $L_{\infty}$-morphisms}\label{sec:Def with Lie action}
Let now $\AOp$ and $\BOp$ be a pair of dg Hopf cooperads. The simplicial monoid $\Map(\AOp,\AOp)$ acts on the simplicial set $\Map(\AOp,\BOp)$.
For a dg Lie algebra as in the previous section $\galg$, we hence obtain an action of $Z_{\bullet}(\galg)$ on $\Map(\AOp,\BOp)$
by pullback of the morphism \eqref{equ:ZtoAut}.

If we fix a morphism of dg Hopf cooperads $\psi: \AOp\rightarrow\BOp$, then we obtain a map of simplicial sets
\begin{equation}\label{equ:ZtoMapCD}
Z_{\bullet}(\galg)\rightarrow\Map(\AOp,\BOp).
\end{equation}
If $\AOp$ and $\BOp$ satisfy the conditions of Section~\ref{sec:defcomplex}, then we may write:
\begin{equation*}
\Map(\AOp,\BOp) = \MC_{\bullet}(\Def(\AOp,\BOp)) = \MC_{\bullet}(\Def(\AOp\xrightarrow{\psi}\BOp)),
\end{equation*}
where we consider the Maurer--Cartan space associated to the $L_{\infty}$-algebra $\Def(\AOp\xrightarrow{\psi}\BOp)$ (see Paragraph~\ref{sec:schematic MC}).
Furthermore, we have an identity
\begin{equation*}
Z_{\bullet}(\galg) = \MC_{\bullet}(\galg[-1]),
\end{equation*}
where we consider the shifted graded vector space $\galg[-1]^* = \galg^{*-1}$
which we regard as an abelian $L_{\infty}$-algebra.
We now have the following proposition:

\begin{prop}\label{prop:U Def trivialization}
The above map \eqref{equ:ZtoMapCD} is identified with the map of Maurer--Cartan spaces
\begin{equation*}
Z_{\bullet}(\galg) = \MC_{\bullet}(\galg[-1])\rightarrow\MC_{\bullet}(\Def(\AOp,\BOp))
\end{equation*}
induced by an $L_{\infty}$-morphism $U : \galg[-1]\rightarrow\Def(\AOp\xrightarrow{\psi}\BOp)$,
which, in the language of Section~\ref{sec:schematic Lie}, we associate to the power series function such that
\begin{equation*}
U(\xi) = \pi\psi(\exp\ad(\xi)-\id)\iota,
\end{equation*}
for $\xi\in\galg[-1]\otimes R$, where the letter $\iota$ denotes the inclusion of generators into $\AOp$
and $\pi$ denotes the projection onto cogenerators of the dg Hopf cooperad $\BOp$.
\end{prop}

\begin{proof}
This proposition is immediate. Though this property is automatic, we just check the $L_{\infty}$-equations in the form~\eqref{equ:new Linfty mor}
for the morphism given in this proposition:
\begin{align*}
M_{\Def(\AOp\xrightarrow{\psi}\BOp)}(U(\xi)) & = \pi(\partial_{\BOp}\Phi(\pi\psi e^{\ad(\xi)}\iota)
- \Phi(\pi\psi e^{\ad(\xi)}\iota)\partial_{\AOp})\iota \\
& = \pi(\partial_{\BOp}\psi e^{\ad(\xi)} -  \psi e^{\ad(\xi)}\partial_{\AOp})\iota \\
& = \pi(\psi(\underbrace{e^{\ad(\xi)}\partial_{\AOp} - \partial_{\AOp}e^{\ad(\xi)}}_{= D\exp\ad(\xi)(\delta\xi)})
+ \underbrace{(\partial_{\BOp}\psi - \psi\partial_{\AOp})}_{= 0} e^{\ad(\xi)})\iota \\
& = D U(\xi)(\delta\xi) = D U(\xi)(M_{\galg[-1]}(\xi)).
\end{align*}
Here the power series of the left-hand side $M_{\Def(\AOp\xrightarrow{\psi}\BOp)}(-)$,
which encodes the $L_{\infty}$-structure on $\Def(\AOp\xrightarrow{\psi}\BOp)$,
is the one from Theorem~\ref{thm:defcx},
twisted via \eqref{equ:twisted M}.
This computation shows tat the map $U$ does define an $L_{\infty}$-morphism.
\end{proof}

\subsection{A version of the Goldman-Millson Theorem}
To any $L_{\infty}$-morphism of filtered complete $L_{\infty}$-algebras $U: \galg\rightarrow\halg$,
we associate a map of Maurer--Cartan spaces $\MC_{\bullet}(U): \MC_{\bullet}(\galg)\rightarrow\MC_{\bullet}(\halg)$.
The Goldman-Millson Theorem asserts that this map is a weak-equivalence
when $U$ induces a quasi-isomorphism on the associated graded
(see~\cite{DolgushevRogers}).
We shall need a slight generalization of this result, due to S. Schwarz:

\begin{thm}[{\cite{Schwarz}}]\label{thm:GM}
Let $\galg$ and $\halg$ be $L_{\infty}$-algebras with complete descending filtrations
\begin{equation*}
\galg = F^1\galg\supset F^2\galg\supset\cdots\quad\text{and}\quad\halg = F^1\halg\supset F^2\halg\supset\cdots
\end{equation*}
which we assume to be compatible with the $L_{\infty}$-structure (as in Section~\ref{sec:Linfinity}).
Let $U: \galg\rightarrow\halg$ be an $L_{\infty}$ morphism, compatible with the filtrations,
and assume that $U$ induces a quasi-isomorphism on the second page
of the spectral sequences (see below).
Suppose further that the following cohomologies vanish:
\begin{equation}\label{equ:vancond}
H^1(\galg/F^2\galg) = 0,\quad H^1(\halg/F^2\halg) = 0,\quad H^0(\halg/F^2\halg) = 0.
\end{equation}
Then the map $U$ induces a weak equivalence on the Maurer--Cartan-spaces:
\begin{equation}\label{equ:GM}
\MC_{\bullet}(U): \MC_{\bullet}(\galg)\xrightarrow{\sim}\MC_{\bullet}(\halg).
\end{equation}
\end{thm}

\begin{proof}
This is essentially the special case $r=2$ of \cite[Theorem 1.1]{Schwarz}.
There is a minor difference in that in \emph{loc. cit.} one also asks for the condition $H^0(\galg/F^2\galg) = 0$ to hold.
However, this condition is in fact never used: it appears non-trivially only in the proof of \cite[Lemma 3.4]{Schwarz}
and therein only in an induction step that is not invoked for $r=2$.
\end{proof}

\section{Graph complexes}\label{sec:graph_complexes}
In this section, we recall the definition of one particular model of the rational cochains of the little discs operads due to Kontsevich \cite{Kontsevich}.
We only sketch the construction. (For details, we refer for example to \cite[Section 7]{FTW} or to \cite[Section 2.2ff]{FW}.)

\subsection{The Kontsevich cooperad $\Graphs_n$}
An undirected graph with $r$ (external) vertices numbered $1,2,\dots, r$ and with $k$ unnumbered (internal) vertices is called admissible if it satisfies the following two conditions.
\begin{enumerate}
	\item Every internal vertex has valency at least 2.
	\item Every connected component contains at least one external vertex.
\end{enumerate}

In pictures we draw external vertices white and internal vertices black, as in the following example.
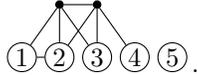
\begin{equation*}
\begin{tikzpicture}
\node[ext] (v1) at (0,0) {1};
\node[ext] (v2) at (0.5,0) {2};
\node[ext] (v3) at (1,0) {3};
\node[ext] (v4) at (1.5,0) {4};
\node[ext] (v5) at (2,0) {5};
\node[int] (w1) at (0.5,.7) {};
\node[int] (w2) at (1.0,.7) {};
\draw (v1) edge (v2) edge (w1) (v2) edge (w1) edge (w2) (v3) edge (w1) edge (w2) (v4) edge (w2) (w1) edge (w2);
\end{tikzpicture}\, .
\end{equation*}

Given a natural number $n$ we assign to each graph a (cohomological) degree wich is
\begin{equation*}
(n-1)(\#\text{edges}) - n(\#\text{internal vertices}).
\end{equation*}
An orientation on an admissible graph is the following data, depending on the parity of $n$.
\begin{enumerate}
	\item For even $n$ an orientation is an ordering of the set of edges.
	\item For odd $n$ an orientation is an ordering of the set of half-edges and the set of internal vertices.
\end{enumerate}
We say that an admissible graph with an orientation is an oriented graph.

Let $\Graphs^2_n(r)$ be the graded vector space spanned by oriented graphs with $r$ external vertices, modulo the following identification: If two oriented graphs are isomorphic we set them equal. If two oriented graphs are isomorphic up to changing the orientation, we set them equal up to the sign of the permutation needed to change the orientation.

The graded vector spaces $\Graphs^2_n(r)$ carry a right action of the symmetric groups $\Sigma_r$ by relabeling external vertices. They furthermore assemble into a dg Hopf cooperad $\Graphs^2_n$. Let us briefly sketch the combinatorial form of the operations.
First, the differential is obtained by edge contraction.
\begin{align*}
\delta
\begin{tikzpicture}[baseline=-.65ex]
\node[ext] (v) at (0,0) {$\scriptstyle j$};
\node[int](w) at (0,.3) {};
\draw (v) edge +(-.5,.5) edge +(.5,.5) edge (w) (w) edge +(-.2,.5) edge +(.2,.5);
\end{tikzpicture}
&=
\begin{tikzpicture}[baseline=-.65ex]
\node[ext] (v) {$\scriptstyle j$};
\draw (v) edge +(-.5,.5) edge +(-.2,.5) edge +(.2,.5) edge +(.5,.5);
\end{tikzpicture}
&
\delta
\begin{tikzpicture}[baseline=-.65ex]
\node[int] (v) at (0,0) {};
\node[int](w) at (0,.3) {};
\draw (v) edge +(-.5,.5) edge +(.5,.5) edge (w) (w) edge +(-.2,.5) edge +(.2,.5);
\end{tikzpicture}
&=
\begin{tikzpicture}[baseline=-.65ex]
\node[int] (v) {};
\draw (v) edge +(-.5,.5) edge +(-.2,.5) edge +(.2,.5) edge +(.5,.5);
\end{tikzpicture}
\end{align*}
The commutative algebra structure is given by gluing graphs along external vertices.
\begin{equation*}
\begin{tikzpicture}[baseline=-.65ex]
\node[ext] (v1) at (0,0) {1};
\node[ext] (v2) at (0.5,0) {2};
\node[ext] (v3) at (1,0) {3};
\node[ext] (v4) at (1.5,0) {4};
\node[int] (w1) at (0.5,.7) {};
\draw (v1) edge (v2) edge (w1) (v2) edge (w1) (v3) edge (w1) ;
\end{tikzpicture}
\wedge
\begin{tikzpicture}[baseline=-.65ex]
\node[ext] (v1) at (0,0) {1};
\node[ext] (v2) at (0.5,0) {2};
\node[ext] (v3) at (1,0) {3};
\node[ext] (v4) at (1.5,0) {4};
\node[int] (w2) at (1.0,.7) {};
\draw   (v2)  edge (w2) (v3)  edge (w2) (v4) edge (w2);
\end{tikzpicture}
=
\begin{tikzpicture}[baseline=-.65ex]
\node[ext] (v1) at (0,0) {1};
\node[ext] (v2) at (0.5,0) {2};
\node[ext] (v3) at (1,0) {3};
\node[ext] (v4) at (1.5,0) {4};
\node[int] (w1) at (0.5,.7) {};
\node[int] (w2) at (1.0,.7) {};
\draw (v1) edge (v2) edge (w1) (v2) edge (w1) edge (w2) (v3) edge (w1) edge (w2) (v4) edge (w2);
\end{tikzpicture}
\end{equation*}
Finally the cooperadic cocompositions are defined by subgraph contraction.
\begin{equation*}
\begin{tikzpicture}[baseline=-.65ex]
\node[ext] (v1) at (0,0) {1};
\node[ext] (v2) at (0.5,0) {2};
\node[ext] (v3) at (1,0) {3};
\node[ext] (v4) at (1.5,0) {4};
\node[ext] (v5) at (2,0) {5};
\node[int] (w1) at (0.5,.7) {};
\node[int] (w2) at (1.5,.7) {};
\draw (v1) edge (v2) edge (w1) (v2) edge (w1)  (v3) edge (w1) edge (w2) (v4) edge (w2) (v5) edge (w2);
\end{tikzpicture}
\mapsto
\begin{tikzpicture}[baseline=-.65ex]
\node[ext] (v1) at (0,0) {1};
\node[ext] (v2) at (0.5,0) {2};
\node[ext] (v3) at (1,0) {*};
\node[int] (w1) at (0.5,.7) {};
\draw (v1) edge (v2) edge (w1) (v2) edge (w1)  (v3) edge (w1);
\end{tikzpicture}
\otimes
\begin{tikzpicture}[baseline=-.65ex]
\node[ext] (v3) at (1,0) {3};
\node[ext] (v4) at (1.5,0) {4};
\node[ext] (v5) at (2,0) {5};
\node[int] (w2) at (1.5,.7) {};
\draw (v3) edge (w2) (v4) edge (w2) (v5) edge (w2);
\end{tikzpicture}
\end{equation*}

We also consider the quotient
\begin{equation}\label{equ:Graphs2quotient}
	\Graphs_n \twoheadleftarrow \Graphs_n^2
\end{equation}
by sending to zero all diagrams with bivalent internal vertices.
It is shown in \cite{Willwacher} that the map \eqref{equ:Graphs2quotient} is a quasi-isomorphism.
Furthermore, we note that $\Graphs_n$ is non-negatively graded for $n\geq 3$.

For $n\geq 2$ there is furthermore a natural map of dg Hopf cooperads
\begin{equation}
\label{equ:graphsen}
\Graphs_n \rightarrow \eop_n^c := H^*(\DOp_n)
\end{equation}
into the cohomology of the little $n$-discs operad, sending all diagrams with internal vertices to zero, and sending an edge between external vertices $i$ and $j$ to an algebra generator $\omega_{ij}\in H^{n-1}(\DOp_n)$.
It is shown by Kontsevich \cite{Kontsevich} and Lambrechts-Voli\'c \cite{LambrechtsVolic} that the map \eqref{equ:graphsen} is a quasi-isomorphism. By the rational formality of the little $n$-discs operads we may hence take $\Graphs_n$ or $\Graphs_n^2$ as a dg Hopf cooperad model of the little $n$-discs operad.

One important observation is that the graph algebras are quasi-free dg commutative algebras:
\begin{equation*}
\Graphs_n(r) = S(\IG_n(r)),
\quad
\Graphs_n^2(r) = S(\IG_n^2(r)).
\end{equation*}
The generating symmetric sequences $\IG_n$ and $\IG_n^2$ are spanned by the internally connected graphs, \emph{i.e.}, by those graphs that remain connected after deleting all external vertices.
As a consequence one obtains the following two results.
\begin{prop}
The dg Hopf cooperads $\Graphs_n$ and $\Graphs_n^2$ are cofibrant objects in $\dgZHOpc$.
\end{prop}
\begin{proof}
	We apply the cofibrancy criterion of Lemma \ref{lem:cofib_criterion}. In the notation of the Lemma, we merely take for $\MOp$ the internally connected graphs ($\IG_n$ or $\IG_n^2$), and for $F^p\MOp\subset\MOp$ the graphs with at most $p$ edges.
Since the composition coproducts preserve, and the differential only removes edges, the conditions of Lemma~\ref{lem:cofib_criterion} are evidently satisfied. Hence cofibrancy of our dg Hopf cooperads follows.
\end{proof}

\begin{cor}
The truncations $\tau_{\sharp}\Graphs_n$ and $\tau_{\sharp}\Graphs_n^2$ are cofibrant objects in $\dgHOpc$ weakly equivalent to $H^*(\DOp_n)$.
\end{cor}

\begin{proof}
The functor $\tau_{\sharp}$ is a left Quillen functor and hence sends cofibrant objects to cofibrant objects.
Next, our objects fit into a commutative diagram
\[
\begin{tikzcd}
\Graphs_n^2 \ar{r}{\sim} \ar{d} & \Graphs_n \ar{r}{\sim} \ar{d} & H^*(\DOp_n) \ar[-]{d}{=} \\
\tau_{\sharp}\Graphs_n^2 \ar{r} & \tau_{\sharp}\Graphs_n \ar{r} & \tau_{\sharp}H^*(\DOp_n)
\end{tikzcd}\, ,
\]
where the lower row is obtained from the upper by applying $\tau_\sharp$, and the downward morphisms are the natural projections to a quotient, or alternatively, given by the adjunction unit.
Since both $\Graphs_n$ and $\Graphs_n^2$ are (arity-wise) augmented and cohomologically connected we can apply Proposition \ref{prop:tru_unit} to see that the vertical arrows are weak equivalences. Hence so are all arrows.
\end{proof}

We note finally that for $n\geq 3$ we have $\tau_{\sharp}\Graphs_n=\Graphs_n$.

\subsection{Graph complex and action}
\label{sec:graph action}
Besides cofibrancy, the important feature of the model $\Graphs_n^2$ for the little discs operad for our purposes is that it is acted upon by a large dg Lie algebra of symmetries, which is again given by graphs.

Let us define the relevant (Kontsevich) graph complex $\GC_n^2$.
We say that an undirected graph is good if is satisfies the following properties:
\begin{enumerate}
	\item It is connected.
	\item It is 1-vertex irreducible, \emph{i.e.}, the graph cannot be split into multiple connected components by removing one vertex.
	\item Every vertex has valence at least 2.
\end{enumerate}
We shall draw such graphs with black vertces only
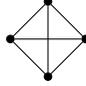
\begin{equation*}
 \begin{tikzpicture}[baseline=-.65ex, scale=.5]
\node[int] (v1) at (0:1) {};
\node[int] (v2) at (90:1) {};
\node[int] (v3) at (-90:1) {};
\node[int] (v4) at (180:1) {};
\draw (v1) edge (v2) edge (v3) edge (v4)
(v2) edge (v3) edge (v4)
(v3) edge (v4);
\end{tikzpicture}
\end{equation*}
Fixing a natural number $n$ we associate a graph with the cohomological degree
\begin{equation*}
-n(\#\text{vertices}-1) + (n-1)(\#\text{edges})\, .
\end{equation*}
We furthermore define the notion of orientation on such a graph as before.
We define the graph complex $\GG_n^2$ as consisting of linear combinations of oriented good graphs, modulo the identification of isomorphic graphs, and orientations up to sign.
The space $\GG_n^2$ is naturally a dg Lie coalgebra.
The differential is again given by edge contraction, and the Lie cobracket is defined by subgraph contraction.
As before we define a quotient dg Lie coalgebra
\begin{equation*}
\GG_n\twoheadleftarrow \GG_n^2
\end{equation*}
by setting to zero all graphs with bivalent vertices.

For convenience we shall also define the dual dg Lie algebras $\GC_n^2=(\GG_n^2)^*$ and $\GC_n=(\GG_n)^*$.
The differential is the dual operation to edge contraction, which is vertex splitting.
In particular, the loop order provides a (complete) grading on $\GC_n$, $\GC_n^2$, which can also be checked to be compatible with the bracket (in the sense that the bracket of a $p$-loop graph and a $q$-loop graph is a linear combination of $p+q$-loop graphs).
We can hence consider the semidirect product
\begin{equation*}
\QQ L\ltimes \GC_n^2,
\end{equation*}
where $L$ acts as the generator of the loop order grading, \emph{i.e.},
\begin{equation*}
[L,\gamma] = (\#\text{loops}) \gamma
\end{equation*}
for a graph $\gamma\in\GC_n^2$.

We shall furthermore recall from \cite{FW, Willwacher} that the dg Lie algebra $\GC_n^2$ acts on the dg Hopf cooperad $\Graphs_n^2$ by biderivations, \emph{i.e.}, derivations with respect to the commutative algebra structure and coderivations with respect to the cooperad structure.
This action can furthermore be extended to a dg Lie algebra action of $\QQ L\ltimes \GC_n^2$ on $\Graphs_n^2$, where we let the grading generator act on a graph $\Gamma \in \Graphs_n^2$ by the formula
\begin{equation*}
L\cdot\Gamma = ((\#\text{edges})-(\#\text{internal vertices}))\Gamma
\end{equation*}
We finally note that the 1-vertex-irreducibility condition we impose on the graphs in $\GC_n^2$ is necessary in order for the action on $\Graphs_n^2$ to be a derivation with respect to the commutative algebra structure. Nevertheless, omitting this condition yields a quasi-isomorphic complex.

\subsection{On dg Lie actions and automorphisms - graph case}\label{sec:graphaction}
We next apply the construction of Section~\ref{sec:dgLieactions} to the dg Lie algebra $\QQ L\ltimes \GC_n^2$
acting on the dg Hopf cooperad $\Graphs_n^2$.
One needs to be slightly careful since the dg Lie algebra $\QQ L\ltimes \GC_n^2$ is not pro-nilpotent, only $\GC_n^2$ is. We shall hence handle the non-pro-nilpotent piece separately and define the simplicial group
\begin{equation*}
G_{n\bullet}:=\QQ^\times \ltimes Z_{\bullet}(\GC_n^2).
\end{equation*}
Here the semidirect product is taken such that an element $\lambda\in \QQ^\times$ acts on $X\in Z_{\bullet}(\GC_n^2)$ as $\lambda^L$, \emph{i.e.}, by multiplying a graph $\Gamma$ of loop order $k$ by $\lambda^k$.

Similarly, we may easily extend the construction of subsection \ref{sec:dgLieactions} to obtain an map of simplicial monoids
\begin{equation*}
\QQ^\times \times Z_{\bullet}(\GC_n^2)
\rightarrow \Map(\Graphs_n^2, \Graphs_n^2)^{\htimes}.
\end{equation*}
Here the action of $\QQ^\times$ is induced from the action of $\QQ^\times$ on $\Graphs_n^2$ such that $\lambda\in \QQ^\times$ acts on a graph $\Gamma\in \Graphs_n^2(r)$ by
\begin{equation*}
\lambda\cdot \Gamma = \lambda^{(\#\text{edges})-(\#\text{internal vertices})}\Gamma.
\end{equation*}
Similarly we obtain a truncated version of the map above,
\begin{equation*}
\QQ^\times \times Z_{\bullet}(\GC_n^2)
=\QQ^\times \times Z_{\bullet}(\tau\GC_n^2)
\rightarrow \Map(\tau_{\sharp}\Graphs_n^2, \tau_{\sharp}\Graphs_n^2)^{\htimes}.
\end{equation*}

\section{The homotopy automorphisms of the little discs operads and proof of Theorem \ref{thm:main_intro}}\label{sec:aut thm proof}
The goal of this section is to show Theorem \ref{thm:main_intro}.
Let us denote by $\EOp_n$ a cofibrant simplicial operad whose realization is weakly equivalent to the little $n$-discs operad $\DOp_n$.
The first weak equivalence of Theorem \ref{thm:main_intro} follows directly from Proposition \ref{prop:good and Aut} for $n\geq 3$
and from Proposition \ref{prop:Aut E2} for $n=2$.

For the second equality, first recall from~\cite[\S II.14]{OperadHomotopyBook} and~\cite{FW} that the little discs operads are rationally formal.
Hence we know that $\Omega_{\sharp}^*(\EOp_n)\sim\eop_n^c = H^*(\DOp_n)$, so that
\begin{equation*}
\Aut^h_{\dgHOpc}(\Omega_{\sharp}^*(\EOp_n))
\sim
\Aut^h_{\dgHOpc}(\eop_n^c).
\end{equation*}
Furthermore, we know from Section~\ref{sec:graphaction} that the simplicial group $\QQ^\times \ltimes Z_{\bullet}(\GC_n^2)$ acts on the cofibrant model $\tau_{\sharp}\Graphs_n^2\sim \eop_n^c$.
Let us fix the canonical quasi-isomorphism $\phi:\Graphs_n^2 \rightarrow W^c(\eop_n^c)$ given by the composition
\begin{equation}\label{equ:phidef8}
\begin{tikzcd}
\Graphs_n^2 \ar{r}{\sim} \ar[bend right]{rr}{\phi}& \eop_n^c \ar{r}{\sim} & W^c(\eop_n^c)
\end{tikzcd},
\end{equation}
where the first horizontal arrow is given by \eqref{equ:Graphs2quotient} and \eqref{equ:graphsen}.
By Proposition \ref{prop:hAut recognition} we can hence reduce the second weak equivalence of Theorem \ref{thm:main_intro} to the following result.
\begin{thm}\label{thm:main2}
The map
\begin{equation}\label{equ:thmmain2}
\QQ^\times \ltimes Z_{\bullet}(\GC_n^2)
\rightarrow
\Map(\tau_{\sharp}\Graphs_n^2,W^c(\eop_n^c))^{\htimes}
\end{equation}
induced by the action on the morphism $\phi$ of \eqref{equ:phidef8} is a weak equivalence of simplicial sets.
\end{thm}

Note that by adjunction
\begin{equation*}
\Map(\tau_{\sharp}\Graphs_n^2,W^c(\eop_n^c))^{\htimes} = \Map(\Graphs_n^2,W^c(\eop_n^c))^{\htimes}.
\end{equation*}
To show Theorem \ref{thm:main2} we furthermore need to separate the reductive part $\QQ^\times$ from our simplicial group on the left-hand side of \eqref{equ:thmmain2}.
To this end we consider the diagram of fiber-sequences
\begin{equation}\label{equ:fibdiag}
\begin{tikzcd}
\Map( \Graphs_n^2, W^c(\eop_n^c))^{\htimes}_1 \ar{r} &
\Map( \Graphs_n^2, W^c(\eop_n^c))^{\htimes}  \ar{r}{\pi}&
\QQ^\times = \pi_0\Aut(\eop_n^c) \\
Z_{\bullet}(\GC_n^2) \ar{r}\ar{u}&
G_{n\bullet}\ar{r}\ar{u}
&
\QQ^\times \ar{u}{=}
\end{tikzcd},
\end{equation}
where $\Map( \Graphs_n^2, W^c(\eop_n^c))^{\htimes}_1$ consists of the connected components that correspond to morphisms inducing the identity map on cohomology.
To show Theorem \ref{thm:main2} it is clearly sufficient to show that the left-hand vertical arrow (on fibers) in \eqref{equ:fibdiag} is a weak equivalence.
Note that
\begin{equation*}
Z_{\bullet}(\GC_n^2) = \MC_{\bullet}(\GC_n^2[-1]),
\end{equation*}
with $\GC_n^2[-1]$ considered as abelian $L_{\infty}$-algebra.
Furthermore the mapping space
\begin{equation*}
\Map_{\bullet}( \Graphs_n^2, W^c(\eop_n^c))^{\htimes}_1
=\MC_{\bullet}(\Def'(\Graphs_n^2\xrightarrow{\phi} W^c(\eop_n^c)))
\end{equation*}
can be realized as the Maurer--Cartan space of a codimension one $L_{\infty}$-subalgebra $\Def'(\Graphs_n^2\xrightarrow{\phi} W^c(\eop_n^c))\subset \Def(\Graphs_n^2\xrightarrow{\phi} W^c(\eop_n^c))$ of the deformation complex.
Furthermore, from Section~\ref{sec:Def with Lie action} we see that the map
\begin{equation*}
Z_{\bullet}(\GC_n^2)
\rightarrow
\Map_{\bullet}( \Graphs_n^2, W^c(\eop_n^c))^{\htimes}_1
\end{equation*}
is induced by an $L_{\infty}$-morphism
\begin{equation}\label{equ:U in main proof}
U: \GC_n^2[-1] \rightarrow \Def'(\Graphs_n^2\xrightarrow{\phi} W^c(\eop_n^c))
\end{equation}
(see Proposition~\ref{prop:U Def trivialization}).
By the Goldman-Millson Theorem \ref{thm:GM}, we can hence reduce Theorem \ref{thm:main2} to the following statement.

\begin{prop}\label{prop:main_computation}
The $L_{\infty}$-morphism $U$ of \eqref{equ:U in main proof} satisfies the assumptions of the Goldman-Millson Theorem \ref{thm:GM}, with the filtrations inherited from the filtration by number of edges on $\Graphs_n^2$.
Concretely, this means that:
\begin{enumerate}
\item\label{prop:main_computation:linear_part}
The linear part of $U$,
\begin{equation}\label{equ:main_computation_U1}
U_1: \GC_n^2[-1] \rightarrow \Def'(\Graphs_n^2\xrightarrow{\phi} W^c(\eop_n^c)),
\end{equation}
induces a quasi-isomorphism on the second page of the spectral sequences.
\item\label{prop:main_computation:cohomology}
We have that
\begin{equation}\label{equ:H1van0}
H^{0}(\GC_n^2 /F^2\GC_n^2)=0.
\end{equation}
and
\begin{equation}\label{equ:H1van}
H^i(\Def'(\Graphs_n^2\xrightarrow{\phi} W^c(\eop_n^c))/F^2\Def'(\Graphs_n^2\xrightarrow{\phi} W^c(\eop_n^c)))=0
\end{equation}
for $i=0,1$.
\end{enumerate}
\end{prop}

\begin{proof}
We begin with statement~(\ref{prop:main_computation:linear_part}).
This statement has essentially already been shown in the literature, see for example \cite[Theorem 2.3.7]{FW} or \cite{FTW}, albeit in a difficult-to-reference form, using different notation and conventions.
We shall guide the reader through the computation again, in our notation and language, but refer to the literature for the proof of some sub-statements.

First, take the spectral sequence associated to the descending complete filtrations by the number of edges in graphs on both sides of \eqref{equ:main_computation_U1}.
We will also temporarily forget the $(-)'$ on the right-hand side of \eqref{equ:main_computation_U1}, and work with $\Def(\Graphs_n^2\xrightarrow{\phi} W^c(\eop_n^c))$ instead.

Recall from Section~\ref{sec:explicit differential} that the differential on the right-hand side has the form $\delta=\delta'+\delta''$
with the ``algebra'' piece $\delta''$ as in \eqref{equ:dprimeprime}
and the ``operadic'' piece $\delta'$ as in \eqref{equ:dprime}
(see Proposition~\ref{prop:explicit differential}).
Now, on the associated graded we ``see'' only those pieces of the differential that do not create edges.
Concretely, this means that the map $\psi = \phi$ appearing in \eqref{equ:dprime} and in \eqref{equ:dprimeprime} can be replaced by the trivial map $\Graphs_n^2\xrightarrow{*} W^c(\eop_n^c)$ factoring through $\Com^c$:
\begin{equation*}
\begin{tikzcd}
\Graphs_n^2 \ar{r}\ar[bend right]{rrr}{*} & \Com^c \ar{r} & \eop_n^c \ar{r} & W^c(\eop_n^c)
\end{tikzcd}.
\end{equation*}
This means that on the 0-th page of our spectral sequence the complex $\Def(\Graphs_n^2\xrightarrow{\phi} W^c(\eop_n^c))$ becomes the same as $\Def(\Graphs_n^2\xrightarrow{*} W^c(\eop_n^c))=\Def(\Graphs_n^2,W^c(\eop_n^c))$.
In (\ref{equ:dprime}-\ref{equ:dprimeprime}) all terms in which $\psi$ is applied non-trivially drop out.
The resulting complex
\begin{equation}\label{equ:endefICG}
(\Hom_{\gSeqc}(\IG_n^2, \mathring{W}^c(\eop_n^c)), \delta)
\simeq
\underbrace{\left( \prod_{r\geq 2}\ICG_n^2(r) \hat{\otimes}_{\Sigma_r}\mathring{W}^c(\eop_n^c)(r),\delta\right)}_{=: (\ICG_n^2 \hat{\otimes}_{\Sigma}\mathring{W}^c(\eop_n^c),\delta)}
\end{equation}
computes the $E_n$-homology of the infinitesimal $\Com$-bimodule $\ICG_n^2 := (\IG_n^2)^*$.
This has been computed at various places in the literature,
see \emph{e.g.} \cite[Theorem 2.3.3]{FW} or the earlier reference \cite{AroneTurchin}. It agrees with the hairy graph complex%\footnote{We tacitly ignore here a conventional degree shift in the definition of the hairy graph complex that appears in other references.}
\begin{equation*}
\HGC_{n,n}^2 := (\ICG_n^2)^{=1}\hat{\otimes}_{\Sigma}\Com\{n\},
\end{equation*}
where $(\ICG_n^2)^{=1}\subset\ICG_n^2$ consists of the graphs all of whose external vertices have valency exactly one.
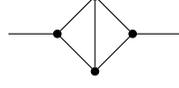
\begin{equation*}
\begin{tikzpicture}[scale=.5]
\node[int] (v1) at (-1,0){};\node[int] (v2) at (0,1){};\node[int] (v3) at (1,0){};\node[int] (v4) at (0,-1){};
\draw (v1)  edge (v2) edge (v4) -- +(-1.3,0) (v2) edge (v4) (v3) edge (v2) edge (v4) -- +(1.3,0);
\end{tikzpicture}
\end{equation*}
For the differentials on the next page of our spectral sequence we have to extract from (\ref{equ:dprime}-\ref{equ:dprimeprime})
all terms in which $\psi$ is applied non-trivially exactly once, and to a graph with exactly one edge.
The only nontrivial such terms come from the internal differential $\delta_{\ICG}$ on $\ICG_n^2$ and from the very last term in \eqref{equ:dprime},
with the latter (say $\delta_{\rm attach}$) acting by attaching one additional hair to a hairy graph, in all possible ways.
Now we also need to remember that on the right-hand side of \eqref{equ:main_computation_U1} we have a codimension one subspace of the deformation complex.
This is reflected on the current page of our spectral sequences by replacing the hairy graph complex $\HGC_{n,n}$ by the codimension 1 subspace
\begin{equation*}
(\HGC_{n,n}^2)' \subset \HGC_{n,n}^2
\end{equation*}
consisting of series of graphs with at least two edges.
In other words we discard multiples of the one-edge hairy graph
\begin{equation*}
\begin{tikzcd}
\draw (0,0) -- (.7,0);
\end{tikzcd}
\end{equation*}
from our complex.
The map resulting from \eqref{equ:main_computation_U1} on the $E^1$-page
\begin{equation*}
(\GC_n^2[-1],\delta) \rightarrow  ((\HGC_{n,n}^2)',\delta_{\ICG}+\delta_{\rm attach})
\end{equation*}
just adds one hair to a non-hairy graph.
It has been computed in \cite[Proposition 2.2.9]{FW} that this map is a quasi-isomorphism,
thus finishing our proof of property (\ref{prop:main_computation:linear_part}).
	
Now turn to the statement~(\ref{prop:main_computation:cohomology}) of the Proposition.
We first note that in the graph complex $\GC_n^2$ there is only one graph with a single edge, which is the tadpole graph $\begin{tikzpicture}[baseline=-0.5ex]
	\node[int](v) at(0,0) {};
	\draw (v) to [ out=120, in=60, loop] (v);
	\end{tikzpicture}$, living in degree $1-n\neq 0$. Hence \eqref{equ:H1van0} follows.	
	
	Now consider the second claim \eqref{equ:H1van}. Note that the above computation identified the cohomology of the associated graded with respect to the filtration by number of edges with the codimension one subspace of the hairy graph complex, removing the unique hairy graph with one edge.
	Hence, there are no further hairy graphs with one edge left and we conclude in particular that \eqref{equ:H1van} must hold.
\end{proof}

\section{Unitary versions of the main results}\label{sec:Lambdaoperadmodel}
For the moment, we have formulated our results in the categories of operads and cooperads without nullary (co)operations.
In particular, the $E_n$ operads so far had no nullary operations. We claim that they can be incorporated with minimal changes.
In particular, Theorem~\ref{thm:main_intro} remains valid in this ``unitary'' setting without change.
We explain this enhancement of our constructions in this section.

\subsection{On $\Lambda$ operads and $\Lambda$ cooperads}
We still assume that our symmetric sequences are concentrated in positive arities.
We use the formalism of $\Lambda$ structures to incorporate a single nullary operation in our objects.
We refer to \cite{OperadHomotopyBook} and \cite[Section 4]{ExtendedRHT}
for details on this formalism, developed by the first author.
We shall just recall that a simplicial $\Lambda$-operad $\POp$ is a simplicial operad as above, together with a collection of maps
\begin{equation*}
u^*: \POp(l)\rightarrow\POp(k),
\end{equation*}
that formally encode the composition with a nullary operation,
and which we associate to any injective order preserving map $u: \{1<\cdots<k\}\rightarrow\{1<\cdots<l\}$.
The $\Lambda$-operations have to satisfy natural compatibility conditions with the operadic structure.
The idea is that the object
\begin{equation*}
\POp_+(r) =
\begin{cases}
\POp(r), & \text{for $r>0$}, \\
*, & \text{for $r=0$},
\end{cases}
\end{equation*}
inherits an operad structure when we use the $\Lambda$-structure to extend the operadic compositions to those involving the new nullary element $*$.
We denote the category of simplicial $\Lambda$-operads by $\sLaOp$.

Dually, we consider dg $\Lambda$ cooperads and dg Hopf $\Lambda$ cooperads, which are $\Lambda$ cooperads in the category of dg commutative algebras.
In this setting the $\Lambda$-structure consists of maps
\begin{equation}\label{equ:CLa}
u_*: \COp(k)\rightarrow\COp(l)
\end{equation}
that formally encode the composition coproducts with a nullary cooperation.
There are again compatibility conditions ensuring that the $\Lambda$-structure are the restrictions
of the composition coproducts of a larger dg cooperad $\COp_+$,
obtained by adding a nullary term $\COp_+(0) = \kk$
to $\COp$.
We just have to take care that a dg $\Lambda$ cooperad is always required to be coaugmented over the commutative cooperad $\Com^c$
in the sense that we assume the existence of a morphism
\begin{equation*}
\Com^c\rightarrow\COp
\end{equation*}
that preserves $\Lambda$-structures. This morphism represents the composition coproducts $\COp_+(0)\rightarrow\COp_+(r)\otimes\COp_+(0)^{\otimes r}$,
where we put the nullary term $\COp_+(0) = \kk$
at all positions.
We omit to mention this extra structure in our terminology (in contrast to~\cite[\S II.11]{OperadHomotopyBook}).
In the case of a dg Hopf $\Lambda$ cooperad $\AOp$, the coaugmentation is actually given by the unit morphisms $\eta: \kk\rightarrow\AOp(r)$
of the dg commutative algebras $\AOp(r)$, where we use the identity $\Com^c(r) = \kk$.

By convention, we still assume that our dg $\Lambda$ cooperads are conilpotent in the sense that the underlying dg cooperad
of a dg $\Lambda$ cooperad is conilpotent in the ordinary sense
when we forget the $\Lambda$-structure.
We denote the category of dg $\Lambda$ cooperads by $\dgLaOpc$ (by $\dgZLaOpc$ when we deal with $\ZZ$-graded objects).
We denote the category of dg Hopf $\Lambda$ cooperads by $\dgLaHOpc$ (by $\dgZLaHOpc$ when we deal with $\ZZ$-graded objects).

\subsection{Simplicial model structure}
We refer to~\cite[\S II.8.4, \S II.11.1]{OperadHomotopyBook} and to \cite[Section 4]{ExtendedRHT}
for the definition of model structures on the categories $\sLaOp$, $\dgLaOpc$
and $\dgLaHOpc$.
We shall just need to recall that the fibrations of $\dgLaOpc$ and $\dgLaHOpc$ are created in $\dgOpc$.
In other words, a morphism $\phi: \AOp\rightarrow\BOp$ in either of the aforementioned categories
is a fibration if and only if this morphism defines a fibration in $\dgOpc$.

We now extend our construction of the right finitely continuous simplicial model structure of Section~\ref{sec:simplicial model structure} to the category $\dgLaHOpc$.
We define the simplicial mapping spaces again by \eqref{equ:Mapping space def}:
\begin{equation}
\label{equ:Mapping space def La}
\Map(\AOp,\BOp)_n= \Mor_{\dgLaHOpc/\Omega^*(\Delta^n)}
(\AOp\otimes\Omega^*(\Delta^n),
\BOp\otimes\Omega^*(\Delta^n)),
\end{equation}
where $\dgLaHOpc/\Omega^*(\Delta^n)$ denotes the category of dg Hopf $\Lambda$-cooperads defined over the ground dg algebra $\Omega^*(\Delta^n)$.
Next one notes that for a dg Hopf $\Lambda$ cooperad $\BOp$
the dg Hopf cooperad $\BOp^K$ of Proposition \ref{prop:adjunctionrelation}
naturally inherits a $\Lambda$ structure
and satisfies the adjunction relation
\begin{equation}
\label{equ:adjunctionrelation La}
\Mor(K,\Map(\AOp,\BOp)) = \Mor_{\dgLaHOpc}(A,\BOp^K)
\end{equation}
in the category of dg Hopf $\Lambda$ cooperads.
To check this claim, we use that the forgetful functor from the category of $\Lambda$ cooperads
to the category of ordinary cooperads
creates cofree objects
and that limits in the category of Hopf $\Lambda$ cooperads are taken in the category of ordinary dg cooperads (see~\cite[\S II.11.1]{OperadHomotopyBook}).

Since the fibrations in $\dgLaHOpc$ are created in $\dgOpc$, we then conclude that the pullback-corner-property, Proposition \ref{prop:pullbackcorner}, continues to hold in the category of Hopf dg $\Lambda$ cooperads.
The $\Lambda$-analogue of Theorem \ref{thm:map comparison} is then a formal consequence, whose proof goes through without change.
Summarizing, we have the following result.

\begin{prop}
The category $\dgLaHOpc$ is simplicially enriched via \eqref{equ:Mapping space def La} and lax cotensored over finite simplicial sets,
so that the adjunction relation \eqref{equ:adjunctionrelation La} and  the pullback-corner-property (see Proposition \ref{prop:pullbackcorner})
hold in $\dgLaHOpc$.
	
Furthermore, the object~\eqref{equ:Mapping space def La} provides a model of the mapping spaces in the model category of dg Hopf $\Lambda$ cooperads.
\end{prop}

We also note that the constructions of the model structure on $\dgZHOpc$ and $\dgZHOpc$ (see Appendix~\ref{sec:modelcat})
extends to the $\ZZ$-graded versions of the categories of $\Lambda$ cooperads $\dgZLaOpc$ and $\dgZLaHOpc$
without changes.
However, we again cannot show the analogous version of Proposition \ref{prop:pullbackcorner} in this setting
(see also the remark at the end of Section~\ref{sec:simplicial model structure}).

\subsection{Cofibrancy criterion}
While the fibrations in our categories of $\Lambda$ cooperads are inherited from ordinary cooperads, the cofibrations are more complicated.
We first recall from \cite[Section 4]{ExtendedRHT} that there are Quillen adjunctions
\begin{equation*}
\begin{tikzcd}
  \dgOpc  \ar[shift left]{r}{\FreeOp_{\Lambda}} & \dgLaOpc \ar[shift left]{l}\ar[shift left]{r}{\Sym} & \dgLaHOpc \ar[shift left]{l}
 \end{tikzcd}\,,
\end{equation*}
with the arrows from right-to-left being the obvious forgetful functors. We will denote, abusing notation, $\Sym_{\Lambda}=\Sym\circ \FreeOp_{\Lambda}$.
We furthermore have analogous adjunctions in the $\ZZ$-graded situation.

Now the cofibrancy criterion Lemma \ref{lem:cofib_criterion} goes through with the following modifications.

\begin{lemm}\label{lem:cofib_criterionLa}
Let $\AOp= (S(\MOp),\partial)$ be a dg Hopf $\Lambda$ cooperad (either $\ZZ$-graded or non-negatively graded).
We assume that $\AOp$ is a quasi-free dg commutative algebra arity-wise, for a generating graded symmetric sequence $\MOp$
which is in turn a free $\Lambda$-module generated by $\MOp_0$.
We write $S(\MOp) = S_{\Lambda}(\MOp_0)$.
We suppose that $\MOp_0$ is equipped with an exhaustive filtration $0 = F^0\MOp_0\subset F^1\MOp_0\subset\cdots\subset F^s\MOp_0\subset\cdots\subset\MOp$
such that the symmetric sequence
$S_{\Lambda}(F^{s-1}\MOp_0) + F^s\MOp_0\subset S_{\Lambda}(\MOp_0)$
is preserved by the differential and the composition coproducs of the cooperad $\AOp$, for every $s\geq 1$.
Then $\AOp$ defines a cofibrant object in the category of dg Hopf $\Lambda$ cooperads ($\dgLaHOpc$ or $\dgZLaHOpc$.)
\end{lemm}

\begin{proof}
We first write $\AOp$ as a colimit of subobjects $F^s\AOp\subset\AOp$ such that $F^s\AOp = (S_{\Lambda}(F^s\MOp_0),\partial)$.
We moreover have a pushout diagram in the category of dg Hopf $\Lambda$ cooperads:
\begin{equation*}
\xymatrix{ \Sym_{\Lambda}(S_{\Lambda}(F^{s-1}\MOp_0),\partial)\ar@{^{(}->}[]!D-<0pt,2pt>;[d]\ar[r] & F^{s-1}\AOp\ar@{.>}[d] \\
\Sym_{\Lambda}(S_{\Lambda}(F^{s-1}\MOp_0) + F^s\MOp_0,\partial)\ar@{.>}[r] & F^s\AOp },
\end{equation*}
for each $s\geq 1$, where we consider coaugmented dg $\Lambda$ cooperads such that $\COp = (S_{\Lambda}(F^{s-1}\MOp_0),\partial)$,
$\DOp = (S_{\Lambda}(F^{s-1}\MOp_0) + F^s\MOp_0,\partial)$,
and we take the morphism of dg Hopf cooperads obtained by applying the $\Lambda$ cooperadic symmetric algebra functor $\Sym_{\Lambda}(-)$
to the inclusion of these coaugmented dg $\Lambda$ cooperads $i: \COp\hookrightarrow\DOp$.
Then the morphism $\Sym_{\Lambda}(i): \Sym_{\Lambda}(\COp)\hookrightarrow\Sym_{\Lambda}(\DOp)$ is a cofibration
by definition of the model structure of dg Hopf $\Lambda$ cooperads
by adjunction from the model structure of coaugmented dg $\Lambda$ cooperads.
We just use that the class of cofibrations in a model category is closed under pushouts and transfinite compositions
to get the conclusion of the lemma.
\end{proof}

\subsection{Deformation complexes}
One can also extend the definition of the deformation complex (and of the deformation $L_{\infty}$-algebra) $\Def(\AOp,\BOp)$ from Section~\ref{sec:defcomplex} to the Hopf $\Lambda$-cooperadic setting.
We merely need to require from $\AOp$ not only that $\AOp(r)=S(\MOp(r))$ is a free dg commutative algebra, but also that the generators $\MOp=\FreeOp_{\Lambda} \MOp_0$ are a free $\Lambda$-module.
Then the analogue of the result of Proposition \ref{prop:pre_defcx} reads:
\begin{prop}\label{prop:pre_defcxLa}
	Given dg Hopf cooperads $\AOp$, $\BOp$ as above and a graded commutative algebra $R$ the map
	\begin{equation*}
	\Mor_{\g\Hopf\Lambda\Op^c_R}(\AOp^{\flat}\otimes R,\BOp^{\flat}\otimes R)
	\rightarrow
	\Mor_{\g\Seq^c_R}(\MOp_0\otimes R,\NOp\otimes R)
	\end{equation*}
	given by precomposition with the inclusion of generators $\MOp_0\rightarrow\AOp$ and by postcomposition with the projection
onto cogenerators $\BOp\rightarrow\NOp$ is a bijection.
\end{prop}
The proof is essentially identical.
The remainder of the construction of the deformation complex also remains the same in the Hopf $\Lambda$ cooperad case.

We will denote the resulting deformation $L_{\infty}$-algebra by $\Def_{\Lambda}(\AOp,\BOp)$, and its twisted version $\Def_{\Lambda}(\AOp\xrightarrow{\psi}\BOp)$, to distinguish them from the non-$\Lambda$ variants.

\subsection{Graphs}
The dg Hopf cooperads $\Graphs_n$ and $\Graphs_n^2$ of Section~\ref{sec:graph_complexes} have a natural $\Lambda$ structure. The $\Lambda$ operations \eqref{equ:CLa} are obtained by merely adding external vertices of valency zero to the graph.
We have seen before that $\Graphs_n^2(r)=S(\IG_n^2(r))$ and $\Graphs_n(r)=S(\IG_n(r))$ are quasi-free dg commutative algebras, generated by the internally connected graphs. Furthermore, the generators $\IG_n^2=\FreeOp_{\Lambda} (\IG_n^2)'$ and $\IG_n=\FreeOp_{\Lambda}  \IG_n'$ are free $\Lambda$ modules. The $\Lambda$ generators $(\IG_n^2)'$ and $\IG_n'$ are given by graphs all of whose external vertices have valency at least 1.

One can hence apply Lemma \ref{lem:cofib_criterionLa} and conclude the following.
\begin{prop}
	The dg Hopf $\Lambda$ cooperads $\Graphs_n$ and $\Graphs_n^2$ are cofibrant objects in $\dgZLaHOpc$.
	
	The truncations $\tau_{\sharp}\Graphs_n$ and $\tau_{\sharp}\Graphs_n^2$ are cofibrant in $\dgLaHOpc$.
\end{prop}

Furthermore, the action of $\QQ\ltimes \GC_n^2$ of Section~\ref{sec:graph action} is obviously compatible with the $\Lambda$-structure.

\subsection{On the space of homotopy automorphisms of $E_n$-operads in the $\Lambda$ operad setting}
Finally one can check that Theorem \ref{thm:main_intro} still holds in the $\Lambda$ operad setting.

\begin{thm}\label{thm:main_La}
	For $n\geq 2$ there are weak equivalences of simplicial monoids
	\begin{equation*}
	\Aut^h_{\sLaOp}(\DOp_n^{\QQ}) \sim
	\Aut^h_{\dgLaHOpc}(\Omega_{\sharp}^*(\DOp_n))
	\sim
	\QQ^\times \ltimes Z_{\bullet}(\GC_n^2)\, .
	\end{equation*}
\end{thm}

For the proof, one follows the proof of Theorem~\ref{thm:main_intro} of Section~\ref{sec:aut thm proof}.
Most steps therein are formal and just go through without changes.
The only place that contains a difference is for the first statement of Proposition \ref{prop:main_computation}:
one now works with a different $L_{\infty}$-algebra on the right-hand side, namely the Hopf $\Lambda$ cooperadic deformation complex
$\Def_{\Lambda}(\Graphs_n^2,W^c(\eop_n^c))$ instead of $\Def(\Graphs_n^2,W^c(\eop_n^c))$ as before.
However, we can naturally see
\begin{equation}\label{equ:DefLatoDefGra}
\Def_{\Lambda}(\Graphs_n^2\rightarrow W^c(\eop_n^c))\subset \Def(\Graphs_n^2\rightarrow W^c(\eop_n^c))
\end{equation}
as a subcomplex.
Following through with the proof of Proposition \ref{prop:main_computation}, we arrive at
\begin{equation*}
(\ICG_n^2)^{\Lambda} \hat{\otimes}_{\Sigma} \eop_n^c\{n\}
\end{equation*}
instead of $\ICG_n^2 \hat{\otimes}_{\Sigma}\eop_n^c\{n\}$ as in \eqref{equ:endefICG}, with $(\ICG_n^2)^{\Lambda}:=((\IG_n^2)')^*$.
This is hence just the normalized subcomplex of an $E_n$-Hochschild complex, which has identical cohomology as the full non-normalized complex.
From that point on, the proof of Proposition \ref{prop:main_computation} goes through identically, and we finally arrive at \ref{thm:main_La}.

\begin{appendix}

\section{The model category of $\ZZ$-graded dg (Hopf) cooperads}\label{sec:modelcat}
The main purpose of this appendix is to extend results of~\cite[Section 1]{ExtendedRHT},
where model structures on the categories of dg cooperads $\dgOpc$
and of dg Hopf cooperads $\dgHOpc$
are constructed.
We check that the same constructions can be performed for the category of $\ZZ$-graded dg cooperads $\dgZOpc$
and for the category of $\ZZ$-graded dg Hopf cooperads $\dgZHOpc$.
We devote the first subsection of the appendix to this subject.
We study truncation functors between $\ZZ$-graded dg (Hopf) cooperads and non-negatively graded dg (Hopf) cooperads in a second subsection.
To complete our study, we make explicit a cofibrancy criterion for dg Hopf cooperads. We devote the third subsection of the appendix to this topic.

\subsection{Model structure}
We define the model structure on $\dgZOpc$ as follows:
\begin{enumerate}
\item The class of weak equivalences consists of the arity-wise quasi-isomorphisms.
\item The class of cofibrations is the class of arity-wise injective morphisms.
\item The fibrations are the morphisms that have the right lifting property with respect to the class acyclic cofibrations.
\end{enumerate}
For comparison, let us mention that the definition of the model category structure on $\dgOpc$ is the same,
except that the cofibrations are only required to be injective in positive degrees (see~\cite{ExtendedRHT}).

We check that:

\begin{prop}
The above definitions provide the category $\dgZOpc$ with a well-defined model category structure.
This model structure is cofibrantly generated.
The generating cofibrations can be taken to be cofibrations between objects of overall finite dimension,
and the generating acyclic cofibrations to be acyclic cofibrations between overall countably dimensional objects that are concentrated in bounded arities.
\end{prop}

\begin{proof}
The proof is the same as that of \cite[Theorem 1.4]{ExtendedRHT}, for which the fact that complexes are concentrated in non-negative degrees plays no role.
\end{proof}

Recall that our cooperads are assumed to be coaugmented. We now consider the category of $\ZZ$-graded dg Hopf cooperads $\dgZHOpc$.
There is an adjunction (as in \cite{ExtendedRHT})
\begin{equation}\label{equ:S omega adjunction}
\Sym: \dgZOpc\rightleftarrows\dgZHOpc :\omega
\end{equation}
between the forgetful functor $\omega$ and a version of the free symmetric algebra functor $\Sym$.
For a dg cooperad $\COp\in\dgZOpc$, we explicitly have $\Sym(\COp)(r) = S(\COp(r))$ for $r\geq 2$, where $S(-)$ denotes the symmetric algebra on the category of cochain complexes,
while the object $\Sym(\COp)(1) = S(\COp(1))\otimes_{S(\QQ)}\QQ$
is defined by identifying the coaugmentation of the cochain complex $\COp(1)$
with the algebra unit of the symmetric algebra $S(-)$.
We use this adjunction to transport the model structure on the category of dg cooperads $\dgZOpc$ to the category of dg Hopf cooperads $\dgZHOpc$.
Hence, the model structure of the category $\dgZHOpc$
is defined as follows:
\begin{enumerate}
\item The weak equivalences are the arity-wise quasi-isomorphisms.
\item The fibrations are the morphisms $f: \AOp\rightarrow\BOp$ such that $\omega(f)$ is a fibration in $\dgZOpc$.
\item The cofibrations are the morphisms that have the left-lifting property with respect to all acyclic fibrations.
\item The model structure is cofibrantly generated, with the generating (acyclic) cofibrations defined by the set of morphisms of the form $\Sym(i): \Sym(\COp)\rightarrow\Sym(\DOp)$,
where $i: \COp\rightarrow\DOp$ is a generating (acyclic) cofibration in $\dgZOpc$.
\end{enumerate}

We again check that:

\begin{prop}
The above definitions provide the category $\dgZHOpc$ with a well-defined cofibrantly generated model category structure.
\end{prop}

\begin{proof}
The proof of this claim is the same as in~\cite[Theorem II.9.3.9]{OperadHomotopyBook}.
\end{proof}

\subsection{Truncation functors}
We need to compare our model categories of dg (Hopf) cooperads in the $\ZZ$-graded and in the non-negatively graded setting.
We use the adjunction:
\begin{equation*}
\tau: \dgZOpc\rightleftarrows\dgOpc :\iota,
\end{equation*}
where $\iota$ is given by the obvious inclusion of categories.
The truncation functor $\tau$, which gives the left adjoint of this inclusion functor, is defined on the category of cochain complexes by:
\begin{equation*}
(\tau V)^k = \begin{cases} V^0/\delta(V^{-1}), & \text{for k=0}, \\
V^k, & \text{for $k>0$},
\end{cases}
\end{equation*}
for any $V\in\dgZVect$. This truncation functor on cochain complexes is lax symmetric comonoidal and, therefore, preserve cooperad structures.
The adjunction relation between the categories of cochain complexes, which underlies our adjunction relation between our categories of dg cooperads, is a Quillen adjunction too.
This observation implies that the above truncation functor $\tau$ carries the (acyclic) cofibrations of $\ZZ$-graded cochain complexes
to (acyclic) cofibrations of cochain complexes,
and since the (acyclic) cofibrations of dg cooperads are nothing but morphisms dg cooperads
that define (acyclic) cofibrations of cochain complexes arity-wise,
we obtain that our truncation adjunction defines a Quillen adjunction between our categories of dg cooperads.

We have an analogous adjunction relation for dg Hopf cooperads:
\begin{equation*}
\tau_{\sharp}: \dgZHOpc\rightleftarrows\dgHOpc: \iota
\end{equation*}
where $\iota$ is again given by the obvious inclusion of categories.
The truncation functor of this adjunction relation $\tau_{\sharp}$ is still defined by the arity-wise application
of a truncation functor on the category dg commutative algebras.
For $A\in\dgZCom$, we explicitly have:
\begin{equation*}
\tau_{\sharp} A = A/(A^*,\delta(A^*),*<0),
\end{equation*}
where we consider the quotient of the dg algebra $A$ by the dg ideal generated by the components $A^*$ of degree $*<0$
of our object.
This functor on dg commutative algebras is still lax comonoidal, and therefore, does induce a functor from the category of $\ZZ$-graded dg Hopf cooperads
to the category of non-negatively graded dg Hopf cooperads.
We again have a Quillen adjunction statement:

\begin{prop}
The adjunction $\tau_{\sharp}: \dgZHOpc\rightleftarrows\dgHOpc: \iota$ is Quillen.
\end{prop}

\begin{proof}
For a symmetric algebra $\AOp = \Sym(\COp)$, we have an obvious identity $\tau_{\sharp}\Sym(\COp) = \Sym(\tau\COp)$, where we consider the image of the dg cooperad $\COp$
under the truncation functor $\tau: \dgZOpc\rightarrow\dgOpc$.
We deduce from this observation that the truncation functor $\tau_{\sharp}$ carries the generating (acyclic) cofibrations of the category of $\ZZ$-graded dg Hopf cooperads
to (acyclic) cofibrations in the category of non-negatively graded dg Hopf cooperads. The proposition follows.
\end{proof}

The adjunction augmentation $\epsilon: \tau_{\sharp}\iota\AOp\rightarrow\AOp$ is obviously the identity morphism, and hence, defines a weak equivalence.
For the adjunction unit, we have the following statement:

\begin{prop}\label{prop:tru_unit}
If $A\in\dgZCom$ is a cofibrant dg commutative algebra, equipped with an augmentation over the ground field, and such that $H^k(A) = 0$ for $k<0$ and $H^0(A) = \kk$,
then the adjunction unit
\begin{equation*}
A\rightarrow\iota\tau_{\sharp} A
\end{equation*}
is a quasi-isomorphism.
\end{prop}

\begin{proof}
By minimal model theory we can construct a cofibrant replacement of the algebra $A$
of the form:
\begin{equation*}
B = (S(V),\partial)\xrightarrow{\sim} A,
\end{equation*}
where $V$ is a graded vector space concentrated in degrees $\geq 1$,
and equipped with an exhaustive filtration
\begin{equation*}
0 = F^0 V \subset F^1 V\subset\cdots\subset F^s V\subset\cdots\subset V
\end{equation*}
such that $\partial F^s V\subset S_{\geq 2} (F^{s-1} V)$.
Then we clearly have $\tau_{\sharp} B = B$.
But $\tau_{\sharp}$ preserves weak equivalences between cofibrant objects by Quillen adjunction, and hence, the morphism $A\rightarrow\iota\tau_{\sharp} A$
fits into a commutative diagram of quasi-isomorphism:
\begin{equation*}
\xymatrix{ B\ar[r]^-{\sim}\ar[d]^-{=} & A\ar[d] \\
\iota\tau_{\sharp}B = B \ar[r]^{\sim} & \iota\tau_{\sharp}A }.
\end{equation*}
The lemma follows.
\end{proof}

This proposition also holds for the $\ZZ$-graded dg Hopf cooperads which satisfy the assumption of the proposition arity-wise
since the functors $\tau_\sharp$ and $\iota$ on dg Hopf cooperads are obtained by an arity-wise application
of the corresponding functors on dg commutative algebras.

\subsection{Cofibrancy criterion for dg Hopf cooperads}
We use the following cofibrancy criterion in our constructions:

\begin{lemm}\label{lem:cofib_criterion}
Let $\AOp$ be a dg Hopf cooperad (either $\ZZ$-graded or non-negatively graded). We assume that $\AOp$ is a quasi-free dg commutative algebra arity-wise:
\begin{equation*}
\AOp(r) = (S(\MOp(r)),\partial),\quad r>0,
\end{equation*}
for a graded symmetric sequence $\MOp = \MOp(1),\MOp(2),\ldots$ equipped with an exhaustive filtration $0= F^0\MOp\subset F^1\MOp\subset\cdots\subset F^s\MOp\subset\cdots\subset\MOp$
such that the symmetric sequence $S(F^{s-1}\MOp) + F^s\MOp\subset S(\MOp)$
is preserved by the differential and the composition coproducts of the cooperad $\AOp$,
for every $s\geq 1$.
Then $\AOp$ defines a cofibrant object in the category of dg Hopf cooperads ($\dgHOpc$ or $\dgZHOpc$).
\end{lemm}

\begin{proof}
The assumption implies that the symmetric sequence $S(F^s\MOp)\subset S(\MOp)$ is preserved by the differential and the composition coproducts of the cooperad $\AOp$.
Hence, the Hopf dg cooperad $\AOp$ has a decomposition $\AOp = \colim_s F^s\AOp$
as a colimit of subobjects
\begin{gather*}
\Com^c = F^0\AOp\subset F^1\AOp\subset\cdots\subset F^s\AOp\subset\cdots\subset\AOp
\intertext{such that}
F^s\AOp = (S(F^s\MOp),\partial).
\end{gather*}

For each $s\geq 1$, we moreover have a pushout diagram
in the category of dg Hopf cooperads:
\begin{equation*}
\xymatrix{ \Sym(S(F^{s-1}\MOp),\partial)\ar@{^{(}->}[]!D-<0pt,2pt>;[d]\ar[r] & F^{s-1}\AOp\ar@{.>}[d] \\
\Sym(S(F^{s-1}\MOp) + F^s\MOp,\partial)\ar@{.>}[r] & F^s\AOp },
\end{equation*}
where we consider coaugmented dg cooperads such that $\COp = (S(F^{s-1}\MOp),\partial)$, $\DOp = (S(F^{s-1}\MOp) + F^s\MOp,\partial)$
(using the assumption that $S(F^{s-1}\MOp) + F^s\MOp\subset\AOp$
is also preserved by the differential and the composition coproducts of the cooperad $\AOp$),
and we take the morphism of dg Hopf cooperads obtained by applying the cooperadic symmetric algebra functor $\Sym(-)$
to the inclusion of these coaugmented dg cooperads $i: \COp\hookrightarrow\DOp$.
The upper horizontal arrow of this pushout diagram is the morphism of dg Hopf cooperads induced by the identity morphism on $S(F^{s-1}\MOp)$,
while the lower horizontal arrow is the morphism of dg Hopf cooperads induced by the obvious inclusion of sub-cooperads $S(F^{s-1}\MOp) + F^s\MOp\subset S(F^s\MOp)$
within $\AOp$.

The morphism $\Sym(i): \Sym(\COp)\hookrightarrow\Sym(\DOp)$ is a cofibration by definition of the model structure of dg Hopf cooperads
by adjunction from the model structure of dg cooperads.
Then we just use that the class of cofibrations in a model category is closed under pushouts and transfinite compositions
to get the result of the lemma.
\end{proof}

\end{appendix}

\bibliographystyle{plain}
\bibliography{MappingSpaceModel}

\end{document}